\documentclass[10pt,english,leqno]{amsart}
\usepackage[latin1]{inputenc}
\usepackage{graphicx,amsbsy,dsfont,multirow}  
\DeclareMathAlphabet\mathbfcal{OMS}{cmsy}{b}{n}

\usepackage[all,color,matrix,arrow,curve]{xypic}
\usepackage{frcursive}
\makeatletter
\def\@tocline#1#2#3#4#5#6#7{\relax
  \ifnum #1>\c@tocdepth 
  \else
    \par \addpenalty\@secpenalty\addvspace{#2}%
    \begingroup \hyphenpenalty\@M
    \@ifempty{#4}{%
      \@tempdima\csname r@tocindent\number#1\endcsname\relax
    }{%
      \@tempdima#4\relax
    }%
    \parindent\z@ \leftskip#3\relax \advance\leftskip\@tempdima\relax
    \rightskip\@pnumwidth plus4em \parfillskip-\@pnumwidth
    #5\leavevmode\hskip-\@tempdima
      \ifcase #1
       \or\or \hskip 1em \or \hskip 2em \else \hskip 3em \fi%
      #6\nobreak\relax
    \dotfill\hbox to\@pnumwidth{\@tocpagenum{#7}}\par
    \nobreak
    \endgroup
  \fi}
\makeatother

\usepackage{hyperref, endnotes}

\usepackage{mathcmd}
\usepackage{amssymb,amsmath, amscd,upgreek, amsthm}
\usepackage{mathrsfs}
\usepackage{euscript}
\usepackage[usenames,dvipsnames]{xcolor}
\numberwithin{equation}{section}
\theoremstyle{plain}{}
\newtheorem{question}{Question}
\newtheorem{theorem}{Theorem}[section]

\newtheorem{coro}[theorem]{Corollary}
\theoremstyle{plain}
\newtheorem{prop}[theorem]{Proposition}

\newtheorem{lemma}[theorem]{Lemma}

\theoremstyle{remark}
\newtheorem{rema}[theorem]{Remark}
\theoremstyle{plain}
\newtheorem{theointro}{Theorem}%
\newtheorem*{coro*}{Corollary}

\newtheorem{notation}[theorem]{Notation}

\def\Q{{\rm Q}}
\def\RR{\mathbf{R}}

\def\HTp{{\H_{\T^+}}}
\def\HT{{\H_{\T}}}
\def\HTpt{{\H_{\T^+, \theta}}}

\def\F{\rm F}

\def\F{{\rm F}}
\def\G{{\rm G}}
\def\I{{{\EuScript U}}}
\def\Iw{{\EuScript  B}}

\def\J{{\rm J}}

\def\KK{{\EuScript  K}}
\def\K{\KK}
\def\L{{ \rm L}}
\def\M{{\rm M}}

\def\N{{\rm N}}
\def\m{{ \mathfrak m}}
\def\mm{{\mathfrak m}}

\def\longw{\mathbf w}
\def\tw{{\widetilde {{\longw }}}}
\def\twm{{\widetilde {{\longw }_\Mf}}}

 \def\dist{{}^\Mf\Wf}

\def\O{\rm O}
\def\P{{\rm P}}
\def\root{\alpha}
\def\R{{\rm R}}
\def\SS{{ S}}
\def\T{{\rm T}}

\def\Triv{{\rm Triv}}
\def\St{{\rm St}}
\def\Ext{{\rm Ext}}
\def\sign{{\rm Sign}}

\def\V{{\rm V}}
\def\W{{\rm W}}
\def\X{{\rm X}}

\def\Z{{\mathbb Z}}

\def\Ap{\mathscr{A}}

\def\Bb{\EuScript{B}}
\def\B{{\rm B}}

\def\Cp{\mathscr{C}}

\def\Dp{\mathscr{D}}

\def\Hh{{ {\mathcal H}}}

\def\Pp{\EuScript{P}}

\def\Uu{\EuScript{U}}

\def\lp{\langle}
\def\rp{\rangle}

\def\Hom{{\rm Hom}}
\def\End{{\rm End}}

\def\Ker{{\rm Ker}}

\def\t{{\mathbf \tau}}

\def\Res{{\rm Res}}
\def\Ind{{\rm Ind}}
\def\ind{{\rm ind}}

\def\Rep{{\rm Rep}}

\def\1{{\mathbf 1}}

\def\RR{\EuScript{R}}

\def\XX{\mathbf{ X}}
\def\XXM{\mathbf{ X}_\M}
\def\Xf{{\rm  X}}
\def\XMf{{{\rm X}_{\Mf}}}
\def\UMf{{{\Uf}_{\Mf}}}
\def\IM{{{\I}_{\M}}}

\def\Mod{{\rm Mod}(\H)}
\def\Irr{{\rm Irr}(\G)}
\def\IrrH{{\rm Irr}(\H)}
\def\Irrf{{\rm Irr}(\Gf)}
\def\IrrHf{{\rm Irr}(\Hf)}
\def\IrrM{{\rm Irr}(\M)}

\def\Hf{{\rm H}}
\def\HMf{{{\rm H}_{\Mf}}}
\def\H{{\mathcal H}}
\def\HM{{{\mathcal H}_{\rm M}}}

\def\Tf{{\overline{\mathbf T} (\fq)}}
\def\U{{\rm U}}
\def\Pf{{\P}}

\def\ii{{\rm Ind}}

\def\cii{{\rm Coind}}
\def\jj
{\textbf{\textsf{J}}}
\def\rr
{\textbf{\textsf{R}}}

\def\r{{\rm Res}}
\def\ll{{\rm L}}

\def\II
{\textbf{\textrm{I}}}

\def\cII{{\rm Coind }}
\def\JJ
{\textbf{\textrm{J}}}
\def\RR
{\textbf{\textrm{R}}}

\def\R{{\rm R}}
\def\LL
{\textbf{\textrm{L}}}

\def\HM{{{\mathcal H}_\M}}
\def\HMp{{{\mathcal H}_{\M^{+}}}}
\def\HMpt{{{\mathcal H}_{\M^{+}, \theta}}}
\def\HMm{{{\mathcal H}_{\M^{-}}}}
\def\HMmt{{{\mathcal H}_{\M^{-}, \theta^*}}}

\newcommand{\Id}{\operatorname{Id}}

\def\val{\operatorname{\emph{val}}}

\def\Mod{{\rm Mod}}

\def\Wf{{\mathds W}}
\def\Mf{{\mathds M}}
\def\Pf{{\mathds P}}
\def\Bf{{\mathds  B}}
\def\Gf{{\mathds  G}}
\def\Tf{{\mathds T}}
\def\Uf{{\mathds  U}}

\def\Nf{{\mathds  N}}
\def\Zf{{\mathds Z}}
\def\Kf{{\mathds K}}
\def\nf{{\mathbf n}}

\def\k{R}

\def\Tp{{\rm T}}
\def\Zp{{\rm Z}}
\def\Gp{ { \rm G}}

\def\ssmod{\sigma}

\def\norm{{\mathscr N}}
\def\Np{{\norm_\G}}

\def\Nm{{{\rm N^{op}}}}

\def\Mp{{{\rm M}^{+}}}
\def\Mm{{{\rm M}^{-}}}

\def\WMp{{{\rm W}_\Mp}}
\def\WMm{{{\rm W}_\Mm}}

\def\triv{{\rm Triv}} 
\usepackage{comment,endnotes}

 \theoremstyle{remark}%
  \newtheorem*{exa}{Example}%

\usepackage{hyperref}
\definecolor{webblue}{rgb}{0, 0.7, 0.5}
\definecolor{webred}{rgb}{0.2, 0.3, 0.6} 
\hypersetup{colorlinks,  citecolor=webblue,  linkcolor=webred, urlcolor=black}

\title{Parabolic induction in characteristic $p$}

\author{Rachel Ollivier and Marie-France Vign\' eras}
\usepackage{csquotes}
\address{University of British Columbia,  1984 Mathematics Road, Vancouver, BC V6T 1Z2, Canada}
\email{ollivier@math.ubc.ca}
\address{Institut de Math\' ematiques de Jussieu, UMR 7586,
175 Rue du Chevaleret
75013 Paris, France}
\email{marie-france.vigneras@imj-prg.fr}
\date{\today}


\usepackage{csquotes}
\MakeOuterQuote{"}
\begin{document}

\maketitle
\begin{abstract} Let  $\F$  (resp. $\mathbb F$)  be a nonarchimedean locally compact field  with residue characteristic $p$ (resp. a finite field  with characteristic $p$). For $k=\F$ or $k=\mathbb F$,   let $\mathbf{G}$ be a connected reductive group over $k$ and $R$ be a commutative ring.  We denote by $\Rep( \mathbf G(k)) $  the category of  smooth $R$-representations of $ \mathbf G(k) $. To a  parabolic $k$-subgroup ${\mathbf P}=\mathbf{MN}$ of $\mathbf G$ corresponds the parabolic induction functor    $\Ind_{\mathbf P(k)}^{\mathbf G(k)}:\Rep( \mathbf M(k))  \to \Rep( \mathbf G(k))$.  This functor has a left and a right adjoint.

Let 
  $\I$ (resp. $\Uf$) be a pro-$p$ Iwahori (resp. a $p$-Sylow) subgroup of $ \mathbf G(k) $ compatible with ${\mathbf P}(k)$  when $k=\F$  (resp. $\mathbb F$). 
  Let  ${H_{ \mathbf G(k)}}$ denote   the pro-$p$ Iwahori (resp. unipotent) Hecke algebra of $ \mathbf G(k) $ over $R$ and  
   $\Mod({H_{ \mathbf G(k)}})$   the category of right modules over ${H_{ \mathbf G(k)}}$.
   There is  a functor  $\Ind_{{H_{ \mathbf M(k)}}}^{{H_{ \mathbf G(k)}}}: \Mod({H_{ \mathbf M(k)}}) \to \Mod({H_{ \mathbf G(k) }})$ called parabolic induction for Hecke modules; it has a left and a right adjoint. 

We prove  that the pro-$p$ Iwahori (resp. unipotent)  invariant functors
 commute with  the parabolic induction functors, namely that  
$\Ind_{\mathbf P(k)}^{\mathbf G(k)}$ and  $\Ind_{{H_{ \mathbf M(k)}}}^{{H_{ \mathbf G(k)}}}$ form a 
  commutative diagram   with    the  $\I$ and $\I \cap \mathbf M(\F)$ (resp.  $\Uf$ and $\Uf\cap \mathbf M(\mathbb F) $) invariant functors. 
  We prove   that the pro-$p$ Iwahori (resp. unipotent)  invariant functors
 also commute with  the right adjoints of the parabolic induction functors. However, they do not commute with the left adjoints of 
  the parabolic induction functors in general; they do if $p$ is invertible in $R$.

When $R$ is  an algebraically closed field  of characteristic $p$,
 we show that an irreducible admissible $R$-representation of $ \mathbf G(\F) $ is supercuspidal (or equivalently supersingular) if and only if the ${H_{ \mathbf G(\F)}}$-module $\m$ of  its $\I$-invariants admits a   supersingular   subquotient, if and only if $\m$ is  supersingular.
 \end{abstract}
\setcounter{tocdepth}{3} 

\tableofcontents

\section{Introduction} 

\medskip

1.1. Let  $\F$   be a nonarchimedean locally compact field  with residue characteristic $p$,  let $\mathbf{G}$ be a connected reductive group over $\F$.  We set $\G:=\mathbf{G}(\F)$  and call it a $p$-adic reductive group. 
Let $R$ be a commutative ring. We denote by $\Rep (\G)$  the category of smooth $R$-representations of $\G $.  
When $R$ is a field, we denote  by   $\Irr $ the set  of irreducible admissible  $R$-representations of $\G$ modulo isomorphism. For $A$  a ring we denote by $\Mod(A)$ the category of right $A$-modules.

When $R$ is the complex field $\mathbb C$,   the category $\Rep (\G)$    has been studied   because of its connection to the theory of automorphic forms. The  classification of $\Irr $  (known as the conjectures of Langlands and Arthur) is now understood  for many groups but not yet for a general $p$-adic reductive group $\G$. 

For an algebraically closed  field $R$ of characteristic $p$,  the description of $\Irr $  is   reduced to the classification of the  irreducible admissible  supercuspidal  representations (\cite{AHHV}), using  the parabolic induction  and  a process of reduction 
 which is similar  to the one established for complex representations of ${\rm GL}(n,\F)$   by Bernstein and Zelevinski (\cite{BZ}, \cite{Zel}). 
 
The study of the  smooth $R$-representations of $\G$ (called mod $\ell$ representations if    the characteristic of $R$ is a prime number $\ell$) for  a general  $R$, is motivated by   the congruences between  automorphic forms and by number theory.
The theory of Harish-Chandra to study complex representations  cannot be extended to the case  when $R$ is a field of characteristic $p$  because  $\G$ does not admit a  Haar measure with values in such a field. However, certain  tools such  as  the parabolic induction or  the Hecke algebras of  {types}     remain available for any  $R$. Note that when $R$ is a field of characteristic $p$,  there is only one natural {type}, namely the trivial $R$-representation of a pro-$p$ Iwahori subgroup $\I$ of $\G$. This is because 
any non zero smooth $R$-representation of $\G$ admits a non zero invariant vector by any pro-$p$ subgroup of $\G$. 

 \medskip

1.2.   We fix a minimal parabolic subgroup $\mathbf B$ of $\mathbf G$ with a Levi decomposition $\mathbf B=\mathbf Z  \mathbf U$  of  Levi subgroup  $\mathbf Z$  and unipotent radical  $\mathbf U$ and  a compatible   Iwahori  subgroup  of $\G$  of pro-$p$-Sylow $\I$.  
Let $R[\I\backslash \G]$ the natural smooth representation of $\G$ on the space of left $\I$-invariant functions $\G\rightarrow R$ with compact support.
Let  $\H$ be   the algebra of $R[\G]$-intertwiners of $R[\I\backslash \G]$. It is the pro-$p$ Iwahori Hecke $R$-algebra of $\G$. We denote by $\Mod(\H)$ the category of  right $\H$-modules and when $R$ is a field, by $\IrrH$ the set of simple right $\H$-modules modulo isomorphism. 
The subspace  $V^\I$ of $\I$-invariant vectors in  $V\in \Rep(\G)$  is a right $\H$-module and we consider the  functor $(-)^\I:V\mapsto V^\I=\Hom_\G(R[\I\backslash \G], V)$ from $\Rep(\G)$ to   $\Mod(\H)$. We call it the $\I$-invariant functor.  It has a left adjoint given by the tensor product functor $-\otimes_{\H}R[\I\backslash \G]$.

  When $R$ is a field, the $\I$-invariant functor induces a bijection from the set of $V\in \Irr$ with $V^\I\neq0$ onto $\IrrH$ if the characteristic of $R$ is $\ell\neq p$ (\cite[I.6.3]{Viglivre} applied to the global Hecke $R$-algebra of $\G$ which is the algebra of  compact and locally constant functions $\G\to R$ with the convolution product given by a left Haar measure on $\G$ with values in $R$) or if $\G={\rm GL}(2, \mathbb Q_p)$ (\cite{Breuil}, \cite{VigGL2}).    When $R$ is a field of characteristic $p$, the $\I$-invariant functor   does not always send a simple object in $\Rep( \G)$ onto a simple object in $\Mod(\H)$ (counter examples are built in \cite{BP} in the case of $\G={\rm GL}(2, \F)$ when $\F$ is a strict unramified extension of $\mathbb Q_p$).

\medskip
We refer to this setup as \emph{the $p$-adic case.}

\medskip
1.3. Replace $\F$  by a finite field $\mathbb F$ of characteristic $p$,   the $p$-adic reductive group $\G$ by the finite reductive group $\Gf={\mathbf G}(\mathbb F)$, the pro-$p$ Iwahori group $\I$ by the unipotent radical $\Uf$ of $\Bf=\Bf(\mathbb F)$, and the pro-$p$ Iwahori Hecke $R$-algebra by the  algebra $ \Hf$ of   all $\Gf$-intertwiners of $R[\Uf\backslash \Gf]$ (called the unipotent Hecke $R$-algebra). We consider the $\Uf$-invariant functor $(-)^\Uf:V\mapsto V^\Uf=\Hom_\Gf(R[\Uf\backslash \G], V)$  from $\Rep(\Gf)$ to   $\Mod(\Hf)$. Its left adjoint is the tensor product functor $-\otimes_{\Hf}R[\Uf\backslash \Gf]$.

Suppose that $R$ is a field and denote by $\Irrf$ the set of irreducible $R$-representations of $\Gf$ modulo isomorphism. If  $R$ has characteristic different from $p$,  the $\Uf$-invariant functor $(-)^\Uf$ induces a bijection between the set of $V\in \Irrf$ with $V^\Uf\neq0$ and the set $\IrrHf$  (\cite[I.6.3]{Viglivre} applied to the group algebra $R[\Gf]$). When $R$ is a field of characteristic  $p$,    under certain mild hypotheses on $R$ (see for example\cite[Thms. 6.10 and 6.12]{CE}), it yields a bijection between $\Irrf$  and $\IrrHf$.

\medskip
We refer to this setup as \emph{the finite case.}

\medskip

1.4.  We  recall the definition of the functor  of parabolic induction in the $p$-adic case. Consider  a parabolic subgroup $\mathbf P$ of $\mathbf G$ containing $\mathbf B$ with Levi decomposition $\mathbf P=\mathbf M\mathbf N$ where $\mathbf Z\subset \mathbf M$. We set  $\P:=\mathbf P(\F), \M:=\mathbf M(\F)$ and $ \N:=\mathbf N(\F)$.
The parabolic induction  functor $\Ind_{\P}^{\G}:\Rep(\M)\to \Rep(\G)$ sends $W\in \Rep(\M)$  to the representation of $\G$ by right translation on the $R$-module $\{f:\G\to W \ | f(mnxk)=mf(x)\: \forall\: (m,n, x,k)\in ( \M,\N,\G,K_f)\}$  where $K_f$ is some open subgroup of $\G$ depending on $f$.  
When $R$ is a field, the parabolic induction respects admissibility because $\P\backslash \G$ is compact; an admissible  irreducible $R$-representation of $\G$ is called supercuspidal when,  for any $\P\neq \G$ and any admissible  $W\in \IrrM$,  it is not  a subquotient of   $\Ind_{\P}^{\G}W$.

The functor  $\Ind_{\P}^{\G} $ is faithful. It has a left adjoint given by the  $\N$-coinvariant  functor $\LL_{\P}^{\G}=(-)_{\N}$. It also has a right adjoint denoted by $\RR_\P^\G$  \cite{Vigadjoint}. The functor $\RR_\P^\G$ has been computed by Casselman and Bernstein when $R=\mathbb C$ and by Dat when $p$ is invertible in $R$ for many groups $\G$ but  not yet for a general $p$-adic reductive group $\G$ (\cite{Dat}). 
When $R$ is a field of characteristic $ p$, the parabolic induction $\Ind_{\P}^{\G} $  is fully faithful and the right adjoint $\RR_\P^\G$  is  equal to Emerton's functor  on admissible representations  (\cite{Vigadjoint}) because $\RR_\P^\G$ respects admissibility (\cite{AHV}).

Before stating our first result, we recall that $\I_{\M}:=\I\cap \M$ is a pro-$p$ Iwahori subgroup of $\M$ compatible with ${\mathbf B}_{\mathbf M}= \mathbf B\cap \mathbf M$.  There exists a functor  $\Ind_{\H_\M}^\H:\Mod_{\H_\M}\to \Mod_{\H }$ defined in  \cite{Vig5} (in the current article, we refer to \S\ref{indcoind-prop} for precise statements and references).
We call  it parabolic induction for Hecke modules. It has a left adjoint  $\L_{\H_\M}^\H$ and a right adjoint  $\R_{\H_\M}^\H$.   
\begin{theointro}\label{thm:padic} The functor $\Ind_{\H_\M}^\H:\Mod_{\H_\M}\to \Mod_{\H }$ and its adjoints satisfy:
$$(- )^\I \circ \Ind_\P^\G    \cong \Ind_{\H_\M}^\H\circ (- )^{\I_\M}, \quad (-)^{\I_\M} \circ {\bf R}_{ \P}^\G\cong \R_{\H_\M}^\H\circ (- )^{\I}.$$
When $p$ is invertible in $R$, we also have $ (-)^{\I_\M} \circ {\bf L}_{ \P}^\G\cong {\L}_{\H_\M}^\H\circ (- )^{\I}$. But in general  there is no functor $F:\Mod_{\H_\M}\to \Mod_{\H }$ such that $ (-)^{\I_\M} \circ {\bf L}_{ \P}^\G\cong  F\circ (- )^{\I}$.   \end{theointro}
The theorem is proved in this
article by answering Questions \ref{qp1} and \ref{qp2} and \ref{qp3} in \S\ref{answersprop}.


\bigskip

 In the finite case, we let   $\Pf:=\mathbf P(\mathbb F), \Mf:=\mathbf M(\mathbb F), \Nf:=\mathbf N(\mathbb F)$ and 
 recall the definition of the parabolic induction  functor $\Ind_{\Pf}^{\Gf}:\Rep(\Mf)\to \Rep(\Gf)$. It sends $W\in \Rep(\Mf)$  to the representation of $\Gf$ by right translation on the $R$-module $\{f:\Gf\to W \ | f(mnx)=mf(x)\:\: \forall\: (m,n, x )\in ( \Mf,\Nf,\Gf)\}$.
 When $R$ is a field, a  representation $V\in \Irrf$ is finite dimensional   and  supercuspidality  is defined as in the $p$-adic case.
 The functor $\Ind_{\Pf}^{\Gf}$  is faithful  with   left adjoint  the  $\Nf$-coinvariant  functor $\LL_\Pf^\Gf:=(-)_{\Nf}$ and  right adjoint the  $\Nf$-invariant  functor $\RR_\Pf^\Gf:(-)^{\Nf}$ (see \S\ref{parabfinitegroup}).
The parabolic induction for Hecke modules  is defined to be  the functor  $\Ind_{\HMf}^\Hf= -\otimes _\HMf \Hf:\Mod_{\HMf}\to \Mod_{\Hf}$. Its right adjoint ${\rm R}_{\HMf}^\Hf$ is the natural restriction functor 
$\Mod(\Hf)\to \Mod(\HMf)$.  Its left adjoint is computed in Prop. \ref{prop:compafunc} (see also Remark \ref{L=res}).
 We prove an analog of Theorem \ref{thm:padic}:
 
 \begin{theointro}\label{thm:finite} The  functor 
 $\Ind_{\HMf}^\Hf:\Mod_{\HMf}\to \Mod_{\Hf}$ and its adjoints satisfy: $$ (- )^\Uf\circ \Ind_{\Pf }^{\Gf }\cong \Ind_{\HMf}^\Hf\circ (- )^{\Uf_\Mf}, \quad (-)^{\Uf_\Mf}\circ {\bf R}_{\Pf }^{\Gf }\cong {\rm R}_{\HMf}^\Hf\circ (- )^{\Uf}.$$
When $p$ is invertible in $R$, we also have  $ (-)^{\Uf_\Mf}\circ{\bf L}_{\Pf }^{\Gf }\cong {\rm L}_{\HMf}^\Hf\circ (- )^{\Uf}$. But in general  the functors $ (-)^{\Uf_\Mf}\circ {\bf L}_{\Pf }^{\Gf }$ and $ \L_{\HMf}^\Hf\circ (- )^{\Uf}$ are not isomorphic. \end{theointro}

The  theorem is proved by answering Questions \ref{qf1}, \ref{qf2} and \ref{qf3} in \S\ref{answers}.


\medskip 

1.5.  A table comparing 
 the parabolic induction functors 
  $$\Ind_\HM^\H: \Mod(\HM)\rightarrow  \Mod(\H)\textrm{ and }
\Ind_\HMf^\Hf: \Mod(\HMf)\rightarrow  \Mod(\Hf)$$
 and  their respective  left and right adjoints $\L_\HM^\H$, $\R_\HM^\H$, $\L_\HMf^\Hf$ and $\R_\HMf^\Hf$
(when $R$ is an arbitrary ring and when  $p$ is invertible in $R$)
is provided in Section \ref{table}.

 \medskip

1.6.  
 In the $p$-adic case,  let $\mathbf T$ be the maximal split central subtorus of  $\mathbf Z$. Let  $\T:=\mathbf T(\F)$ and denote by   ${\rm X}^+_*(\T)$ the monoid of dominant cocharacters of $\T$ with respect to $\B$. The algebra $\H$ is a finitely generated module over its center $\mathcal Z$ and $\mathcal Z$  is a finitely generated module over a subalgebra $\mathcal Z_\T$ isomorphic to $R[\X_*^+(\T)]$ (\cite{Compa} when $\bf G$ is $\F$-split; generalized by \cite{Vig2}). 

We suppose that  $R$ is an algebraically closed field of characteristic $p$. The supercuspidality of  $V\in \Irr$  can be seen on the action of    $\mathcal Z$  on $V^\I$.    A $\mathcal Z$-module $\mathcal M$  is called supersingular if for all  $v\in \mathcal M$ and all non invertible $x\in \X_*^+(\T)$, corresponding to  $z_x\in \mathcal Z_\T$, there exists a positive integer $n$ such that $z_x^nv=0$. A right $\H$-module is called supersingular when it is supersingular as a $\mathcal Z$-module.  Abe described $\IrrH $  via parabolic induction from supersingular simple modules (\cite{Abeclassi}).
When $\G=\Zp $, all finite dimensional $\H$-modules are supersingular and all irreducible admissible $R$-representation of $\G$ are supercuspidal.   We establish the following:

\begin{theorem}\label{thm:ss}
Suppose that $R$ is an algebraically closed field of characteristic $p$. Let $V $  be an irreducible admissible $R$-representation of $\G$. Then the following are equivalent:
\begin{enumerate}
\item
 $V$ is supercuspidal.
 
 \item   the finite dimensional $\H$-module $V^{\I}$ is supersingular.
 \item   the finite dimensional $\H$-module $V^\I$ admits a supersingular subquotient.
 \end{enumerate}
 \end{theorem}

We choose a special parahoric subgroup $\K$ of pro-$p$ Sylow $\I$.  For an  irreducible smooth representation $W$ of  $\K$, the center $\mathcal Z(W)$ of the algebra  of  intertwiners of the compactly induced representation $\ind_\K^\G(W)$ and  the center of $\H$ have   a similar structure, and there is a similar  notion of a supersingular $\mathcal Z(W)$-module.  A representation $V\in \Irr$ is supersingular when the $\mathcal Z(W)$-module $\Hom_\K(W,V)$ is non zero and supersingular for some   $W$.  A corollary of the description of $  \Irr $  via parabolic induction from supersingular irreducible admissible representations \cite{AHHV}, is that  for $V\in \Irr$, supercuspidality is equivalent to supersingularity.  If $V^\I$ is supersingular, then $V$ is supersingular (see Remark \ref{21}). The   converse  (1) $\Rightarrow$(2) follows from the commutativity of Diagram \eqref{diag2'prop} (which is contained in  Theorem \ref{thm:padic}
by passing to left adjoints in the identity involving 
 $\R^\G_\P$ and  $(-)^\I$)
and  from the computation of $\mathcal M \otimes_{\H} R[\I \backslash \G]$ for  a simple $\H$-module $\mathcal M$ done  in  \cite{AHV}. Lastly, the argument to prove (3) $\Rightarrow$ (1)  
follows from the computation of $V^\I$ done in \cite{AHV}.   Our first proof  did not use \cite{AHV} and started with a reduction to the case where $\G$ is almost simple, simply connected and isotropic. But  this allowed us to prove 
$(3)\Rightarrow (1)$   only when   the index  $[\G:\G^{der}C]$ is finite, where $\mathbf G^{der}$ and $ \mathbf C$  are the derived group and the center of $\mathbf G$, $\G^{der}=\mathbf G^{der}(\F)$ and  $C =\mathbf G(\F)$.

 \bigskip {\bf Acknowledgements} We thank Noriyuki Abe for suggesting the counter example of Prop. \ref{cex}  and for generously  sharing his  recent results with us. We are also thankful to Guy  Henniart for his continuous interest and helpful remarks.  Our work was carried out at the  Institut de Mathematiques de Jussieu -- Paris 7,  the University of British Columbia and the Mathematical Sciences Research Institute. We would like to acknowledge the support of these institutions. The first author
 is partially funded by NSERC Discovery Grant.

\section{Notation and Preliminaries\label{nots}} 


Let  $\F$   be a nonarchimedean locally compact field  with residue characteristic $p$.    
 The  ring of integers of $\F$ is denoted by $\O_\F$. The residue field  of $\F$ is a finite field  $\mathbb F_q$ with $q$  elements, where $q$ is a power of  the prime number $p$.  
 We choose a uniformizer
$\varpi$ and denote by 
  $\val_{\F}$  the valuation  on   $\F$   normalized by 
  $\val_{\F}(\varpi)=1$.

We introduce the following objects and notation, keeping in mind that  throughout the article, the group of $\F$-points of an algebraic $\F$-group $\mathbf Y$ will be denoted by $\rm Y$.

 \begin{itemize}

\item[-]  $\mathbf{G}$  a connected reductive group over $\F$; 
\item[-]  $\mathbf{T}$  a maximal $\F$-split torus of $\mathbf{G}$, of $\mathbf G$-centralizer $\mathbf Z$ and  $\mathbf G$-normalizer   $\mathbf N_{\mathbf G}$; 
\item[-]  $\Phi $ the relative root system of $\mathbf{T}$ in $\mathbf{G}$. To $\alpha\in \Phi$ is associated the root subgroup  $\mathbf U_\alpha$. We have  $\mathbf U_\alpha \subset \mathbf U_{2\alpha}$ if $\alpha, 2\alpha \in \Phi$.
\item[-] $\mathbf B$ is a minimal parabolic $\F$-subgroup  of $\mathbf G$ with Levi decomposition $\mathbf B= \mathbf Z \mathbf U$;  we have $\B=\Zp \U$;  
\item[-] $\varphi= (\varphi_{\alpha})_{\alpha \in \Phi}$  a  valuation  of the root datum $(\Zp, (\U_{\alpha})_{\alpha \in \Phi})$  of $\Gp$ of type $\Phi$ generating $\Gp$  \cite[Def. (6.2.1)]{BTI} which is compatible with the valuation $\val_\F$  \cite[(5.1.22) formula  (2)]{BTII}, discrete
\cite[Def. (6.2.21)]{BTI} and special \cite[(6.2.13) Def.]{BTI}.
\end{itemize}

To $(\mathbf T, \mathbf B, \varphi)$ is associated a reduced root system $\Sigma$ with a basis $\Delta$, an apartment   $\Ap$  in the  semisimple  building  of $\Gp$,  a special vertex $x_0\in \Ap$    \cite[1.9]{Tit},  and an alcove $\Cp\subset \Ap$ of vertex $x_0$. 
 The  fixer of $x_0$, resp. $\Cp$, in the kernel of the Kottwitz homomorphism $\kappa_G$ \cite[7.1 to 7.4]{K}, is a special maximal parahoric group $\K$   of $\Gp$ of  pro-unipotent radical $\K^1$, resp. an Iwahori subgroup $\Iw$ of unique pro-$p$ Sylow subgroup $\I$.

\subsection{\label{products}  Root system and Weyl groups attached to $\G$}

\subsubsection{Relative root system}
The  relative  root system $\Phi $ of $\Tp$ in $\Gp$ is related to $\Sigma$ by a   map $e:\Phi\to \Nf_{>0}$   such that $\Sigma = \{e(\beta) \beta \ | \beta \in \Phi\}$. For a  root $\beta\in \Phi $ such that  $ \beta/2$ does not belong to $\Phi$, the image of the homomorphism $\U _\beta-\{1\}\xrightarrow{\varphi_\beta} \mathbb Q$ given by the  valuation $\varphi$ is $e(\beta)^{-1} \mathbb Z$  \cite[\S 3.5 (39)]{Vig1}. When $\bf G$ is $\F$-split, $\Phi=\Sigma$ is a reduced root system and $e(\beta)=1$ for all $\beta \in \Phi$. 
 The map $\beta \mapsto e(\beta) \beta$ induces a bijection from the set $\Phi_{red}$  of reduced roots of $\Phi$  onto $\Sigma$, and from $\Delta $ onto a basis $\Pi$ of  $\Phi$.

 To $\alpha \in \Sigma$  we associate the group  
 $\U _\alpha:=\U _\beta$   where $\beta$ is the unique reduced root of $\Phi$ such that $\alpha= e(\beta)\beta $. To   an affine root $(\alpha,r)\in \Sigma_{aff}:= \Sigma \times \mathbb Z$    corresponds a compact open subgroup 
$$\Uu _{(\alpha, r)}=\Uu _{(\beta,  e(\beta)^{-1}r)}{ =\{1\} \cup  \{ x\in \U _\beta - \{1\} \ | \  \varphi_\beta (x) \geq e(\beta)^{-1}r\}}.$$
We identify $\Sigma$ with the subset $\Sigma\times \{0\}$ of  $\Sigma_{aff}$.   In particular, we write $\Uu_{\alpha}$ instead of $\Uu_{\alpha,0} $.

\subsubsection{Iwahori decompositions}We have the inclusions $\K^1\subset \I \subset \Iw \subset \K$. These groups are generated by their intersections  with, respectively,  $\Zp$, $\U$, and $ \U^{op} : =\mathbf U^{op}(\F)$ where $\mathbf B^{op}=\mathbf Z \mathbf U^{op}$ is the opposite parabolic subgroup of  $\mathbf B$. We describe these intersections:\begin{itemize}
\item[-]  $\Zp\cap \K^1=\Zp\cap \I = \Zp^1$ is the unique pro-$p$ Sylow subgroup of the  unique parahoric sugroup 
 $\Zp^0:=\Ker\, \kappa_\Zp$ of $\Zp$. We have  $\Zp\cap \K=\Zp\cap \Iw =\Zp^0 $. The intersection   $ \Tp^0:= \Tp\cap \Zp^0$, resp.   $\Tp^1:=\Tp\cap\Zp^1$, is the maximal compact subgroup, resp. the 
  pro-$p$ Sylow subgroup of $\Tp^0$. The group $\Zp^0$ is an open  subgroup of the maximal compact subgroup of $\Zp$. 
 \item[-] $\U \cap \K^1$ (resp.  $ \Uu_0:= \U \cap \I=\U \cap \Iw= \U \cap \K$) is the image of $ \prod_{\alpha\in \Sigma^+} \Uu_{(\alpha, 1)}$ (resp. $ \prod_{\alpha\in \Sigma^+} \Uu_{\alpha }$) by the multiplication map.
   \item[-] $\U^{op}\cap \K^1=\U^{op}\cap \I=\U^{op}\cap \Iw$ (resp. $ \Uu_0^{op}:=\U^{op}\cap \K$) is the image of $ \prod_{\alpha\in \Sigma^-} \Uu_{(\alpha, 1)}$ (resp. $\prod_{\alpha\in \Sigma^-} \Uu_{ \alpha }$) by the multiplication map.
\end{itemize}
The products above are ordered in some arbitrary chosen way, $\Sigma^+\subset \Sigma$ is the subset of positive roots and $\Sigma^-=-\Sigma^+$.  The groups  $\K^1\subset \I \subset \Iw$ (but not $\K$) have an Iwahori decomposition: they are the product (in any order) of their intersections with $ \U^{op},\Zp,  \U$.  For the pro-$p$ Iwahori subgroup $\I \subset \G$ this means that the product map 
\begin{equation}\label{rootsubgroupeq}
    \prod_{\alpha\in \Sigma^-} \Uu_{(\alpha, 1)}\times  \Zp^1 \times  \prod_ {\alpha\in \Sigma^+} \Uu_{\alpha}\overset{\sim}\longrightarrow \I
\end{equation} is a bijection.

 The subgroups  $\Zp^0$ and $\Zp^1$ of $\Zp$   are normalized by  $\Np:=\mathbf N_{\mathbf G}(\F)$. The quotients $\W_0:=\Np/\Zp,$ $\W:= \Np/\Zp^0$,  and   $\W(1):=\Np/\Zp^1$ of  $\Np$ by $\Zp$,  $\Zp^0$ and   $\Zp^1$ respectively, are the finite, Iwahori, pro-$p$ Iwahori Weyl groups. 

The finite Weyl group $\W_0$, the Weyl group of the root system  $\Phi$  and  the Weyl group of the reduced root system $\Sigma$ are canonically isomorphic. The subgroup $\Lambda:=\Zp/\Zp^0$ of $\W$ is commutative and finitely generated,  the subgroup $\Lambda(1)=\Zp/\Zp^1$ of $\W(1)$ is  finitely generated but not always commutative.  The subgroup  $\Lambda_\T:=\T/\T^0$ of  $\Lambda $ has finite index and identifies with the group  ${\rm X}_*(\mathbf T)$ of algebraic  cocharacters 
of $\mathbf T$.

\begin{notation}\label{not'}
Let $\Gp' $ (resp. $\Gp'_{\alpha}$ for $\alpha \in \Phi$) denote the subgroup of $\G$ generated by $\U, \U^{op}$ (resp. $\U_{\alpha}, \U_{-\alpha}$). It  is normal  in $\G$ and $\Gp=\Zp \Gp'$. The group $\G_{aff}= \Zp^0\G'$  is  generated by all the parahoric subgroups of $G$. For any subset $X\subset \Gp$ we put $X':=X\cap \Gp'$. 
\end{notation}

The quotient $\W':=\Np'/(\Zp^0)'{=(\norm_\Gp\cap \G_{aff})/\Zp^0}  \subset \W$ is called the affine Iwahori Weyl group (Notation \ref{not'}). The quotient  ${}_1\W' :=\Np'/(\Zp^1)' \subset \W(1)$  is the pro-$p$-affine Iwahori Weyl group. It is  contained in  the inverse image $\W' (1){=(\Np\cap \G_{aff})/\Zp^1} $ of $\W'$ in $\W(1)$, and is usually different.

 \begin{notation}For a subset $X\subset \W$ we denote by $X(1)$  its preimage   in  $\W(1)$. \label{not(1)}
\end{notation}

\subsubsection{Lifts of the elements in the (pro-$p$) Iwahori Weyl group}
 For $w\in \W$,   we will sometimes introduce a lift $\tilde w\in \W(1)$ for $w$. 
 For $w$ in $\W$ or $\W(1)$, we will  sometimes introduce a lift $\hat w\in \Np$ for $w$.
 Given $w\in \W$, the notation $\hat {\tilde w}$ will  correspond to the choice of an element in $\norm_\G$ 
 with projection $\tilde w$ in $\W(1)$ and $w$ in $\W$. We represent this choice by the following diagram
\begin{equation}\label{picklift}\begin{array}{ccccc}  \Np&\rightarrow& \W(1)&\rightarrow &\W\cr \hat{\tilde w}&\mapsto & \tilde w&\mapsto &w\cr\end{array}\end{equation}

\subsubsection{\label{bruhatdec}Bruhat decompositions of $\G$ and $\G'$}

 Throughout the article, we will sometimes omit the notation $\hat w$, $\hat {\tilde w}$ or $\tilde w$  and we will just use $w$   in the formulas that do not depend on the choice of the lift. This is the case in the following statements.  As a  set, $\Gp$  is the disjoint union  of the
double cosets (see \cite[1.2.7, 4.2.2 (iiii)]{BTI},   \cite[Prop. 3.34,  Prop. 3.35]{Vig1}):
\begin{equation}\label{Bruhataffine}\Gp=\sqcup _{w\in\W}\Iw  w \Iw= \sqcup _{w\in \W(1)}\I  w \I,
\end{equation}
and  the group $\Gp'$ is the disjoint union of double cosets 
\begin{equation}\label{Gp'}\Gp'=\sqcup _{w\in\W'}\Iw'  w \Iw'= \sqcup _{w\in{}_1\W'}\I'  w \I'
\end{equation} where again we do not specify the choice of the lifts in $\norm_\G'$ of  $w$ in $\W'$ (resp.  $w\in {}_1\W'$) since the double cosets above do not depend on this choice. The equalities  \eqref{Gp'}
follow easily from the analogous equalities for $\G_{aff}= \Zp^0\G'$  proved in \cite[Prop. 3.34]{Vig1}.

\subsubsection{\label{affine+decomp}Affine roots and decompositions of the Iwahori-Weyl group $\W$} Choosing $x_0$ as its origin, the apartment  $\Ap$  identifies with the real vector space
 $\V:=  {\mathbb R}\otimes _{\mathbb Z}(\X_*(\mathbf T)/\X_*(\mathbf C)) $ where 
 ${\rm X}_*(\mathbf C)$ is the group of algebraic  cocharacters 
of  the connected center $\mathbf C$  of $\mathbf G$. 
We fix a $\W_0$-invariant scalar product $\langle\,.\,,.\rangle$ of $\V$.  Let $\mathfrak H$ denote the set of the affine hyperplanes $ \langle\,.\,,\alpha\rangle =-r$ associated to the affine roots $(\alpha,r)\in \Sigma_{aff}$.  The alcove $\Cp$ identifies with the alcove of $(\V,\mathfrak H)$ of vertex $0$ contained in the dominant Weyl chamber $\Dp$ determined by $\Delta$. 

 For any affine root $A=(\alpha,r)\in \Sigma_{aff}$, we have the orthogonal reflection $s_A $ at the associated affine hyperplane and the  subgroup $\Gp'_{s_A}\subset \Gp'$ generated by $\Uu _{A}$ and $\Uu _{-A}$. The Weyl group  $\W_0$ of $\Sigma$ is  generated by the set $S:= \{s_\alpha \ | \ \alpha\in \Delta\}$.  
 
  There is  a partial order on $\Sigma$ given by $\root\preceq \beta$ if and only if $\beta -\root$ is a linear combination with integral nonnegative coefficients of simple roots. 
{Let  $\Sigma_m$  be the set of   roots in $\Sigma $ that are minimal elements  for $\preceq$. The walls of the alcove $\Cp$ are the affine hyperplanes associated to the set   $\Delta_{aff}:= \Delta \cup\{(\alpha,1),\, \alpha\in\Sigma _m\}$ of   simple affine roots.}  The affine Weyl group $\W_{aff}$ of $\Sigma$ is generated by the set  $S_{aff} := \{ s_A , \:\: A \in \Delta_{aff}\}$. {The set of $\W_{aff}$-conjugates of  $S_{aff} $ is the set
$\mathfrak S := \{ s_A , \:\: A \in \Sigma_{aff}\}$  of reflections  of $\W_{aff}$. }

{Let $s=s_A\in  \SS_{aff}$ be a simple affine reflection, with $A=(\alpha, r)\in \Delta_{aff}$. We associate to $s$ a positive  root  $ \beta \in \Phi^+$  such that $\alpha = e(\beta)\beta$  . When the root system $\Phi$ is reduced, $\beta$ is determined by this relation. When  the root system $\Phi$ is not reduced, we choose $\beta$ to be  the unique root  such that,   either $\beta$   is reduced and not multipliable, or $\beta$ is multipliable and $\U_{A}\neq  \U_{2\beta,   e(2\beta)^{-1}r} \U_{ \alpha, r+1}$, or $\beta$ is not reduced and $\U_{A}= \U_{\beta,e(\beta)^{-1}r} \U_{ \alpha, r+1}$  \cite[\S 4.2]{Vig1}.   We then set $\beta_s:=\beta$, we denote by  $\G'_{s}:= \G'_{\beta_s}$ the subgroup generated by $\U_{\beta_s}\cup \U_{-\beta_s}$,  we let $ \mathscr U_{s}:=\U_{\beta_s, e(\beta _s)^{-1}r } $, $\mathscr U_{s}^{op}:=\U_{-\beta_s, - e(\beta_s )^{-1}r  } $ and consider  the compact open subgroup  of $\G'_{s}$
 \begin{align*}\mathscr{G}'_s \ \text{  generated by  } \ \mathscr U_{s}\cup \mathscr U_{s}^{op}.
  \end{align*}
  Later, we will also consider the groups $\mathscr U_{s,+}:=\U_{\beta_s, e(\beta _s)^{-1}(r +1)} $, $\mathscr U_{s,+}^{op}:=\U_{-\beta_s, - e(\beta_s )^{-1}(r -1) } $, 
  $\Uf_{s}:=\mathscr U_{s}/\mathscr U_{s,+}$ and $\Uf_{s}^{op}:=\mathscr U_{s}^{op}/\mathscr U_{s,+}^{op}$. The set  $\Pi$ of simple roots  of $\Phi$ is contained 
in the set $\Pi_{aff}:=\{ (\beta_s , e(\beta_s)^{-1}r) \ | \ s\in S_{aff}\}$    only if  $\U_{\beta ,0} \neq \U_{2\beta ,0}\U_{\beta ,e_\beta^{-1}}$ when $\beta\in \Pi$ is multipliable.   \\ To an element   $u\in \mathscr U_{s}-\{1\}$ is associated a lift of $s$ in ${\mathscr N}_\G$.  It is the single element $n_s(u)$ in the intersection ${\mathscr N}_\G \cap \U_{-\beta_s }u \U_{-\beta_s }$. As this intersection  is equal to  ${\mathscr N}_\G \cap \mathscr U_{s}^{op}u \mathscr U_{s}^{op}$  \cite[6.2.1 (V5)]{BTI} \cite[\S 3.3 (19)]{Vig1}, we have $ n_s(u)\in \mathscr{G}'_s $.  
 \begin{equation}\label{defiadmi}\textrm{The elements  $n_s(u) \in {\mathscr N}_\G\cap \mathscr{G}'_s$ for $u\in \mathscr U_{s}-\{1\}$, are called the admissible lifts of $s$. } \end{equation}
  \begin{lemma}\label{admissible}  When $s\in S$, we have 
   $\mathscr{G}'_s \subset  \G'_{s} \cap \K$.
\end{lemma}
\begin{proof}  When $s\in S$ we have $r=0$; the groups  $\U_{\beta_s,0} $ and $\U_{-\beta_s,0}$, hence also $\mathscr{G}'_s$, are contained in $ \G'_{s} \cap \K$.
\end{proof}
 \begin{notation}  \label{qsnot}To  $s\in \SS_{aff}$ we associate  
\begin{itemize}
\item[-]  the facet   of the alcove $\Cp$ fixed by $s$, and the parahoric subgroup $\KK_s$ of $\G$ fixing  that facet that is to say the fixer of the facet in the kernel $\ker \kappa_G$ of the Kottwitz homomorphism; we have $\Zp\cap \KK_s=\Zp^0$ \cite[Prop. 3.15]{Vig1} and 
 $\mathscr{G}'_s \subset  \KK_s$ \cite[formulas (4.4), (4.9)]{Vig1}.
\item[-]  the pro-unipotent radical $\KK_s ^1$ of $\KK_s $ and  the finite reductive group $\Kf_s:=\KK_s/\KK_s ^1$;
\item[-] the order $q_s$  of the unipotent radical  of a proper parabolic subgroup of $\Kf_s$; it is a nontrivial power of   $q$;
\item[-] the subgroup $\Gf '_{s}$   of $\Kf_s$ generated by $\Uf_{s}\cup \Uf_{s}^{op}$ and the group $\Zf'_{s}:=\Gf '_{s} \cap \Zf$;
 \item[-] some admissible lift $n_s\in {\mathscr N}_\G \cap \mathscr{G}'_s $ of $s$.
 \end{itemize}
\end{notation}
 \begin{rema}\label{Zsk}   The group  ${\Zp}'_s={\Zp} \cap \mathscr G'_s$ is equal to  $\Zp^0  \cap \mathscr G'_s $ (Notation  \ref{qsnot}); the images  in   $\Kf_s=\KK_s/\KK_s ^1$ of  $\mathscr U_{s}, \mathscr U_{s}^{op} $ are respectively $\Uf_{s}, \Uf_{s}^{op}$ \cite[Prop. 3.23]{Vig1}; hence the image in  $\Kf_s$ of $\mathscr{G}'_s$  is $\Gf '_{s}$ and the image  in $\Kf_s$ of ${\Zp}'_s$  is contained in $ \Zf'_{s}$; 
 the group $ \Zf'_{s}$ is normalized by   the image of $n_s$ in $\Kf_s$. 
 \end{rema}

  \medskip An element $z\in \Zp$ acts by translation on $\V$ via the unique homomorphism $\Zp \xrightarrow{\nu}\V$ defined by
  \begin{equation}\label{normalization}
    \lp\nu(t),\, \chi\rp =-\val_{ \F}(\chi(t))  \qquad \textrm{for any } \chi\in \Phi \textrm{ and }t \in \T. 
\end{equation} 
 The action of $\W_0$, seen as the Weyl group of $\Sigma$,  and of $\Zp$ on  $\V$ combine to define an action, also denoted by $\nu$,  of $\Np$  on $\V$ by affine automorphisms.   
 The action of $\Np $  on $\V$   respects the set $\mathfrak H$ of   affine hyperplanes.  Being trivial on $\Zp_0$, it identifies with an action of the Iwahori Weyl group $\W$,  also denoted by $\nu$.  The  action of $\W'$ on $\V$ gives an isomorphism $\W' \simeq \W_{aff}$. This allows to identify $\W_0\subset \W_{aff}$ with the  subgroup $\W'_0\subset \W'$ of elements fixing $0$.   The subgroup   $\Omega \subset \W$ normalizing $\Cp$ is isomorphic to the image of the Kottwitz morphism $\kappa_\G$; it is a commutative finitely generated group.
 
\medskip   
 These considerations  imply that $\W$ admits two decompositions 
$\W=\W' \rtimes \Omega = \Lambda \rtimes \W'_0$ as a semidirect product. Therefore, 
\begin{equation}\label{Weylgroups}\W \simeq \W_{aff} \rtimes \Omega \simeq  \Lambda \rtimes \W_0.
\end{equation}
We inflate the 
   length function of  the Coxeter system $(\W_{aff}, S_{aff})$  to a map $\ell: \W\to \mathbb  N$, called the length of $\W$, such that  $\Omega \subset \W$ is the subset of elements of length $0$. 
   
    \medskip  
 The   positive roots $\Sigma^+$ and $\Phi^+$ take nonnegative values on the dominant Weyl chamber $\Dp$. The
 affine roots ${\Sigma_{aff}^+}:=\{(\alpha, r),\: \alpha\in\Sigma, \,r>0\}\cup  \Sigma^+ $ 
 take nonnegative value on $\Cp$. We set  ${\Sigma_{aff}^-}=-{\Sigma_{aff}^+} $.
  The action of  $w=w_0\lambda\in  \W=\W'_0\ltimes \Lambda$ on $\V$ induces an action 
\begin{equation}\label{actionLambda}
 (\alpha, r)\mapsto  (w_0(\root),r - \langle \nu(\lambda), \root\rangle)
\end{equation}
on $\Sigma_{aff}$.  For $w\in \W$, 
the group $\hat w \Uu _{(\alpha, r)}\hat w^{-1}$ does not depend on the lift $\hat w\in \Np$ and is equal to $\Uu _{w(\alpha, r)}$.
 The length  $\ell(w)$ of $w\in \W$  
 is the number of $A\in{\Sigma_{aff}^+}$ such that
$w(A)\in { \Sigma_{aff}^-}$ is negative. 
The length function $\ell:\W\rightarrow \mathbb N$ inflates to a length function $\ell:\W(1)\rightarrow \mathbb N$.

\subsection{\label{finitegroup}The finite reductive group  $\Gf$}

Let $\mathbf{G}_{x_0}$ and $\mathbf{G}_{\Cp}$ denote the connected Bruhat-Tits group schemes over $ \O_\F$ whose $\O_\F$-valued points are $\K$ and $\Iw$ respectively. 
   Their reductions over the residue field $\mathbb{F}_q$ are denoted by $\overline{\mathbf{G}}_{x_0}$ and $\overline{\mathbf{G}}_\Cp$. 
   Note that $\Gp= \mathbf{G}_{x_0}(F)=\mathbf{G}_{\Cp}(F)$.
   By \cite[3.4.2, 3.7 and 3.8]{Tit}, $\overline{\mathbf{G}}_{x_0}$ is a connected  reductive $\mathbb F_q$-group.  
   Let 
   $\overline{\mathbf{T}}$ and  $\overline{\mathbf{Z}}$  be the reduction over  $\mathbb{F}_q$ of  the connected  group schemes over $\O_\F$ whose $\O_\F$-valued points are $\Tp^0$ and $\Zp^0$ respectively. The group $\overline{\mathbf T}$  is a maximal
$\mathbb F_q$-split subtorus of $\overline{\mathbf{G}}_{x_0}$ of centralizer $\overline{\mathbf Z}$ which is a $\mathbb F_q$-torus. 
    
The quotient $\Gf:=\K/\K^1$ of $\K$  by  $\K^1$ is a finite reductive group  isomorphic to  $\overline{\mathbf G}_{x_0}(\mathbb F_q)$. The Iwahori subgroup $\Iw{\subseteq  \K}$  is the preimage in $\K$ of the minimal parabolic subgroup $\Bf$   of $\Gf$ of  unipotent radical  $\Uf :=(\K\cap \U)/ (\K^1\cap \U)$ and Levi decomposition $\Bf=\Zf \Uf$ where  
$\Zf=\Zp^0/\Zp^1\cong  \overline{\mathbf{Z}}(\mathbb F_q) $ contains the maximal split torus   $\Tf: =\Tp^0/\Tp^1\cong \overline{\mathbf{T}}(\mathbb F_q)$;  the group $\Zf$ is a maximal torus of $\Gf$. The pro-$p$ Iwahori subgroup $\I {\subseteq \K}$   is the preimage in $\K$  of $\Uf$; the $\Gf$-normalizer  of $\Tf$  is  denoted by $\norm_\Gf$. 
The relative root system of $\Gf $ attached to $\Tf$ identifies with the  set $ \Phi_{\Gf }:=\{ \beta_s \ | \ s\in S\}$ (notation in   \S \ref{affine+decomp})  of roots $\beta\in  \Phi$ which are not multipliable, \emph{i.e.} such that $2\beta \not\in \Phi$ or   such that $\Uu_{\beta}\neq \Uu_{2\beta} \Uu_{\beta + e(\beta)}$  if $2\beta\in \Phi$.  The unipotent subgroup $\Uf_{\beta_s} =\Uf_s \subset \Gf$ attached to $\beta_s \in \Phi_\Gf$ is $\Uu_{\beta_s}/ \Uu_{\beta_s+e_{\beta_s}}$.
   The set $ \Pi_{\Gf}=\{ \beta \in\Phi_{\Gf} \ | \  \beta/2 \not\in\Phi_{\Gf} , \ e_\beta \beta \in \Delta \}$  is a basis of  $ \Phi_{\Gf}$. The map sending  $\beta\in  \Pi_{\Gf}$ to the unique positive reduced root  $\alpha \in \Phi$ in $ \mathbb Q\beta  $ induces a bijection between $\Pi_\Gf$ and   the basis $\Pi$ of $\Phi$, giving  a bijection between the set $S_{\Gf}:=\{s_\beta \ | 
   \beta\in \Pi_{\Gf}\}$ of reflections in the Weyl group of $\Phi_\Gf $ and $S$. 
  The root systems $\Phi_\Gf $ and $ \Phi$ have isomorphic Weyl groups \cite[3.5.1]{Tit}.  
   
   The definition of a  strongly split $BN$-pair of characteristic $p$ is given in  \cite[Def. 2.20]{CE}.  By  \cite[\S 3.5 Prop. 3.25]{Vig1}, we have:
   \begin{equation}\label{ssBNp}
\text{\it$(\Gf, \Bf , \mathscr N_{\Gf}, S)$ and the decomposition $\Bf = \Zf \Uf$ form a strongly split $BN$-pair of characteristic $p$}.
\end{equation}
where $S$ is seen as a subset of  $\Wf=  \mathscr N_\Gf /\Zf $ (isomorphic to the Weyl group of $\Sigma$).  
 We have $\mathscr N_\Gf =({\mathscr N}_\G\cap \K)/({\mathscr N}_\G\cap \K^1) =({\mathscr N}_\G\cap \K)/\Zp^1$ and  the reduction map ${\mathscr N}_\G\cap \K \to \mathscr N_\Gf $ induces an isomorphism  $({\mathscr N}_\G\cap \K)/\Zp^0\overset{\simeq}\rightarrow \Wf $.  The  group $({\mathscr N}_\G\cap \K)/\Zp^0$ is contained in $\W'$  and its action on $V$ fixes  $0$. 
    We recall  that the  the fixer $\W'_0$   of $0$  in $\W$  defined after \eqref{normalization}   has the same order as $\W_0$, hence as $\Wf $ and as  $({\mathscr N}_\G\cap \K)/\Zp^0$. We deduce $\W'_0= ({\mathscr N}_\G\cap \K)/\Zp^0$.  The  natural   surjective map $\W\to \W_0$ induces an isomorphism $({\mathscr N}_\G\cap \K)/\Zp^0=\W'_0 \overset{\simeq}\rightarrow \W_0$.     The group $\W'_0(1)$ is the preimage of $\W'_0$ in $\W(1)$ (Notation \ref{not(1)}). In summary, we have
 \begin{equation}\label{isomorphismW} 
\Wf =  \mathscr N_\Gf/\Zf \cong \W_0={\mathscr N}_\G/\Zp   \cong ({\mathscr N}_\G\cap \K)/\Zp^0 =\W'_0 \subset \W ; \quad \mathscr N_\Gf = ({\mathscr N}_\G\cap \K)/\Zp^1= \W'_0(1)  \subset  \W(1).
\end{equation} 
   The action of $\Wf$ on $\Sigma$  via the identification $\Wf\cong \W_0'$ coincides with the natural action of $\Wf$ on $\Sigma$, and the length function on $\Wf$  coincides with the restriction to $\W_0'$ of $\ell:\W\rightarrow \mathbb N$. The length function on $\Wf$ is still denoted by $\ell$. It inflates to a map $\norm_\Gf\rightarrow \mathbb N$ which we also denote by $\ell$.

Because of \eqref{ssBNp}, we have the Bruhat decompositions $\Gf=\sqcup_{w\in \Wf} \Bf w \Bf= \sqcup_{n\in \mathscr N_\Gf} \Uf n \Uf$  \cite[Prop. 6.6]{CE}, corresponding to the disjoint unions $\K=\sqcup_{w\in\W'_0} \Iw  w \Iw= \sqcup_{w\in\W'_0(1)}\I  w \I$.  We do not specify the choice of a lift for $w\in \Wf$ in $\norm_\Gf$ since the double coset $ \Bf w \Bf$ does not depend on it (compare with the decompositions in \ref{bruhatdec}).

\subsection{\label{subsec:lifts}Lifts of the elements in the finite Weyl group}

\begin{rema}\label{liftsWf}
We will keep in mind the following diagram:

\begin{equation}\xymatrix{
 \norm_\G\cap\K      \ar[r]& ( \norm_\G\cap\K)/\Zp^0=\W_0'\subset \W\ar[d]_{\mod \K^1}^{\cong}      \\
&
      \norm_\Gf/\Zf=\Wf.   }\label{compinj}\end{equation}

      Given an element $w$  in $\Wf$, we may pick a lift for $w$ in $\norm_\Gf$. We may also identify $w$ with an element of $\W_0'$ and pick a lift in  $\norm_\G\cap \K$.  In \S\ref{par:dist} in particular, we will go back and forth between lifts  in $\norm_\Gf$ and in $\norm_\G\cap \K$ for elements in $\Wf$.

\end{rema}

For $s\in S$ choose  $n_s\in {\mathscr N}_\G\cap \mathcal G'_s $  an admissible lift of $s$. This yields  a map $s\mapsto n_s: S\to {\mathscr N}_\G\cap \K$ (Lemma \ref{admissible}).

 \begin{prop}\label{lift}   There is a unique extension of the map $s\mapsto n_s: S\to {\mathscr N}_\G\cap \K$ to a map $w\mapsto n_w:\Wf \to {\mathscr N}_\G\cap \K$
 such that $n_{ww'}=n_w n_{w'}$ for $w,w'\in \Wf$ such that $\ell(ww')=\ell(w)+\ell(w')$. 
 \end{prop} 
 This will be useful in \S\ref{sec:frob}.
  \begin{proof} This result   is proved in  \cite{AHHV}[IV.6 Prop.] for the group $\Gf$. The  arguments in loc. cit. are valid for $\G$. The unicity follows from the reduced decomposition of $w\in \Wf$ as a product of  elements of $S$. The existence follows from  \cite[IV.1.5,Prop. 5]{Bki-LA} once we know that for $s,s'$ distinct in $S$, and $m$ the order of $ss'$, then $ (n_s n_{s'})^r=  (n_{s'}n_s )^r$ if $m=2r$ and $ (n_s n_{s'})^rn_s=  (n_{s'}n_s )^rn_{s'}$  $m=2r+1$. 
  That  follows from \cite[Prop. (6.1.8) (9)]{BTI}, because we may assume that $\bf G$ is semisimple simply connected of relative rank $2$, by replacing $\bf G$ by  the simply connected covering  $\bf M^{sc}$  of the derived group of  $\bf M=\Ker \alpha \cap \Ker \alpha'$ where $\alpha, \alpha' \in \Pi$ correspond to $s,s'$. \end{proof}

\begin{notation}\label{notnf}
For $w\in \Wf$, we denote by $\nf_w$ the image of $n_w$ in   $\Gf$ via the map $\K \to \K /\K^1= \Gf$.
We have 
$$ \nf_{ww'}=\nf_w \nf_{w'}\textrm{ for $w,w'\in \Wf$ such that $\ell(ww')=\ell(w)+\ell(w')$}.$$ 
\end{notation}

\begin{rema}  If $\longw$ is the longest element of  $ \Wf$ and $s\in S$,  then $ s':= \longw s \longw^{-1}$  is an element of $S$ and we have \begin{equation}\label{conjbynw}
n_{\longw} n_s = n_{s'} n_{\longw}\textrm{ and }\nf_{\longw} \nf_s = \nf_{s'} \nf_{\longw}.
 \end{equation} 
We have indeed  $\longw=   s' \longw s$ with $\ell(\longw  s)= \ell(\longw)-1$. So $n_{\longw}= n_{s'} n_{\longw s} $ and $n_{\longw}n_s= n_{s'} n_{\longw s}n_s=n_{s'}n_{\longw}$. Note that  \eqref{conjbynw} implies  $n_{\longw} n_w = n_{\longw  w \longw^{-1}} n_{\longw}$ and $n_w n_{\longw}  =n_{\longw} n_{\longw  w \longw^{-1}} $ for all $w\in \Wf$, by  induction  on the length of $w$  and using  $\longw^2=1$. 
\end{rema}

\begin{exa} For example, when $\mathbf G$ is $\F$-split, we fix an épinglage for $\Gp$ as in SGA3 Exp. XXIII, 1.1. In particular, to $\root\in\Phi$ is attached a central isogeny $\phi_\root: {\rm SL}_2(F)\rightarrow \Gp_{s_\root}$ where $\Gp_{s_\root}$ is the subgroup of $\Gp$ generated by $\Uu_\root$ and $\Uu_{-\root}$. 
  With the notation of \ref{affine+decomp}, $\Pi=\Delta$ and $ \Pi_m=\Delta_{aff}\setminus \Delta$.
     For  $\root\in\Pi $,   set $n_{s_{\root}}:= \phi_\alpha\begin{pmatrix}0&1\cr -1&0\end{pmatrix}$.  For  $\root\in\Pi_m$, set  $n_{s_{(\root, 1)}}:= \phi_\alpha\begin{pmatrix}0&-\varpi^{-1}\cr \varpi&0\end{pmatrix}$. We have  $\T=\Zp$ and $q_s=q $  for any $s\in S_{aff}$ (Notation \ref{qsnot}).
\end{exa}

\subsection{Parabolic  subgroups}

\subsubsection{Standard parabolic subgroups\label{standparab}}

Let $J\subseteq \Pi$  and $\Phi_J \subseteq \Phi$   the  subset of all linear combinations of elements in $J$. 
To $J$ we attach the following subgroups of $\mathbf G$: 
 the subtorus $\mathbf T_J$ of $\mathbf T$  with dimension 
$dim(\mathbf T)-\vert J\vert$ equal to the connected component of $\bigcap_{\root\in J}\ker \root\subseteq \mathbf T$  and the Levi subgroup $\mathbf M_J$ of $\mathbf G$ defined to be the centralizer of $\mathbf T_J$. 
The group $\mathbf M_J$  is  a reductive connected algebraic $\F$-group 
   of maximal $\F$-split torus $\mathbf T$ and minimal parabolic ${\mathbf B}_J:={\mathbf B} \cap  {\mathbf M_J} ={\mathbf Z} ({\mathbf U}\cap {\mathbf M_J})$; we have $\mathbf N_{\mathbf M_J} (\mathbf T)=\mathbf N_{\mathbf G} (\mathbf T)\cap   {\mathbf M_J} $.  
 The restriction $\varphi_J$ of $\varphi$ to the root datum $(\Zp, (\U_\root)_{\root\in  \Phi_J })$ of $\M_J$ associated to $\T$ is a special discrete valuation compatible with $\val_F$.

The same objects   as the ones we attached to $\G$ in  \S\ref{products} and \S\ref{finitegroup}  can be attached to   ${\M_J}$.   We introduce an index $J$ for the objects   attached to ${\M_J}$. When the set $J=\{\alpha\}$ contains a single element $\alpha$ we denote $\M_J$ by $\M_\alpha$.

In particular, we associate to $(\mathbf T, {\mathbf B}  _J,  \varphi_J)$  a reduced root subsystem $\Sigma_J$ with basis $\Delta_J$, a special maximal parahoric subgroup $\K_J$ of   $\M_J$ of pro-unipotent radical $\K_J^1$, an Iwahori subgroup $\Iw_J$ of unique pro-$p$ Sylow subgroup $\Uu_J$, a   finite,   Iwahori,  pro-$p$ Iwahori Weyl group of $\M_J$  denoted respectively by $\W_{J,0}$,  $\W_J$, and $\W_J(1)$, and a real vector space $\V_J$.  We note that $\M_J$ and  $\G$ have the same $\Zp$ and $\norm_{\M_J}= \norm_\G \cap \M_J$. We have:
\begin{itemize}
 
\item[-] $\Sigma_J=\{e_\alpha \alpha \ | \ \alpha \in \Phi_J\}\subset \Sigma, \quad \Delta_J=\{e_\alpha \alpha \ | \ \alpha \in J\}\subset \Delta$,   
 
\item[-]  $\K_J= \K\cap \M_J, \quad \K_J^1=  \K^1\cap \M_J, \quad \Iw_J= \Iw\cap \M_J, \quad \I_J= \I\cap \M_J$ (these equalities are justified in \cite[Proposition 4.2]{VigFunct}), 
  
\item[-] the pro-$p$ Iwahori Weyl group $\W_J(1)$ of $\M$ coincides with the preimage  of $\W_J$ in $\W(1)$ by the quotient map $\W(1)\to \W$. Therefore the notation $\W_J(1)$ is consistent with the one introduced in \ref{not(1)}.\end{itemize}

Note that   $\I_\emptyset=\Zp^1$ and $\V_\emptyset=\{0\}$. When $J\neq \Pi$,  the  real vector space $\V_J$ is only a strict quotient of $\V$ (see \S\ref{affine+decomp}). This difficulty arises  for the semisimple  Bruhat-Tits building but not for  the   extended  Bruhat-Tits buildings of $\G$ and of $\M_J$. We  put  on $\V_J$  the $\W_{J,0}$-scalar product  image of the $\W_0$-scalar product on $\V$ by the surjective linear map 
$$\V\xrightarrow{p_J} \V_J \quad \langle \alpha,v\rangle= \langle \alpha,p_J(v)\rangle \quad (v\in \V, \alpha \in \Sigma_J).$$ For $(\alpha, n)\in \Sigma_{J,aff}$, the inverse image of the affine hyperplane $\Ker _{V_J}(\alpha+n)\in \mathfrak H_J$ is $\Ker _{\V}(\alpha+n)\in \mathfrak H$ (defined after \eqref{Gp'}). The image $p_J(\Cp)$ of the  alcove $\Cp$  is contained in  the  alcove $\Cp_J$ of $(\V_J,\mathfrak H_J)$.  Recall that $\Sigma_{J,aff} $ is contained in $\Sigma_{aff} $.  For $A\in \Sigma_{J,aff} $, we attach to  the orthogonal reflection $s_{A,J}$ of $V_J$ with respect to $\Ker_{V_J}(A)$ the element $s_A$. This defines a map $\mathfrak S_{J}\to \mathfrak S$   which induces an injective homomorphism $\W_{J,aff} \hookrightarrow \W_{aff}$ (see the notations in \ref{affine+decomp},  those relative to $\M_J$ have an index $J$). 

The image  of $S_J=\{s_{\alpha,J}, \alpha\in \Delta_J\}$ is contained in $S$, but the image of $S_{J,aff}$  is not contained in $S_{aff}$. 
 The map $p_J$ is $\W_{J}$-equivariant  \cite[Lemma 4.1]{VigFunct}. We have 
 $$\W_{J} \simeq \Lambda \rtimes \W_{J,0} \simeq \W_{J,aff} \rtimes \Omega_J.$$  
\begin{rema}\label{compadmissible}For $s\in S_J$ (but not in  $S_{J,aff}$), the positive root $\beta_s\in \Phi_J$,  the groups $\mathscr U_s$ and $\mathscr G'_s$
associated to $s$,  and the admissible lifts of $s$ (as in \S \ref{affine+decomp}) for the group $\M_J$ are the same as those  for $\G$.  
\end{rema}
\medskip Let  $\N_J$ be the subgroup of $\G$ generated by all $\U_\root$ for $\root\in \Phi^+\setminus \Phi_J$
and $\N^{op}_J$ be the subgroup of $\G$ generated by all $\U_\root$ for $\root\in \Phi^-\setminus \Phi_J$.
Then $\P_J:=\M_J\N_J$ is a parabolic subgroup of $\G$. If $J=\emptyset$, then  $\P_\emptyset=\B$,  $\M_\emptyset=\Zp$ and $\N_\emptyset=\U $. A parabolic subgroup of the form $\P_J$ for  $J\subseteq \Pi$ is called \emph{standard}
and  $\P^{op}_J=\M_J\N^{op}_J$ is  its opposite parabolic subgroup.  

\medskip 

 As before in \S\ref{finitegroup}, $\Pi$ identifies canonically with the set $\Pi_{\Gf} $ of simple roots of $\Gf$ with respect to $ \Bf$  in the root system  $\Phi_{\Gf}$ of $\Gf$ with respect to $\Tf$. The subset $J\subset \Pi$ identifies with a subset $J_{\Gf}\subset \Pi_{\Gf}$ and the root system $\Phi_J$  generated by $J$ in $\Phi_{\Gf}$ with the the root system  $\Phi_{\Mf_J}$   generated by $J_{\Gf}$ in $\Phi_{\Gf}$.    
The standard parabolic subgroup   of $\Gf$ attached to $J_{\Gf}$    is  equal to  
the image $\Pf_J$  of $\P_J\cap\K$ in $\Gf $ \cite[Proposition 3.26]{Vig1} of Levi decomposition  $\Pf_J= \Mf_J\Nf_J$ equal to the image of $\P_J\cap\K = (\M_J\cap\K)(\N_J\cap \K)$.   If $J=\emptyset$, then $\Pf_\emptyset=\Bf$.
The $\Mf_J$-normalizer of  $\Tf$ is $\norm_{\Mf_J}=\Mf_J\cap \norm_\Gf$.

The set $S_{J_{\Gf}}$ corresponding to $J_{\Gf}$ in $S_{\Gf}$ identifies canonically with $S_J $  and the group $\norm_{\Mf_J}/\Zf$  with the subgroup  $\Wf_J$  of $\Wf$ generated by $S_J$. As  in \eqref{isomorphismW}, the finite Weyl group $\Wf_J$ identifies with a subgroup of $\W_J$, and $\norm_{\Mf_J}$ as a subgroup of $\W(1)$. As in \eqref{ssBNp},
  \begin{equation}\label{ssBNpJ}\text{\it$(\Mf_J, \Bf_J, \norm_{\Mf_J}, S_J)$ and he decomposition $\Bf _J= \Zf \Uf_J$  form a strongly split $BN$-pair of characteristic $p$ } 
\end{equation}
 where $\Bf_J=\Bf\cap \Mf_J$ and  $\Uf_J=\Uf\cap \Mf_J$. 
We have the decompositions $\Mf_J=\sqcup_{w\in \Wf_J} \Bf_J w \Bf_J= \sqcup_{n\in \mathscr N_{\Mf_J}} \Uf_J n \Uf_J$ and  $\Pf_J=\sqcup_{w\in\Wf_J} \Bf w \Bf$.

 \subsubsection{\label{posimonoid}Iwahori decomposition and positive (resp. negative) monoids}

 Let $J\subset \Pi$. We will write from now on  $\P$ instead of $\P_J$, $\M$ instead of $\M_J$ and  $\N$ instead of $\N_J$ and will otherwise replace the index $J$  for the objects introduced in \S\ref{standparab}  by the index $\M$.   In particular, the finite, Iwahori, and pro-$p$-Iwahori Weyl groups of $\M$ are now  respectively denoted  by $\W_{0,\M}$, $\W_\M$ and $\W_\M(1)$. 
 
 \medskip

  The  group $\Iw_\N=\Iw\cap \N=\I\cap \N=   \I_\N$ (resp. $\Iw_\Nm=\Iw\cap \Nm=\I\cap \Nm=\I_\Nm$)  is the image  by the multiplication map of 
 $ \prod_{\alpha\in \Sigma^+-\Sigma_\M^+} \Uu_{(\alpha, 0)}$ (resp. 
 $ \prod_{\alpha\in \Sigma^--\Sigma_\M^-} \Uu_{(\alpha, 1)}$), as in \S \ref{products}. 

 Similarly,   the group $\Iw_\M$, resp. $\I_ \M$ (which was denoted by $\Iw_J$, resp. $\I_\J$  in \S\ref{standparab}) is the image of $ \prod_{\alpha\in \Sigma_\M^-} \Uu_{(\alpha, 1)}\times\T^0\times \prod_{\alpha\in \Sigma_\M^+} \Uu_{(\alpha, 0)}$ (resp. $ \prod_{\alpha\in \Sigma_\M^-} \Uu_{(\alpha, 1)}\times\T^1\times \prod_{\alpha\in \Sigma_\M^+} \Uu_{(\alpha, 0)}$).   

\medskip

 The groups $\Iw$ and $\I$ have an Iwahori decomposition with respect to $\P$:   \begin{equation}\Iw=\Iw_\N\:\Iw_ \M\: \Iw_\Nm\textrm{ and } \I=\I_\N\: \I_\M\: \I_\Nm\label{f:iwadecomp}\end{equation} 
 with any  order of the products.      

   \medskip

An element  $m\in \M$  contracts $\I_{\N}$ and dilates $\I_{\Nm}$ if it satisfies the condition:
$m \I_{\N} m^{-1}\subseteq \I_{\N}$ and $ m^{-1} \I_{{\Nm}} m\subseteq \I_{{\Nm}}$ (compare with \cite{Selecta} II.4 or \cite[(6.5)]{BK}).  
 This property of an element  $m\in \M$  depends only on
the double coset $\I_\M m  \I_\M$.  
Such an element will be called $\M$-positive. Note that
if $m\in \K_\M$ then $m \I_\N  m^{-1}= \I_\N $  and  $m^{-1}\I_{\Nm } m=\I_{\Nm}. $

\medskip
The monoid  of $\M$-positive elements in $\M$ will be denoted by $\M^{+}$. 
 The monoid $(\M^{+})^{-1}\subseteq \M$ will be denoted by $\Mm$. Its elements 
 contract $\I_{\Nm}$ and dilate $\I_{\N}$ and  are called 
 $\M$-negative.

The elements $w$ in $\W_\M$ satisfying  $w(\Sigma^+-\Sigma_\M^+)\subset \Sigma_{aff}^+$  are called  $\M$-positive and form a  monoid denoted by $\WMp$. Denote by  $\WMm$ the monoid $(\WMp)^{-1}$.

\begin{rema} \label{lemma:Mp}We have
\begin{enumerate}
\item The monoid $\WMp$ is isomorphic to the semi-direct product $ \Wf_\M\ltimes \Lambda _{\M^+}$ where
$$\Lambda _{\M^+}:=\{\lambda\in \Lambda  \ | \  -\alpha(\nu(\lambda)) \geq 0  \ \textrm{ for all $\alpha\in\Sigma^+-\Sigma^+_\M$}\}.$$
\item  The set $\WMp$  is a system of representatives of the double cosets $\Iw_\M\backslash  \Mp  /\Iw_\M$. The preimage $ \W_{\M^+}(1)$ of $\WMp$ in $ \W(1)$  
is a system of representatives of the double cosets  $\IM\backslash \Mp/\IM$.
\item  The set $\WMm$  is a system of representatives of the double cosets $\Iw_\M\backslash  \Mm  /\Iw_\M$. The preimage $ \W_{\M^-}(1)$ of $\WMm$ in $ \W(1)$  
is a system of representatives of the double cosets  $\IM\backslash \Mm/\IM$.\end{enumerate} 
To justify the above, we refer to \cite[Lemma 5.21]{Compa} when $\G$ is split, but this is general: in particular for (1)  see \cite[\S 2.1]{Vig5}.

\end{rema}

\subsubsection{\label{par:dist}Distinguished cosets representatives in $\Wf$ and $\W_0'$} 
Let $J\subseteq \Pi$ and $\P:=\P_J $, $\Pf:=\Pf_J $ the corresponding standard parabolic subgroups of $\G$ and $\Gf$ respectively.
 The following lemma is proved for example in \cite[2.3.3]{Carter}. The reduced root system $\Sigma_\M$ attached to $\M$ was previously denoted by $\Sigma_J$ in \S\ref{standparab}, and the  Weyl group $\Wf_\Mf$  of $\Mf$ by $\Wf_J$ .

\begin{lemma} \label{lemmaDJ}The set $\dist$  of  elements $d\in \Wf$  satisfying $ d^{-1}  \Sigma_{\M}^+\subset \Sigma^+$ 
is a system of representatives of the right cosets $\Wf_{\Mf}\backslash \Wf$ . It satisfies
$\ell( wd) = \ell(w) + \ell(d)$
for any $w\in \Wf_\Mf $  and $d\in \dist$ . In particular, $d$ is the unique element with minimal length in $\Wf_\Mf d$.

\end{lemma}

\begin{rema}  
\begin{enumerate}
\item Recall that $\Wf$ is naturally isomorphic to the subgroup $\W_0'$ of $\W$. 
The set $\dist$  may therefore be seen as the subset of  the elements in  $\W_0'$ satisfying 
 $d^{-1}  \Sigma_{\M}^+\subset \Sigma^+$.  
 
 \item For $d\in \Wf$ seen in  $\W_0'$, the condition 
 $ d^{-1}  \Sigma_{\M}^+\subset \Sigma^+$  is equivalent to
$d^{-1} (\I\cap \M) d \subset \I$ {(the latter condition  does not depend on the choice of a lift for $d$ in $\norm_\G$)}.  This follows immediately from the decomposition of $\IM= \I\cap \M$  given in \S\ref{posimonoid}. 
By reduction modulo $\K^1$, we also have  $d^{-1}( \Uf\cap \Mf) \, d\subset \Uf$  for any $d\in \dist$ (this  does not depend on the choice of a lift for $d$ in $\norm_\Gf$).
\item By \eqref{f:iwadecomp}, point (2)  implies that 
\begin{equation}\label{goodadd}\I\hat w \I\hat d\I=\I\hat w  \hat d\I \textrm{for $w\in \Wf_\Mf$, $d\in \dist$ seen in $\W$ and with respective lifts $\hat w$, $\hat d$ in  $\norm_\G\cap\K$ and}\end{equation}
\begin{equation}\label{goodaddfinite}\Uf \hat w \Uf \hat d\Uf=\Uf \hat w  \hat d\Uf \textrm{ for $w\in \Wf_\Mf$, $d\in \dist$ with respective lifts $\hat w$, $\hat d$ in  $\norm_\Gf$}.\end{equation}

\end{enumerate}
\label{rema:incluDJ}
\end{rema}

\begin{lemma}\label{lem:addlengths} For  $w\in \WMp$ and $d\in \dist$, we have
$\ell(w)+\ell(d)=\ell(wd)$.  In particular, for $\hat w, \hat d\in \norm_\G$  lifting $w$ and $d$ respectively, we have $\I \hat w \I \hat d\I= \I\hat w \hat d\I$. \end{lemma}
\begin{proof} This is proved for $\G={\rm GL}(n,  \F)$ in \cite[Lemma 2.3]{Oparab}. The proof is general as noted in \cite[Lemma 2.22]{Vig5}. 
\end{proof}

\subsubsection{Double and  simple cosets decompositions in $\Gf$ and $\G$} Let $J\subseteq \Pi$ and $\P:=\P_J $, $\Pf:=\Pf_J $ the corresponding standard parabolic subgroups of $\G$ and $\Gf$ respectively, with Levi decompositions $\P=\M \N$ and $\Pf=\Mf \Nf$. 
In part (1) of the lemma below, the lifts $\hat d$ of $d\in \dist$ are taken in $\Gf$. In part (2) they are taken in $\K$.

 \begin{lemma}\label{lemma:PGU} 
 \begin{enumerate}
 \item
 We have $\Gf=\sqcup_{d\in  \dist } \Pf \hat d\Uf$ and for any $d\in  \dist $:    \begin{enumerate}
\item  $(\hat d\Uf \hat d^{-1}\Nf )\cap   \Mf   \subseteq \Uf\cap \Mf  $,
\item $(\Pf \cap \hat d\Uf   \hat d^{-1})\Nf  = (\Uf\cap \Mf ) \Nf =\Uf.$
 \end{enumerate}
 \item 
  We have $\G=\sqcup_{d\in  \dist } \P  \I \hat d\I=\sqcup_{d\in  \dist } \P  \hat d\I$ and for any $d\in  \dist $:
 \begin{enumerate}
 \item $\hat d^{-1}(\I\cap \M ) \hat d\subseteq \I$.
\item  $\hat d\I \hat d^{-1}\cap   \P \I  =  \hat d\I  \hat d^{-1} \cap  \I $,
\item $(\P \cap \hat d\I  \hat d^{-1})\N  = (\I\cap\M ) \N .$
 \end{enumerate}
 \end{enumerate}
 \label{PI}
\end{lemma}

\begin{proof}   
\begin{itemize}
\item[(1)]  That $\Gf$ is the union of all $\Pf \hat d \Bf=\Pf   \hat d\Uf$  for $d\in \dist $ follows from the Bruhat decomposition for $\Gf$ and  the decomposition of $\Pf $ as the disjoint union of all $\Bf  \hat w\Bf$ for $w\in \Wf _J$ (see \eqref{ssBNp} and \eqref{ssBNpJ}).   Let $d\in \dist$.
By  Remark \ref{rema:incluDJ}(2), we have $\Uf\cap \Mf  \subset \hat d \Uf \hat d ^{-1} $ and $\Uf^{op}\cap \Mf  \subset \hat d  \Uf^{op} \hat d^{-1}  $. 
Since $\hat d \Uf\hat d^{-1} \cap \Uf^{op} \cap \Pf\subset \hat d \Uf \hat d^{-1} \cap \Uf^{op} \cap \Mf  $,   it implies that $ \hat d \Uf \hat d^{-1} \cap \Uf^{op} \cap \Pf=\{1\}$.
Since  an element in $\Uf$ can be written uniquely as the product of an element in $\Uf\cap \hat d^{-1}\Uf\hat  d$ by an element in $\Uf\cap \hat d^{-1}\Uf^{op}\hat d$ (\cite[2.5.12]{Carter}), we deduce that  
$\hat d\Uf \hat d^{-1}\cap \Pf \subset \Uf$. Points (1)(a) and (1)(b) then follow easily (for (1)(b) also use $\Uf\cap \Mf  \subset \hat d \Uf\hat  d ^{-1} $).

\item[(2)]  Recall that here the lifts for $d\in \dist$ are in $\K$. The first identity 
 is proved for $\G={\rm GL}(n, \F)$ in \cite[6A1]{Oparab} and for $\G$ split in  \cite[Lemma 5.18]{Compa}.   \\ We review the arguments  to check that they generalize to the case of  an arbitrary reductive group.   
Recall that $\Wf\cong \W_0'$ indexes the doubles cosets of $\K$ mod $\Iw$, that $\dist $ is a system of representatives of  the right cosets $\Wf_\Mf \backslash \Wf$, and that for $w\in \Wf$ with lift $\hat w\in \K$, we have $\P  \I \hat w \I=\P \hat w \I$ since  $\hat w^{-1}\I_\Nm \hat w\subset \I$.
From the Iwasawa  decomposition $\G=\P \K$,  we get that $\G$ is the union of all $\P   \hat d \I$ for $d\in \dist $. It is a disjoint union because,  given  $d,d'\in \dist $,  if $\hat d'\in \P  d \I$ then $\hat d'\in  (\K\cap \P ) \hat d \I$.  Reducing mod $\K^1$ and using (1) gives $d'=d$ (we recall that  the image of $\K\cap \P$ in $\Gf$ is $\Pf$). For $d\in\dist$, we have  $\hat d^{-1}\I_\Nm \hat d\subset \K^1\subset \I$ since $\hat d$ lies in $\K$, so it is easy to see that $\P \I \hat d \I=\P \I_\Nm \hat d \I= \P \hat d \I$. \\
(a) This is Remark \ref{rema:incluDJ}(2).\\
(b) Note that $\hat d \I \hat d^{-1}$ is contained in $\K$.  An element in $\P  \I$ is of the form $ p u$ for $p\in \P $ and $u\in \I_\Nm$. Suppose that it lies in $\hat d\I \hat d^{-1}$.  Then $p\in \K\cap \P $ and the projection of $pu$  in $\K/\K^1=\Gf$ is equal to
$\bar p\in  \hat d\Uf\hat d^{-1}$ where $\bar p\in \Pf $ is the projection of $p$.  But $\Pf \cap  \hat d \Uf\hat  d^{-1}\subset \Uf$ by (1)(b) so $\bar p\in\Uf$ and $p\in \I$.
 We have checked  $\P  \I\cap\hat d\I\hat d^{-1}\subset \I$ which  gives (2)(b).\\
 (c) By (b),  we have $\P \cap\hat   d\I\hat d^{-1}\subset \P \cap \I$ which is contained in $(\M \cap \I)\N $. It gives one inclusion. The other inclusion comes from (a).
 \end{itemize}
\end{proof}

\begin{lemma}\label{lemma:cosets} 
Let $w\in {\W}_{\M^+}(1)$ with lift $\hat w\in \M^+$ and $d\in {}^ \Mf\Wf$ with lift  $\hat d \in \K$.

\begin{enumerate}
\item    Consider a  decomposition into simple cosets $\IM  \hat w \IM=\sqcup_{x }\IM   \hat wx$ for $x \in \hat w^{-1}\IM \hat w\cap \IM\backslash \IM$.
Then 
$$ \I  \hat w \I=\sqcup_{x}\I \hat wx \I_\Nm.$$ Furthermore,    we have $\P\I \hat w \I=\P\I$. For $u\in \I_\Nm$ and $x\in \IM$, we have $u\in  \P\I   \hat wx$ if and only if  $\I  \hat wx=\I \hat wxu$.

\item  Consider a decomposition into simple cosets  $\I \hat  d\I=\sqcup_{y}\I\hat  dy$ for  $y\in \I\cap \hat  d^{-1}\I \hat  d\backslash \I$.
Then  $$\P\hat  d\I= \P\I\hat  d \I=\sqcup_{y}\P\I\hat  dy.$$

\item We have $ \I  \hat w\hat  d  \I = \I   \hat w  \I  \hat  d  \I = \sqcup_{x, u_{x}, y }  \I   \hat w x  u_x {\hat  d } y$ for 
$$x\in \hat w^{-1} \IM  \hat w\cap  \IM \backslash  \IM, \ u_{x}\in (\hat wx)^{-1} \I  \hat wx\cap \I_\Nm\backslash \I_\Nm, \ y\in \I\cap \hat d^{-1}\I \hat d\backslash \I.$$

\end{enumerate}

\end{lemma}

\begin{proof} (1) The first fact  is proved in \cite[II.4]{Selecta}, but we give a proof for the convenience of the reader.  We have $\I \hat w \I=\I \hat w \IM\I_\Nm$ because $\I \hat w \I=\I \hat w\I_\N \IM \I_\Nm$ and $ \hat w \I_\N  \hat w^{-1}\subset \I_\N $ since $w$ is $\M$-positive. The    decomposition of  $ \I  \hat w \I$ is disjoint because for $x\in \I_\M$ we have  $\I \hat wx \I_\Nm= \I_\N  \I_\M  \hat  wx  \I_\Nm$ (since $w$ is $\M$-positive and $x$ normalizes $\I_\Nm$) and the decomposition in $\N\M\Nm$ is unique. We have  $\P\I  \hat  w \I=\P\I$ because   
$\P\I  \hat w \I=\P \I_\Nm \hat  w \I=\P \hat  w^{-1} \I_\Nm  \hat w \I$ and $ \hat w^{-1} \I_\Nm  \hat w\subset \I_\Nm$ since $w$ is $\M$-positive.
The last fact is drawn from 
 \cite[Proposition 7 p.77]{SS91} and we recall the argument.  
Let  $x\in \IM$. Note first that $1\in \P \I \hat w x=\P \hat w^{-1}\I_\Nm \hat w x $ since  $x\in \P$. So, given $u\in \I_\Nm$,  if $\I \hat wx=\I \hat wxu$ then $u\in \P\I \hat wx$.
Now let $u\in \I_{\Nm}$  and suppose that $u\in \P\I \hat w x=\P \hat w^{-1}\I_\Nm \hat w x$. Then there is $p\in \P$ and $k\in  \hat w^{-1}\I_\Nm \hat w\subseteq \I_\Nm$ such that $u = p k x=p x (x^{-1}k x)$. Since $x\in \IM$ normalizes $\I_\Nm$, it implies that  $px\in\I_\Nm\cap\P=\{1\}$. Therefore $x u x^{-1}\in  \hat w^{-1}\I_\Nm \hat w$ and $\I \hat wxu=\I \hat w x$.\\
(2)  The first equality $\P\hat  d\I= \P\I\hat  d \I$ is given by Lemma \ref{PI}(2). We have to  prove that for $y\in \I$, if 
 $\P \I\hat  d y\cap  \P \I\hat  d \neq \emptyset$ then  $ \I\hat  d  y=\I \hat  d$.
So let $y\in \I$ and suppose that $\P \I\hat  d y\cap  \P \I\hat  d \neq \emptyset$.  This means that  there is $p\in \P$, $u,v\in \I_\Nm$ such that $p v= u \hat d y\hat d^{-1}$ and therefore $p\in\K\cap \P$. Reducing mod $\K^1$, recalling $\I_\Nm\subset \K^1$, 
and  denoting respectively by $\bar p$, $\bar {\hat d}$  and $\bar y$ the images of $p$, $\hat d$ and $y$ in $\Gf $, we get  $\bar {\hat d}\bar y \bar {\hat d}^{-1}\in \Pf\cap \bar {\hat d}\Uf \bar {\hat d}^{-1}$ which implies $\bar {\hat d}\bar y\bar {\hat d}^{-1}\in \Uf$ by Lemma \eqref{PI}(1)(b). Therefore, $\hat d y\hat  d^{-1}\in \I$ and  $ \I\hat  d  y=\I \hat  d$.\\
(3) The first equality $ \I  \hat w\hat d  \I = \I   \hat w  \I  \hat d \I$ comes from Lemma \ref{lem:addlengths}.
From (1)   we deduce $\I \hat w \I \hat d \I =\cup_{x,y}\I \hat w x \I \hat dy= \cup_{x,y}\I \hat  w x \I_{\Nm} \hat dy=  \cup_{x, u_{x}, y }  \I  \hat w x  u_x {\hat d } y$ for $x,u_{x},y$ ranging over the sets described in the lemma. It remains to prove that this is a disjoint union: it is a disjoint  union  on $y$   by (2) because    $ \I  \hat w x u_x   \subseteq \I \hat w \I \subseteq \P\I $ by (1). Fixing $y$, it is a disjoint union on $x$ by (1). Fixing $y,x$,  the equality $ \I  \hat w x  u  {\hat d } y= \I  \hat w x  u' {\hat d } y$  
for $u,u'\in \I_\Nm$, is equivalent to $u'\in    (\hat w x)^{-1} \I  \hat w x  u$.
 \end{proof}

\subsection{Pro-$p$ Hecke algebras and universal representations}

\subsubsection{Presentation of the pro-$p$ Iwahori Hecke algebra of $\G$\label{prophecke}}
Let  $R$ be a commutative ring. The universal $R$-representation  $\XX  := \ind_\I^\Gp(1) $ of $G$  is the compact induction 
of the trivial character of $\I$ with values in $\k$. We see it as the space of $R$-valued functions with compact support in $\I\backslash \Gp
$, endowed with the action of $\Gp$ by right translation. The  pro-$p$ Iwahori Hecke $R$-algebra of $G$ is the  $\k$-algebra of the $\Gp$-equivariant endomorphisms of $\XX$. It  will be denoted by $\Hh$. It is naturally isomorphic  to the convolution algebra  $\k[\I \backslash \G/\I]$ of the $\I$-biinvariant functions with compact support in $\G$, with product  defined by 
$$  f \star f': \Gp\rightarrow \k, g\mapsto \sum_{u\in \I\backslash \Gp} f(gu^{-1})f'(u)$$ for $f,f'\in \k[\I \backslash \G/\I]$. We will identify $\Hh$ and $\k[\I \backslash \G/\I]$ without further notice.
 For $g\in \G$, we  denote by $\tau_g\in \Hh$ the characteristic function of $\I g \I$. For $w\in\W(1)$, the double coset $\I \hat w \I$ does not depend on the chosen lift $\hat w\in \norm_\G$ and we simply write $\tau_w$ for the corresponding element of $\H$.  
The set of all $\tau_{w}$ for $w\in  \W(1)$   is a basis for $\Hh$.   

For $s\in \SS_{aff}$, let   $n_s\in \mathscr N_G$ be an admissible lift of $s$ (see \eqref{defiadmi}) and   $\Zf'_{s}$ the subgroup of  $\Zf$ defined in Notation \ref{qsnot}.

The product in $\Hh$ is given by the following braid and, respectively, quadratic relations  \cite[thm.2.2]{Vig1}:   
\begin{equation} \label{braid} \textrm{$\t_{ww'}=\t_w \t_{w'}$ for $w, w'\in \W(1)$ satisfying $\ell(ww')=\ell(w)+ \ell(w')$,}\end{equation}
\begin{equation} \label{Q}   \textrm{
$\t_{n_s}^2=q_s  \t_{n_s^2} + c_{n_s}  \t_{n_s}$ for $s\in \SS_{aff}$, where $c_{n_s}=\sum_{z\in \Zf'_s }c_{n_s}(z) \tau_z$, \ }
  \end{equation}
 for positive integers $c_{n_s}(z) $ satisfying: $ c_{n_s}(z)= c_{n_s}(z^{-1})= c_{n_s}(z xs(x)^{-1})$ if $x\in \Zf$ and $z\in \Zf'_{s}$ (the precise definition of these integers is given in \cite[step 2 of Prop. 4.4]{Vig1}. If we set $c_{n_s} (z)=0$ for $z\in \Zf \setminus  \Zf'_{s}$, the latter relations remain valid for all $z\in \Zf$. Furthermore,  $c_{n_s}  \t_{n_s}=\t_{n_s}c_{n_s} $ and the sum over $z\in \Zf $ of all $c_{n_s}(z)$  is equal to $q_s-1$. 
 We identify the positive integers $q_s$  and $ c_{n_s}(z)$  with their natural image in $R$. \\
If $R$ has characteristic $p$,  then in $\H$ we have  $c_{n_s}= c_s$ where  
\begin{equation}c_s:= (q_s-1) |\Z'_{s}|^{-1}\sum_{z\in \Zf'_{s}} \tau_z\:\:\in\H.\label{simplify}\end{equation}

\begin{exa}  Following up the  example at the end of  \S\ref{subsec:lifts}, suppose $\mathbf G$ is $\F$-split. Then we have,   $q_s=q $ and $c_s=\sum_{z\in  \Tf'_{s}} z$ for any $s\in S_{aff}$. 
\end{exa}

\begin{rema}\label{Qgeneral} In  $\Hh$, the quadratic relation \eqref{Q} satisfied by  $\t_{n_s}$  extends to $\t_{\hat s}$ for an arbitrary lift $\hat s\in  \mathscr N_G$ of $ s\in \SS_{aff}$. The  quadratic relation satisfied by  $\t_{\hat s}$  is $\t_{\hat s}^2=q_s  \t_{\hat s ^2} + c_{\hat s}  \t_{\hat s}$ where $c_{\hat s} :=c_{n_s}\t_t $ and  $t:= n_s^{-1}\hat s  \in \Zf$, because
$$\t_{\hat s}^2  =\t_{n_s}\t_{t}\t_{n_s}\t_{t}=\t_{n_s}^2\t_{n_s^{-1} t n_s}\t_{t}=(q_s  \t_{n_s^2} + c_{n_s}  \t_{n_s})\t_{n_s^{-1} t n_s t}=q_s  \t_{\hat s ^2}+ c_{n_s}  \t_t \t_{\hat s}.$$
     When   $s_1, s_2\in S_{aff}$  admit  lifts $\tilde s_1, \tilde s_2 $ in $\W(1)$ which are conjugate $\tilde s_2=w\tilde s_1 w^{-1}$ by  $w\in \W(1)$, then for  arbitrary lifts  $\hat s_1,\hat s_2\in  \mathscr N_G$ of  $\tilde s_1, \tilde s_2$ we have $ c_{\hat s_2}=wc_{\hat s_1}w^{-1}$ (we note  $w (\sum _{z\in \mathds Z} c(z) \tau_z) w^{-1}:= \sum _{z\in \mathds Z} c(z) \tau_{wzw^{-1}}$ for $c(z)\in \mathbb Z$).  This follows from the equivalence of   the properties (A) and (B) in  \cite[Thm 4.7]{Vig1} applied to $R=\mathbb Z$. 
\end{rema}
 
\begin{rema}\label{inverse} Suppose that $p$ is invertible in $R$. The elements $\t_{w}$ for $w\in \W(1)$ are  invertible in  $\Hh$.
For $s\in S_{aff}$ the element $q_s \tau_{n_s^2}$  is indeed invertible with inverse $q_s^{-1} \tau_{n_s^{-2}}$
because $q_s$ is a power of $p$ (Notation \ref{qsnot}) and  $n_s^2\in 
\Zf$.  The  inverse of $\tau_{n_s}$  in $\H$ is $q_s^{-1} \tau_{n_s^{-2}}(\tau_{n_s}-c_{n_s} )$. \end{rema}

\begin{notation}\label{notationetoile}
 By \cite[Lemma 4.12, Prop. 4.13]{Vig1},  there is a   unique family  $(\t^*_w)_{w\in\W(1)}$  in $\Hh$   satisfying the relations   
\begin{itemize}
\item $\t^*_{w_1}\t^*_{w_2}=\t^*_{w_1w_2}$ if $\ell(w_1)+\ell(w_2)=\ell(w_1w_2)$,
\item $\t_u=\t^*_u$ if   $\ell(u)=0$,
\item $\t^*_{n_s}:=\t_{n_s}- c_{n_s}$ if  $s\in \SS_{aff}$.
\end{itemize}
The set of all $\t^*_w $ for $w\in  \W(1)$ form another basis  of  $\Hh$.
\end{notation}

 \subsubsection{Pro-$p$ Iwahori Hecke algebra attached to a Levi subgroup\label{LeviHecke}}  Let $J\subset \Delta$. We consider the associated standard parabolic subgroup $\P$ with Levi decomposition $\P=\M\N$.
Recall that  $\Iw_\M=\Iw\cap \M$ is an Iwahori subgroup in $\M$  with  pro-$p$ Sylow subgroup is $\IM=\I\cap \M$ (notation introduced in \S\ref{posimonoid}, see also \S\ref{standparab}).
The $R$-algebra of  the $\M$-endomorphisms of the universal $R$-representation $\XX_\M=\ind_\IM^\M(1)$  is the pro-$p$ Iwahori Hecke $R$-algebra $\H_\M$ of $\M$.  For $g\in \M$ (resp. $w\in\W_\M(1)$), we denote by  $\tau^\M_g$ (resp. $\tau^\M_w$) the characteristic function of $\IM g\IM$ (resp. $\IM \hat w\IM$). 
 A basis for $\H_\M$ is given by the set of all $\tau^\M_w$ for $w\in\W_\M(1)$.  
 Another basis is given by the set of all $\tau^{\M,*}_w$ for $w\in\W_\M(1)$, where $(\tau_{w}^{\M, *})_{w\in \W_\M(1)}$ is defined for $\Hh_\M$ as it was for $\H$ in Notation \ref{notationetoile}.

Denote by $\H_{\Mp}$  (resp. $\H_\Mm$) the subspace of $\H_\M$ with basis the set of all 
$\tau^\M_w$ for $w\in\W_\Mp(1)$ (resp. $w\in\W_\Mm(1)$) as defined in Remark \ref{lemma:Mp}.  The algebra $\H_\M$ does not inject in $\H$ in general, but $\H_\Mp$ (resp.  $\H_\Mm$) does:
the linear  maps 
 \begin{equation}\label{incluheckeposi}\H_{\M}\overset{\theta}\longrightarrow \H, \quad \tau^\M_m\longmapsto \tau_m   \ (m\in \M),  \quad \H_{\M}\overset{\theta^*} \longrightarrow\H, \quad  \tau_m^{\M,*}\longmapsto \tau_m^*   \ (m\in \M),\end{equation}
 restricted to $\H_{\Mp}$ or to  $\H_{\Mm}$ respect the products, and identify  $\H_\Mp$  and $\H_\Mm$   with  four subalgebras of $\H$.  The algebra $\HM$ is a localization of $\H_{\Mp}$  (resp. $\H_\Mm$) at a central element   (see \cite[II.4]{Selecta}, \cite[(6.12)]{BK}, \cite[Thm. 1.4]{Vig5}).
\begin{rema}\label{extension} When  $p$ is invertible in $R$,   the  homomorphisms $ \theta|_{\HMp}$, $ \theta|_{\HMm}$,   $ \theta^*|_{\HMp}  $ and   $ \theta^*|_{\HMm}  $ extend uniquely to four embeddings from $\H_{\M}$ into $\H$, denoted by $\theta^+,\theta^-,  \theta^{*+}$ and by $ \theta^{*-}$, determined by the formula
$$\theta^+(\tau^\M_m)=\tau_a^{-n}\tau_{a^nm}  ,\ \theta^-(\tau^\M_m)=\tau_{a^{-1}}^{-n}\tau_{a^{-n}m} ,\ \theta^{*+}(\tau^{\M,*}_m)=(\tau^*_a)^{-n}\tau^*_{a^nm}  ,\ \theta^{*-}(\tau^{\M,*}_m)=(\tau^*_{a^{-1}})^{-n}\tau^*_{a^{-n}m} $$
for  $m\in \M$, where $a\in C_\M$ is a strictly $\M$-positive  element (recall that $\mathbf{ C_M}$ is the connected center of $\mathbf M$):  $\cap _{n\in \mathbb N}a^n \I_{\N} a^{-n}=\{1\}$ and $\cap _{n\in \mathbb N} a^{-n} \I_{{\Nm}} a^n=\{1\}$), $ n\in \mathbb N$ such that $a^n m\in \Mp$ and $a^{-n} m\in \Mm$ (in particular $a$ is $\M$-positive). 
  This follows from  Remark \ref{inverse} and   from  
\cite[II.6)]{Selecta}. For the convenience of the reader, we give the proof for $ \theta^*|_{\HMp}$. The arguments are the same for the other homomorphisms. 

1. \ $\theta^{*+}(\tau^{\M,*}_m)$ is well defined: $(\tau^*_a)^{-n}\tau^*_{a^nm}$ does not depend on the choice  of the pair  $(a,n)$. We show
$(\tau^*_a)^{-n}\tau^*_{a^nm}=(\tau^*_b)^{-r}\tau^*_{b^rm}$ for any other  pair $ (b,r)$.  Using  that 
for $x,y\in \Mp$ such that  $ \tau^{\M,*}_x$ and $\tau^{\M,*}_y $ commute,   the elements $\tau^*_x$ and $\tau^*_y$ commute  as $\theta^*|_{\HMp}$ respects the product, the equality is equivalent to 
$(\tau^*_b)^{r}\tau^*_{a^nm}=(\tau^*_a)^{n} \tau^*_{b^rm}$. It is also equivalent to 
$(\tau^{\M,*}_b)^{r}\tau^{\M,*}_{a^nm}=(\tau^{\M,*}_a)^{n} \tau^{\M,*}_{b^rm}$ because $\theta^*|_{\HMp}$ is injective and respects the product. Then it is equivalent to $ \tau^{\M,*}_{b^ra^nm}= \tau^{\M,*}_{a^nb^rm}$
because the image of $a,b$ in $\W_\M$ have length $0$. This equality  is true because  $a$ and $b$ commute.

2. \ $\theta^{*+}$ respects the product: $\theta^{*+}(\tau^{\M,*}_m)\theta^{*+}(\tau^{\M,*}_{m'})=\theta^{*+}(\tau^{\M,*}_m \tau^{\M,*}_{m'})$ for $m,m'\in \M$.  We choose the same pair $(a,n)$ for $m$ and $m'$. Using  the  arguments given in 1. we show   that $(\tau^*_a)^{-n}$  and $\tau^*_{a^nm} $ commute:
$(\tau^*_a)^{-n} \tau^*_{a^nm}= \tau^*_{a^nm}(\tau^*_a)^{-n}\Leftrightarrow  (\tau^*_a)^{n}\tau^*_{a^nm}=\tau^*_{a^nm}(\tau^*_a)^{n} \Leftrightarrow  (\tau^{\M,*}_a)^{n}\tau^{\M,*}_{a^nm}=\tau^{\M,*}_{a^nm}(\tau^{\M,*}_a)^{n} \Leftrightarrow  
 \tau^{\M,*}_{a^{2n}m}=\tau^{\M,*}_{a^nm _a^n}$, which is true because $a$ and $m$  commute. So $(\tau^*_a)^{-n}\tau^*_{a^nm}(\tau^*_a)^{-n}\tau^*_{a^nm'}= (\tau^*_a)^{-2n}\tau^*_{a^nm}\tau^*_{a^nm'}=(\tau^*_a)^{-2n}\tau^*_{a^nm a^nm'}=(\tau^*_a)^{-2n}\tau^*_{a^{2n}m m'}$.
  \end{rema}
 
 The modulus  $\delta_\P$ of the parabolic group $\P=\M\N$  is   the generalized index   $\delta_P(x)= [x\I_\P x^{-1}: \I_\P]$ for $x\in \P$ and $\I_\P=\I \cap \P$ \cite[I.2.6]{Viglivre}.  
The reductive $p$-adic group $\M$ is unimodular \cite[I.2.7 a)]{Viglivre},  hence $$\delta_\P(x)= [x\I_\N x^{-1}: \I_\N]\quad (x\in \P)$$ where $\I_\N=\I \cap \N$.
When $p$ is invertible in $R$, the modulus of $\P$ over $R$ is the character $x\to  \delta_\P(x) 1_R $ of $\P$ with values in $R$.

\begin{lemma} \label{lemma:thetacompar} When $p$ is invertible in $R$, we have
 $\theta^{*+}(\tau^{\M}_m) =\theta^-(\tau^{\M}_m)\delta_{\P }(m)$ for all $m\in \M$.\end{lemma}
\begin{proof}  As $\theta^{*+}$ respects the product, it  respects  the inverse:  
 $\theta^{*+}((\tau^{\M,*}_m)^{-1})=((\tau^*_a)^{-n}\tau^*_{a^nm})^{-1}=(\tau^*_{a^nm})^{-1} (\tau^*_a)^{n}$ for all $m\in \M $ and $a\in C_\M$  a strictly $\M$-positive  element such that $a^n m\in \Mp$. Let
  $\mu=m^{-1}$.
 As $(\tau^*_w)^{-1}= \tau_{w^{-1}} q_w$ and $q_w=q_{w^{-1}}$,  we obtain  
  $\theta^{*+}(\tau^{\M}_\mu)q_{M,\mu} = \tau_{\mu a^{-n}}q_{\mu a^{-n}} (\tau_{a^{-1}} q_{a^{-1}})^{-n}$.  This is valid for all $\mu \in \M$. We replace $\mu$ by $m$ to get   $\theta^{*+}(\tau^{\M}_m) = \tau_{m  a^{-n}}(\tau_{a^{-1}})^{-n} q_{m a^{-n}} ( q_{a^{-1}})^{-n}q_{\M,m}^{-1}$ 
  for all $m \in \M$  and   $m a^{-n}\in \Mm$ as 
  the inverse map sends $\M^+$ onto $\M^-$. We have $\tau^\M_m= \tau^\M_{m  a^{-n}}\tau^\M_{a^n}=\tau^\M_{m  a^{-n}}(\tau^\M_{a^{-1}})^{-n}$ and $\theta^-$ respects the product: $\theta^-(\tau^{\M}_m) =\theta^-(\tau^{\M}_{m  a^{-n}}) \theta^-(\tau^\M_{a^{-1}})^{-n}=\tau _{m  a^{-n}} (\tau _{a^{-1}})^{-n}$.  We deduce that $\theta^{*+}(\tau^{\M}_m) =\theta^-(\tau^{\M}_m)\delta( m)$ where 
  $ \delta(m)=q_{m a^{-n}} ( q_{a^{-1}})^{-n}q_{\M,m}^{-1}$ for $m\in \M, \: a\in C_\M $  strictly $\M$-positive, $m a^{-n}\in \Mm$.
 As the map $\theta^-$ is injective,  $\delta( m)$ is well defined. 

 It remains to show  $\delta =\delta_\P $ on $\M$.  It is   well known that $\delta_\P$ respects the product on $\M$  \cite[I.2.6]{Viglivre}.  The restriction of $\delta$ to the monoid $\M^-$ is $\delta(w)=q_w q_{M,w}^-1$. It  respects the product by applying \cite[Lemma 2.7]{Vig5}. Furthermore, we have   $\delta(m)=\delta(m a^{-n})\delta(a^{n})$ for $m\in \M$, $a\in C_\M $  strictly $\M$-positive  and $m a^{-n}\in \Mm$ (this is  because $\ell_\M(a)=1$ hence  $\delta(a^{n})= ( q_{a^{-1}})^{-n}$ and   $ q_{\M,m a^{-n}}= q_{\M,m  }$). Therefore $\delta$ respects the product on $\M$.
The two homomorphisms  $\delta$ and $\delta_\P$ from $\M$ to $R$ are  trivial on $\K\cap \M$;  as $\M=(\K\cap \M)\Zp (\K\cap \M)$ by the Bruhat decomposition and $\Zp=\cup_{n\in \mathbb N} (\Zp\cap \M^-)a^{n}$ for $a\in C_\M$ strictly $\M$-positive,   $ \delta$  and  $\delta_{\P} $  are equal on   $ \M$ if   they are equal on    $ \Zp\cap \M^-$.  Let  $z\in  \Zp\cap \M^-$. We have   $q_z=  [\I z \I:\I] =[\I z \I_\P: \I]=[  \I_\P z \I_\P:\I_\P]= [  \I_\M z \I_\M:\I_\M] [  \I_\N z \I_\N:\I_\N]= q_{\M,z} [  \I_\N:(\I_\N\cap z^{-1} \I_\N z)]= q_{\M,z} \delta_\P(z)$. Hence 
  $\delta(z)=q_z q_{\M,z}^{-1}  =\delta_\P(z)$ for $z\in  \Zp\cap \M^-$. Therefore $\delta =\delta_\P $ on   $ \M$.
  \end{proof}

 \subsubsection{Iwahori Hecke algebra\label{defiIH}}
 Considering the Iwahori subgroup $\Iw$  instead of its pro-$p$ radical $\I$ in \S\ref{prophecke}, one defines the Iwahori Hecke algebra $\Hh_\Bb$ of the $\G$-endomorphisms of the universal $R$-representation $\XX_\Iw  := \ind_\Iw^\Gp(1) $. 
  It is naturally isomorphic to the convolution algebra  $\k[\Iw \backslash \G/\Iw]$ of the $\Iw$-biinvariant functions with compact support in $\G$. For $g\in \G$, denote by $\uptau_g$ the characteristic function of $\Iw g \Iw$.  
 For $w\in \W$, denote by $\uptau_{w}$ the characteristic function of $\Iw \hat w \Iw$ which does not depend on the choice of a lift $\hat w\in \Np$ 
for $w$.  The set of all $\uptau_{w}$ for $w\in \W$ is a basis for $\Hh_\Iw$.
 The product in $\Hh_\Iw$ is given by the following braid and, respectively, quadratic relations:\begin{equation} \label{braidIw} \textrm{$\uptau_{ww'}=\uptau_w \uptau_{w'}$ for $w, w'\in \W$ satisfying $\ell(ww')=\ell(w)+ \ell(w')$.}\end{equation}
\begin{equation} \label{QIw}
\uptau_{s}^2= q_s\uptau_1+(s-1)\uptau_{n_s}\textrm{ for any $s\in S_{aff}$.}\end{equation}These relations are well known by \cite{IM}, \cite{Bo}, \cite{Lu} under certain hypotheses on $\mathbf G$. It has been checked by Vign\'eras \cite[Thm 2.1 and Rmk  2.6]{Vig1} that they hold for a general connected $ \F$-group $\mathbf G$. The key result for the quadratic relation is the following (\cite[Prop. 2.8]{IM},  \cite[5.2.12]{BTII}, \cite[Theorem 33]{Vig1}): for $s\in S_{aff}$, we have \begin{equation}\Iw s \Iw s\Iw=\Iw s \Iw\sqcup \Iw.\label{keyquad}\end{equation}
Note that conversely, once we admit the quadratic relations, this identity can be deduced from them using the definition of the convolution product on  $\k[\Iw \backslash \G/\Iw]$  (choose for example $\k=\mathbb C$). 
\begin{lemma} For $s\in S_{aff}$, we have $\Iw s \Iw s\Iw=\Iw s \Iw s\sqcup \Iw s.$
\label{keyquad'}
\end{lemma}
 
\begin{proof} The double coset $\Iw s \Iw$ decomposes into $q_s$ simple right cosets \cite[Prop. 3.38]{Vig1}.
The set $\Iw s \Iw s$ is the disjoint union of $q_s$ right cosets contained in $\Iw s \Iw s \Iw$. It does not contain $\Iw s$ since $\Iw s \Iw$ and $\Iw$ are disjoint. Therefore,
and since $\Iw s \Iw s \Iw$ is the disjoint union of $q_s+1$ simple cosets (by \eqref{keyquad}), we have  
  $\Iw s \Iw s=\Iw s \Iw s \Iw -\Iw s$.

\end{proof}
\begin{rema}  \label{chartriv}For $w\in \W_{aff}$, denote by $v_s(w)$ the multiplicity of $s\in S_{aff}$ in a reduced decomposition of $w$ and let $$q_w:= \prod_{s\in S_{aff}} q_s^{v_s(w)}.$$ Set $q_w:= q_{w_{aff}}$ for any $w\in \Omega w_{aff} \Omega$ and $w_{aff}\in \W_{aff}$. 
 The element $q_w$ for $w\in \W$ is well defined and does not depend on a reduced expression of $w\in \W$ (\cite[Chap. 1]{Vign}).  The formulas   
 $${\rm Triv}_{\H_\Bb}(\uptau_w):=q_w. 1_R, \quad {\rm Sign}_{\H_\Bb}(\uptau_w):=(-1)^{\ell(w)}.1_R$$ define respectively  the trivial character  ${\rm Triv}_{\H_\Bb}$ and the sign  character ${\rm Sign}_{\H_\Bb}$ of $\Hh_\Bb$ with values in $R$.\end{rema}
\begin{rema}\begin{enumerate}
\item   The Iwahori Hecke $R$-algebra $\H_\Iw$  is a quotient of the pro-$p$ Iwahori Hecke $R$-algebra $\H$ by the augmentation map sending to $1_R$ the elements of the form $\tau_t$  for   $t\in \Zp^0/\Zp_1= \Zf$. In particular, ${\rm Triv}_{\H_\Bb}$ and  ${\rm Sign}_{\H_\Bb}$ extend to characters of $\H$ that we denote  respectively by ${\rm Triv}_{\H}$ and  ${\rm Sign}_{\H}$ and call the trivial and sign characters of $\H$ (compare with  \cite[Def. 2.7]{Vig3}).

\item Suppose that  
$\vert \Zp^0/\Zp^1\vert=\vert \Zf\vert$ is invertible in the commutative ring $R$, then one can define the following central idempotent of $\H$ (see \cite[\S 1.3]{Vigprop} in the split case):
$${\varepsilon}_{\bf 1}=\frac{1}{\vert \Zf\vert}\sum_{t\in \Zf}\tau_t .$$ We have $ \varepsilon_{\bf 1}\XX\cong \XX_\Iw$ and the algebra  $\varepsilon_{\bf 1}\H$ with unit $\varepsilon_{\bf 1}$ is isomorphic to $\H_\Iw$ (see \cite[\S 2.3.3]{Compa} in the split case, but the proof is general). 
\end{enumerate}
\label{rema:idempo}
\end{rema}
\subsubsection{Trivial representation of $\G$ and trivial character of the pro-$p$ Iwahori Hecke algebra}
 Let $R$ be a commutative ring. Consider the trivial character ${\rm Triv}_{\H}$ of $\H$  (Remark \ref{rema:idempo}(1)). It is defined by $\tau_{\tilde w}\longmapsto q_w. 1_R,$ for all $\tilde w\in \W(1)$ with projection $w\in \W$.
 Let $\Triv_\G:\G\rightarrow R^*$ denotes  the trivial character of $\G$. As a right $\H$-module, we have
 $(\Triv_\G)^\I\cong \Triv _\H$. This is because, for all $\tilde w\in \W(1)$  as above, we have   $q_w=[\Bb w \Bb, \Bb]=[\I \tilde w\I, \I]$ (\cite[Cor. 3.30]{Vig1}) (recall that for $X=\Bb$ or $\I$, we denote  by $[XwX,X]$   the number of left cosets mod $X$ in $XwX$. It is equal to the number of right cosets mod $X$  in $XwX$). 
 
 \begin{lemma}\label{lemmatriv} Let $R$ be a commutative ring with  $q_s+1$ invertible in $R$ for all $s\in S_{aff}$. We have an isomorphism of representations of 
 $\G$
\begin{equation}{\rm Triv}_\H\otimes_{\H} \XX\cong {\rm Triv}_\G .\label{isotriv}\end{equation}

\end{lemma}

\begin{proof}
  By Remark \ref{rema:idempo}(1) we have $\H_\Bb= R \otimes_{R[\Zf]}\H$ and $\XX_\Bb= R \otimes_{R[\Zf]}\XX$, for the augmentation map $R[\Zf] \to R$ sending the elements $t\in \Zf$ to $1_R$
and for the linear map $R[\Zf]\to \H$ sending the elements $t\in \Zf$ to $\tau_t$. Therefore, 
we have a natural isomorphism of $\G$-representations ${\rm Triv}_\H\otimes_{\H}\XX\cong   {\rm Triv}_{\H_\Bb}\otimes_{\H_\Bb}\XX_\Bb$ (without any hypothesis on the $q_s $). We show here that ${\rm Triv}_{\H_\Bb}\otimes_{\H_\Bb} \XX_\Bb\cong {\rm Triv}_\G$ when is $q_s+1$ invertible in $R$ for all $s\in S_{aff}$. From the remark before the lemma, one gets immediately that 
the right $\H_\Iw$-module  $({\rm Triv}_{\G})^\Bb$  is  isomorphic to 
${\rm Triv}_{\H_\Bb}$ so by adjunction, ${\rm Triv}_{\G}$ is a quotient of the representation ${\rm Triv}_{\H_\Bb}\otimes_{\H_\Bb} \XX_\Bb$:  the map is given by $\lambda\otimes {\bf 1}_\Bb\mapsto \lambda$ where ${\bf 1}_\Bb$ is the characteristic function of $\Bb$. We want to check that this map is injective and it will follow from the fact that we have a  decomposition of $R$-modules \begin{equation}\label{decomtriv}\XX_\Bb= R{\bf 1}_{\Bb}+ \ker(\triv_{\H_\Bb})\XX_\Bb.\end{equation}
The $R$-module  $\XX_\Bb$ is free and decomposes as the direct sum of all $\XX_\Bb(w)$ for  $w\in \W$ where we denote by $\XX_\Bb(w)$ the $R$-module of all functions with support in $\Bb w \Bb$.

We show by induction on $\ell(w)$ that 
$\XX_\Bb(w)$ is contained in $R{\bf 1}_{\Bb}+ \ker(\triv_{\H_\Bb})\XX_\Bb$. If $\ell(w)=0$ then   ${\rm Triv}_{\H_\Bb}(\uptau_w)=1$ and  $\XX_{\Bb}(w)= R {\bf 1}_{\Bb w}$  is contained in $R{\bf 1}_{\Bb}+ \ker(\triv_{\H_\Bb})\XX_\Bb$ since the characteristic   function 
${\bf 1}_{\Bb w}$ satisfies ${\bf 1}_{\Bb w}= \uptau_{w} {\bf 1}_{\Bb}= {\bf 1}_{\Bb}+(\uptau_w-1) {\bf 1}_{\Bb}\in R{\bf 1}_{\Bb}+ \ker(\triv_{\H_\Bb})\XX_\Bb$. 
Now suppose $\ell(w)\geq 0$ and let  $s\in S_{aff}$ such that $\ell(s w)=\ell(w)+1$.   Note  that $\XX_\Bb(s w)$ is generated as a representation of $\Bb$ by the  characteristic function $f:={\bf 1}_{\Bb  n_s \hat w}=\hat w^{-1}{\bf 1}_{\Bb n_s}$, where $\hat w$ is a chosen  lift for $w$ in $\norm_\G$ and $n_s$ is defined as in Notation \ref{qsnot}. Therefore, to prove that 
$\XX_\Bb(s w)$ is contained in $R{\bf 1}_{\Bb}+ \ker(\triv_{\H_\Bb})\XX_\Bb$, 
it is enough to consider the case of   $f$.
 Since  $\uptau_{n_s}$ corresponds to the double coset $\Bb n_s \Bb$ and commutes to the action of $n_s^{-1}\in\G$, we have $\uptau_{n_s}  {\bf 1}_{\Bb {n_s}}={\bf 1}_{\Bb {n_s} \Bb n_s}$ so, by Lemma \ref{keyquad'}, we have
$  \uptau_{n_s}  {\bf 1}_{\Bb {n_s}}={\bf 1}_{\Bb {n_s} \Bb n_s \Bb}-{\bf 1}_{\Bb {n_s}}$. On the other hand,
$  \uptau_{n_s} {\bf 1}_{\Bb}= {\bf 1}_{\Bb n_s \Bb}=  {\bf 1}_{\Bb {n_s} \Bb n_s \Bb}- {\bf 1}_{\Bb}$ by \eqref{keyquad}. We deduce \begin{equation}\label{sstar}{\bf 1}_{\Bb n_s}-{\bf 1}_{\Bb}=\uptau_{n_s}( {\bf 1}_{\Bb}-{\bf 1}_{\Bb {n_s}}).\end{equation}
By translation by $\hat w^{-1}$ which commutes with the action of $\H$, we get:
\begin{align*} f-{\bf 1}_{\Bb\hat w}= 
\uptau_{n_s}({\bf 1}_{\Bb\hat w}- f)= (\uptau_{n_s}-q_{s})({\bf 1}_{\Bb\hat w}- f)- q_{s}( f-{\bf 1}_{\Bb\hat w}). \end{align*}
If  $1+q_{s}$ is invertible in $R$, this means that $f$ lies in ${\bf 1}_{\Bb\hat w}+\ker(\triv_{\H_\Bb})\XX_\Bb$. Conclude by induction.
This proves the decomposition \eqref{decomtriv} which is enough to prove that the natural surjective map ${\rm Triv}_{\H_\Bb}\otimes_{\H_\Bb} \XX_\Bb\rightarrow {\rm Triv}_{\G}$ is injective. Note in passing that from the surjectivity of this map we also get that the decomposition \eqref{decomtriv} is a direct decomposition.  \end{proof}

 \subsection{Finite Hecke algebras and universal representations\label{finiteHecke}} 
 
 \subsubsection{Definitions}
 Let $J\subseteq\Pi$. Let
 $\P$ be  the  associated standard parabolic subgroup of $\Gp$, and $\Pf$ with Levi decomposition $\Pf=\Mf \, \Nf$ the corresponding parabolic subgroup of $\Gf$.   
Let  $\Pp$ (resp. $\EuScript N$)  be the preimage in $\K$ of $\Pf$ (resp. $\Nf$). The group $\Pp$ is  a parahoric subgroup of $\Gp$ containing $\I$.  We have $J=\Pi$ if and only if $\Pp=\K$. Extending  by zero functions  on  $\Pp$ to $\K$, and respectively  on $\K$  to $\Gp$,  induces   embeddings which are respectively $\Pp$ and $\K$ equivariant:
\begin{equation}\label{incluniv}
\ind^{\Pp}_\I(1)\hookrightarrow \ind^{\K}_\I(1) \hookrightarrow   \XX \ .
\end{equation}   Passing to $\I$-invariants induces  naturally embeddings of the  convolutions algebras
\begin{equation}\label{inclhecke}
(\ind^{\Pp}_\I(1))^\I\hookrightarrow (\ind^{\K}_\I(1))^\I \hookrightarrow   \XX^\I \ .
\end{equation} 

Let $\Uf_\Mf:=\Mf\cap\Uf$ (it was previously denoted by $\Uf_J$ (see  \eqref{ssBNpJ}) . The universal $R$-representation  $$    \Xf_{\Mf} := \ind^{\Mf}_{\Uf_\Mf}(1), $$ of  $\Mf$  is the   induction 
of the trivial character of $\Uf_{\Mf}$ with values in $\k$.   The associated finite unipotent Hecke $R$-algebra is
$$ \Hf_\Mf := \End_{k[\Mf]}(\mathbf{\Xf}_\Mf)\cong [\ind^{\Mf}_{\Uf_\Mf}(1)]^{\Uf_\Mf}.$$ 
It has basis the set of all characteristic functions of the double cosets $\UMf n \UMf$ for $n\in \norm_{\Mf}$ (see the decomposition after  \eqref{ssBNpJ}).
When $\J=\Pi$ and  therefore $\Mf=\Gf$, we omit the subscript and simply write $\Xf$ and $\Hf$ instead of $\Xf_\Gf$ and $\Hf_\Gf$ respectively.
\begin{prop} 
 As a left (resp. right) $\HMf$-module, $\Hf$ is  free with basis the set of all  $\t_{\tilde d}$ (resp. $\t_{\tilde d^{-1}}$) for  $d\in \dist$, where $\tilde d$ denotes a lift for $d$ in $\norm_\Gf$. 
\label{liberte}
\end{prop}
\begin{proof} This is a consequence of  the relation $\ell(w)+\ell(d)=\ell(wd) $ for  $w\in \Wf_\M, d\in \dist$ (Lemma \ref{lem:addlengths})  and of the braid relations \eqref{braid}.
\end{proof}

 Since $\Pp/\I\simeq  \Pf/\Uf=\Mf/\UMf$,  $\I\backslash\Pp/\I\simeq  \UMf\backslash \Mf/\UMf$, and $\Pp/\EuScript N\simeq \Mf$,
the representation $\Xf_\Mf$  can also be seen as a representation of $\Pp$ trivial on $\EuScript N$ and as such it is isomorphic to $ \ind^{\Pp}_\I(1)$. Likewise $\Xf$ is isomorphic to $ \ind^{\K}_\I(1)$ when seen as a representation of $\K$ trivial on $\K^1$. Via these isomorphisms, \eqref{incluniv}  gives embedddings
 \begin{equation}\Xf_\Mf\hookrightarrow \Xf\hookrightarrow \XX,\end{equation} where the first one is the  $\Mf$-equivariant  morphism sending ${\bf 1}_\UMf$ onto ${\bf 1}_{\Uf}$.  The embeddings of algebras \eqref{inclhecke} can be reformulated  as
embeddings \begin{equation} \label{incluhecke2}\HMf\hookrightarrow \Hf\hookrightarrow \H .\end{equation}
The first one sends, for $m\in \Mf$, the characteristic function of  $\UMf m \UMf$ to $\Uf m \Uf$.  The second one  sends, for $g\in \K$, the characteristic function of  $\Uf (g{\textrm{ mod }} \K^1) \Uf$ to $\I g \I$.
The algebra $\HMf$ can therefore be seen as a subalgebra of $\H_\M$. When $w\in \norm_{\Mf} $ we will  consider $\tau_w^\M$ as an element of $\HMf$ or of $\HM$ depending on the context.  In particular, when $s\in S_J$ (defined in \S\ref{standparab}) and $\nf_s\in\norm_\Mf$ is the image of  an admissible lift $n_s\in \M\cap\K$ in $\Mf$ as   in Notation \ref{notnf} (see also  Remark \ref{compadmissible}) then $\t_{\nf_s}$  satisfies the same quadratic relation \eqref{Q} as $\t_{n_s}$.

Recall that $\M\cap\K$ is contained in the positive (resp. negative) submonoid $\Mp$ (resp. $\Mm$) of $\M$.  The embedding  \eqref{incluhecke2} of $\HMf$ into $\H$ is the restriction of   $\theta$ defined in \eqref{incluheckeposi} to $\HMf$, i.e. the map 
$$\HMf\longrightarrow \H,\:\quad \tau_w^\M,  \: w\in \norm_\Mf \longmapsto \tau_w$$ 
where we identify $\norm_\Mf$ as a subgroup of $\W(1)$ via the isomorphisms \eqref{isomorphismW}.

\medskip

We will often identify $\HMf$ with its image in $\Hf$ or in $\H$.

 \subsubsection{Trace maps and involutions on the finite Hecke algebras\label{sec:frob}}  
 The commutative ring $R$  is arbitrary unless otherwise mentioned. A $R$-linear form $\Hf_\Mf\rightarrow R$ is called non degenerate if its kernel  does not contain any non zero left or right ideal.  In the next proposition, we define such a non degenerate form $\delta_\Mf$ and associate to it an automorphism $\iota_\Mf$ of $\Hf_\Mf$  such that 
 $\delta_\Mf(ab)= \delta_\Mf(\iota_\Mf(b) a)$. 
  When $\Mf=\Gf$ we will omit the subscript and simply write $\delta$ and $\iota$.
 
 \begin{prop} 
\begin{enumerate} 

\item   Pick $\widetilde {{\longw }_{\Mf}}\in \norm_{\Mf}$ a lift for the longest element  ${\longw }_{\Mf} $ in $\Wf_\Mf\cong \norm_{\Mf}/\Zf$. 
Let $\delta_\Mf: \HMf\rightarrow \k$ be the linear map sending 
$\t_{\widetilde {{\longw }_{\M}}}$ to $1$ and
$\t_{n}$ to zero for all $n\in\norm_{\Mf}-\{\widetilde {{\longw }_{\M}}\}$. Then $\delta_\Mf$ is non degenerate. 
 The automorphism of $\HMf$ defined by
$$\iota_\Mf: \tau_{n}\mapsto  \tau_{\widetilde {{\longw }_{\Mf}} n \widetilde {{{\longw }_{\Mf}}} ^{-1}}\text{ for all $n\in \norm_\Mf$}$$ satisfies $\delta_\Mf(ab)= \delta_\Mf(\iota_\Mf(b) a)$ for any $a,b\in \HMf$.
\item Suppose that  $\k$  is a ring of  characteristic different from $p$. The linear $R$-form $\delta'_\Mf: \HMf\rightarrow \k$  sending 
$\t_1$ to $1$ and
$\t_{w}$ to zero for all $w\in\norm_\Mf-\{1\}$ is non degenerate and  symmetric: it satisfies $\delta'_\Mf(ab)=\delta'_\Mf(ba)$ for any $a,b\in \HMf$.  The corresponding  automorphism $\iota'_\Mf$ of $\HMf$ is the identity map.  

 \end{enumerate}
\label{tracemap}
\end{prop}

We comment on this proposition and give its  proof in the case when $\Mf=\Gf$ so as to avoid the subscripts $\Mf$.

\begin{rema}  \begin{enumerate}
\item   For $t\in \Zf$ and $\tw_1\in  \norm_{\Gf}$ a lift of the longest element  $\longw$ in $\Wf$, 
 $\tw_2:=t \tw_1$ is another  lift of $\longw $. Denote  by $\iota_1$ and $\iota_2$ the corresponding  automorphisms of $\Hf$ as in Prop.\ref{tracemap}(1).  Then for any right $\Hf$-module $\m$,   denote by 
 $\m\iota_1$, resp.  $\m\iota_2$, 
 the space $\m$ endowed with the right action of $\Hf$ given by $(m,h)\mapsto 
m\iota_1(h)$ resp. $(m,h)\mapsto  m\iota_2(h)$. Define analogously, for 
a left $\Hf$-module $\m$,   the twisted modules
 $\iota_1\m$ and $\iota_2\m$. 
The 
map $m\mapsto m\tau_{t^{-1}}$ (resp. $m\mapsto \tau_{t}m$) yields an isomorphism of right $\Hf$-modules 
$\m\iota_1\overset{\simeq}\rightarrow \m\iota_2$ (resp. $\iota_1\m\overset{\simeq}\rightarrow\iota_2\m$).

\item The map $\iota$  of Prop. \ref{tracemap}(1) is an involution if and only if  the chosen  lift $\tw$ for $\longw$ is such that $\tw^2$ is central in $\norm_\Gf$.  Since $\longw^2=1$, we know that $\tw^2\in \Zf$. 
We pick  $\tw:=\nf_\longw$ as in  Notation \ref{notnf} and check that $\iota$ is then an involution. 
 Since $\Zf$ is commutative, it remains to show that $\nf_\longw^2$ commutes with $\nf_s$ for any $s\in S$, and this follows from the formula  \eqref{conjbynw}. 
\end{enumerate}

\label{rema:tracemap}
\end{rema}

\begin{proof}[Proof of Proposition \ref{tracemap}]
(1)  
If $\k$ is a field of characteristic $p$, it is proved in  \cite[Lemma 2.2]{Saw}  that the kernel of $\delta$  as defined in the lemma does not contain any non zero left or right ideal (see also \cite[Prop. 6.11]{CE}, the necessary hypotheses  being satisfied by  \eqref{ssBNpJ}).  One easily checks that the proofs therein work over a ring of arbitrary characteristic.

 The  $\k$-linear map
$\iota$ given in the lemma  is bijective. It remains to check that  $\iota$  respects the product in $\Hf$ and that $\delta(\tau_u\tau_v)=\delta(\iota(\tau_v)\tau_u)$ for all $u,v\in\norm_\Gf$.   First we show that we only  need to do it for one choice of lift in $\norm_\Gf$  for $\longw$.
If  $n$ and $n':=  tn$ with $t\in \Zf$ are two  such lifts, we denote respectively by 
$(\delta, \iota)$  and $(\delta', \iota')$ the pairs attached to them as in the proposition and we suppose that $(\delta, \iota)$  satisfies the properties above. 
The conjugation  by the invertible element $\tau_t$  is an automorphism of the algebra $\Hf$. We have $\delta'= \delta(\tau_{t^{-1}} -)$ and $\iota'=\iota(\tau_{t} -\tau_{t^{-1}})$. It implies that $\iota'$ also respects the product in $\Hf$. Furthermore, for 
$u,v\in\norm_\Gf$, we have $\delta'(\iota'(\tau_v)\tau_u)=\delta(\tau_{t^{-1}}\tau_{tn v {(tn)}^{-1}}\tau_u)= \delta(\iota(\tau_v)\tau_{t^{-1}u})=\delta(\tau_{t^{-1}u}\tau_v)=\delta'(\tau_u \tau_v)$. So the pair $(\delta', \iota')$ also satisfies the properties above.  Consequently, it suffices to write the proof for a specific choice of lift for $\longw$ and we choose    $\nf_{\longw}\in \norm_{\Gf}$ as defined in Notation \ref{notnf}.

With this choice, we first check that $\iota$  respects the product.  Let $u,v\in \norm_\Gf$ and suppose that $\ell(u)+\ell(v)=\ell(uv)$. 
Conjugating by $\nf_\longw$ preserves the length in $\norm_\Gf$ since  ${  {{\longw }}}(\Sigma^+)=\Sigma^-$. Therefore,  we also have  $\ell(\nf_\longw u\nf_\longw^{-1})+\ell(\nf_\longw v\nf_\longw^{-1})=\ell(\nf_\longw uv\nf_\longw^{-1})$ and $\iota(\tau_{u}\tau_v)=\iota(\tau_{uv})=\tau_{\nf_\longw uv\nf_\longw^{-1}}=\tau_{\nf_\longw u\nf_\longw^{-1}}\tau_{\nf_\longw v\nf_\longw^{-1}}=\iota(\tau_u)\iota(\tau_v)$.
Now let $s\in S$ and consider the quadratic relation 
$\t_{\nf_s}^2=q_s \t_{\nf_s^2}+ c_{{n_s}}\t_{\nf_s}$ satisfied by $\t_{\nf_s}$ (see \eqref{Q} and \S\ref{finiteHecke}).  First note that  in fact we have ${  {{\longw }}}(\Pi)=-\Pi$,  therefore
$s'= \longw s \longw^{-1}\in S$ and  $q_s=q_{s'}$ (\cite[\S 2.1]{Vig1}). Furthermore, $\nf_\longw \nf_{s} \nf_\longw^{-1}= \nf_{s'}$ by \eqref{conjbynw} and    $n_\longw \Uf_{s}n_\longw^{-1}  =\Uf_{s'}$, $n_\longw \Uf_{s}^{op} n_\longw^{-1} = \Uf_{s'}^{op}$, $n_\longw \Zf'_s n_\longw^{-1}=\Zf'_{s'}$.   
In the formula   \eqref{Q} we have $c_{n_s}(z)=| \nf_{s}\Uf_{s}   \nf_{s} \cap \Uf_{s} \nf_{s} z\Uf_{s}|$   for $z\in  \Zf'_s$ \cite[Proposition 4.4, (62)]{Vig1}, therefore  $c_{n_{s'}} (\nf_\longw \,.\,\nf_\longw ^{-1})= c_{n_s}(\, .\,)$  as in Remark \ref{Qgeneral}  hence $\iota(c_{n_{s}}) =c_{n_{s'}}$. We have 
$$(\iota(\tau_{\nf_s}))^2  = (\tau_{\nf_{s'}})^2=q_{s'} \t_{\nf_{s'}^2}+ c_{{n_{s'}}}\t_{\nf_{s'}}$$ 
while $$\iota(\t_{\nf_s}^2)=\iota(q_s \t_{\nf_s^2}+ c_{{n_s}}\t_{\nf_s})= q_s  \tau_{\nf_{s'} ^2}+\iota( c_{n_{s}})\iota( \t_{\nf_{s}})=
q_s  \tau_{\nf_{s'} ^2}+c_{n_{s'}} \t_{\nf_{s'}},$$ the second identity in the line above being true because $c_{n_s}$ is a linear combination of $\tau_t$ for $t\in \Zf$ and $\ell(t)=0$. We have proved that $(\iota(\tau_{\nf_s}))^2=\iota(\t_{\nf_s}^2).$ This concludes the proof of the fact that $\iota$ respects the product.\\
 
 Lastly, we show that $\delta(\tau_u\tau_v)=\delta(\iota(\tau_v)\tau_u)$ for all $u,v\in\norm_\Gf$. 
The property  is easy to check  when $\ell(v)=0$.
We now  show it when $v$ has length  $1$ and then proceed by induction on $\ell(v)$. When $v$ has length $1$,   it is enough to treat the case $v=\nf_s$. Then as above, there is $s'\in S$ such that $\nf_{s'}= \nf_\longw n_s  \nf_\longw^{-1}$. First suppose that $u \nf_s= \nf_\longw$.
Then necessarily $\ell(\nf_{\longw})=\ell(u)+1$ and $\tau_u\tau_{\nf_s}=\tau_{\nf_\longw}$ has image $1$ by $\delta $. On the other hand $\nf_{s'} u=\nf_\longw$ with $1+\ell(u)=\ell(\nf_\longw)$, so 
the image of $\tau_{\nf_{s'}}\tau_u=\iota (\tau_{\nf_s})\tau_u$ by $\delta $ is also $1$. \\
Now suppose that $u \nf_s\neq  \nf_\longw$ which implies  ${\nf_{s'}}u\neq \nf_\longw$. 
This means in particular that if $\ell(u\nf_s)=\ell(u)+1$ (resp.  $\ell({\nf_{s'}}u)=1+\ell(u)$), then
$\delta(\t_u\t_{\nf_{s}})=0$ (resp.  $\delta(\t_{{\nf_{s'}}}\t_u)=0)$.
\begin{quote}
 1/ If $\ell(u{\nf_{s}})=\ell(u)+1$ then the only case left to examine is 
 $\ell({\nf_{s'}}u)=\ell(u)-1$. The quadratic relation \eqref{Q} implies that $\tau_{{\nf_{s'}}}\tau_u$ is a $\k$-linear combination of the elements $\tau_w$ for $w\in \norm_\Gf$ with $\ell(w)\leq 	\ell(u)<\ell(u{\nf_{s}})$. Since none of these $w$  is  $ \nf_\longw$ we have $\delta(\tau_{{\nf_{s'}}}\tau_u)=0$. \\
 2/ If $\ell(u{\nf_{s}})=\ell(u)-1$, then  again by \eqref{Q} we know that $\tau_{u}\tau_{\nf_{s}}$ is a $\k$-linear combination of the elements $\tau_w$ for $w\in \norm_\Gf$ with $\ell(w)\leq 	\ell(u)$. 

 2a/ If $u$ does not have maximal length in $\norm_\Gf$, then $\delta(\tau_{u}\tau_{\nf_{s}})=0$ and the only case left to examine is  $\ell({\nf_{s'}}u)=\ell(u)-1$. 
 But then by the same argument about the decomposition of $\tau_{{\nf_{s'}}}\tau_u$, we get  $\delta(\tau_{{\nf_{s'}}}\tau_u)=0$.

 2b/  If  $u$ has maximal length, then $u= \nf_\longw t$ for $t\in \Zf$.

On the one hand, $\tau_u \tau_{\nf_{s}}=\tau_{\nf_\longw}\tau_t \tau_{\nf_s}=\tau_{\nf_\longw}\tau_{\nf_s}\t_{\nf_s^{-1} t\nf_s}= \tau_{\nf_\longw \nf_s^{-1}}\tau_{\nf_s}^2\t_{\nf_s^{-1} t\nf_s }$ is equal to 
$\tau_{\nf_\longw \nf_s^{-1}}(q_s \tau_{\nf_s ^2}+ \t_{\nf_s} c_{n_s}) \t_{\nf_s^{-1} t\nf_s }$ as  $ \t_{\nf_s}$ and $ c_{n_s}$ commute, hence  $\tau_u \tau_{\nf_{s}}=q_s \tau_{\nf_\longw   t\nf_s }+ \tau_{\nf_\longw}     c_{n_s}   \t_{\nf_s^{-1} t\nf_s }.$ So 
$$\delta(\tau_u \tau_{\nf_{s}})=c_{n_s} (\nf_s^{-1} t^{-1}\nf_s)=c_{n_s} (t^{-1})=c_{n_s} (t ) .$$   The first equality above comes from the fact that $c_{n_s}$ and $\tau_{\nf _s}$ commute and the second equality  is recalled after \eqref{Q}.

  On the other hand,     $\iota(\t_{\nf_{s}})\t_u= \t_{{\nf_{s'}}}\t_u=\t_{\nf_{s'}}\t_{ \nf_\longw}\t_{ t}=\t_{\nf_{s'}}^2\t_{\nf_{s'}^{-1} \nf_\longw}\t_{  t}$, so  $\iota(\t_{\nf_{s}})\t_u=(q_{s'} \tau_{\nf_{s'}^2}+ c_{n_{s'}}  \tau_{\nf_{s'}} )\t_{\nf_{s'}^{-1} \nf_\longw}\t_{ t}
 =  q_{s'}  \t_{ \nf_{s'}   \nf_\longw  t}
+   c_{n_{s'}}  \tau_{\nf_\longw } \tau_{  t}
 $. 
 So,   $$\delta(\iota(\t_{\nf_{s}})\t_u)  =c_{n_{s'}} (\nf_\longw t^{-1}  \nf_\longw ^{-1}) = c_{n_{s}}(t^{-1})=
 c_{n_{s}}(t).$$ 
because $\nf_\longw   c_{\nf_{s}} \nf_\longw^{-1}= c_{\nf_{s'}} $ (Remark \ref{Qgeneral}), in other terms $ \sum _{t\in \mathds Z}c_{\nf_{s}}(t) \nf_\longw  t \nf_\longw^{-1}=  \sum _{t\in \mathds Z}c_{\nf_{s'}} (t) t$, 
  and $ \nf_\longw ^2\in \Zf$.

This concludes the proof  of
 $\delta(\iota(\t_v)\t_u)=\delta(\t_u\t_v)$.
 \end{quote}

\medskip

 We now prove part (2) of Prop. \ref{tracemap}.  
Here  $\k$  is a ring of  characteristic different from $p$ and $\delta': \Hf\rightarrow \k$   is the map sending 
$\t_1$ to $1$ and
$\t_{w}$ to zero for all $w\in\norm_\Gf-\{1\}$.
It is proved in \cite[Lemma 2.3]{Saw} that $\delta'$ is symmetric and non degenerate
when $\k$ is a field of characteristic different from $p$. The arguments go through in  the case of
 a ring of characteristic different from $p$.  We recall first the argument of 
 the non degeneracy of $\delta'$.
 The support of an element in $\Hf$ of the form    $a=\sum_{w\in \norm_\Gf}\lambda_w \tau_w$, $\lambda_w\in\k$,
 is defined to be the set of all $w$ such that $\lambda_w\neq 0$ and its height, if $a\neq 0$,  the minimal  length of the elements in  its support.  
If $a$ has height $0$, then there is $t\in \Z$ in its support and $\delta'(a\tau_t^{-1})\neq 0$ so the right ideal generated by $a$ is not contained in the  kernel of $\delta$.
 If $a$ has height $k\geq 1$, then  there is $s\in S$ such that $a\tau_{\nf_s}$ has height $k-1$. 
This is because there is  $w_0\in\norm_\Gf$ in the support  of $a$ with length $k$ and we choose  $s\in  S$ such that   $\ell(w_0 \nf_s)=\ell(w_0)-1$. Then
$a\tau_{\nf_s}$ is the sum 
\begin{itemize} 
\item  of
$\sum_{w,  \ell(w)\geq k+1}\lambda_w \tau_{w}\tau_{\nf_s}$ which has height  $\geq k$  if it is not zero
(use \eqref{braid} and \eqref{Q}), 
\item  of $\sum_{w, \ell(w)=k, \ell(w\nf_s)=\ell(w)+1}\lambda_w\tau_{w\nf_s}$ 
which has height  $k+1$  if it is not zero,
\item and of $\sum_{w,  \ell(w)=k, \ell(w\nf_s)=k-1}\lambda_w(q_s \tau_{w\nf_s}+\tau_wc_{n_s} )$ which has height  $k-1$ since  $\lambda_{w_0}q_s\neq 0$ in $\k$.
\end{itemize}
By induction on the height of $a$ we obtain that the right ideal generated by $a$ is not contained in the kernel of $\delta'$. One would proceed the same way for left ideals. \\

It remains to check that $\delta'$ is symmetric that is to say that $\delta'(\tau_w\tau_{x})=\delta'(\tau_x\tau_w)$ for any $x,w\in\norm_\Gf$. The identity is clear when $x$  or $w$ has length zero.  Now suppose $x=\nf_s$ for some $s\in S$ and $w$ has length $\geq 1$.  By \eqref{Q} both $\tau_w\tau_{\nf_s}$ and  $\tau_{\nf_s}\tau_w$ are linear combinations of elements $\tau_y$ with $\ell(y)\geq \ell(w)-1$. So if $\ell(w)\geq 2$, we have
$\delta'(\tau_w\tau_{\nf_s})=\delta'(\tau_{\nf_s}\tau_w)=0$.
Now suppose $w= \nf_{s'}$ for some $s'\in S$: if $s\neq s'$, then $\delta'(\tau_w\tau_{\nf_s})=\delta'(\tau_{\nf_s}\tau_w)=0$; otherwise, $\tau_w\tau_{\nf_s}=\tau_{\nf_s}^2=\tau_{\nf_s}\tau_w$ so these elements have the same image by $\delta'$. We conclude  that  $\delta'(\tau_w\tau_{x})=\delta'(\tau_x\tau_w)$ for any $x,w\in\norm_\Gf$ by induction on the length of $x$. \\
Note that that $\delta'$ is symmetric for  any ring $R$  (in a  general context   \cite[Prop. I.3.6]{Viglivre}). It is only for its non degeneracy that we need $R$ to have  characteristic different from $p$.

\end{proof}

  \begin{notation}\label{tmod}  Let $A$, $B$, $C$ be $R$-algebras,
  $\alpha : A\rightarrow B$ and 
$\beta : B\rightarrow C$   morphisms of algebras.
For any left, resp. right, $B$-module $\mm$ we let $\alpha \mm$, resp.\ $\mm \alpha$, denote the space $\mm$ with  a left, resp. right,  $A$-action through the endomorphism $\alpha$ namely  $(a,m)\mapsto \alpha(a)m$, resp.    $(m, a)\mapsto m \alpha(a)$. Note that, for a right $C$-module $\mm$,  
$$(\mm \beta)\alpha\cong \mm\, \beta\circ \alpha \quad\textrm{as right $A$-modules}$$ and
for a left $C$-module $\mm$,  
$$\alpha ( \beta \mm )\cong  \beta \circ \alpha\,  \mm \quad\textrm{as left $A$-modules}.$$ 
If $D$ is an $R$-algebra and $\gamma: C\rightarrow D$ a morphism of algebras, let $\mm$ be a $(B, D)$-bimodule. Then we denote simply by $\alpha \mm \gamma$ the $(A,C)$-bimodule $\alpha(\mm\gamma)\cong (\alpha \mm) \gamma$.

  \end{notation}
  
   \textbf{We attach to $\Mf$ the choice of a pair $(\delta_\Mf, \iota_\Mf)$ as in Proposition \ref{tracemap}(1).}
 Let  $j=\iota_\Mf$ or $\iota_\Mf^{-1}$.
    Let $\mm$  be a left, resp. right, $\HMf$-module $\mm$;
by   Remark \ref{rema:tracemap}(1),
 the structure of left, resp. right, $\HMf$-module of $j\mm$, resp. $\mm j$,  does not depend on the choice of $(\delta_\Mf, \iota_\Mf)$.

\begin{prop}   There is an  isomorphism of  $(\Hf, \HMf)$-bimodules
\begin{equation}\label{RisobimoH}
\Hom_\HMf (\Hf, \HMf)\cong \Hf\iota^{-1} \iota_\Mf
\end{equation} where the right action of $\HMf$ on the right hand side is via the embedding $$\HMf\xrightarrow{ \iota^{-1}\circ \:\iota_\Mf}\Hf.$$

\label{RbimoH}
\end{prop}

\begin{proof}  First recall that given a  $(\HMf, \Hf)$-bimodule  $\mm$, there is a natural structure of $(\Hf, \HMf)$-bimodule on
\begin{itemize}
\item[-]  $\Hom_{\HMf}(\mm, \HMf)$ given by $((h, h_\Mf),f)\mapsto [ m\mapsto f(m.h)h_\Mf ]$,
\item[-]  $\Hom_{\Hf}(\mm, \Hf)$ given by $((h, h_\Mf),f)\mapsto [ m\mapsto  h\, f(h_\Mf .m)  ]$,
\item[-]  $\Hom_{\k}(\mm, \k)$ given by $((h, h_\Mf),f)\mapsto [ m\mapsto f(h_\Mf .m .h)  ]$.

\end{itemize}

We will check below  the following properties.
The map $f\mapsto \delta_\Mf\circ f$
   induces
   \begin{enumerate}
   \item  an isomorphism of right,  resp.\ left,  $\HMf$-modules between $\HMf\cong \Hom_\HMf(\HMf,\HMf)$  and $\Hom_\k(\iota_\Mf \HMf,\k)$  (resp. $\Hom_\k(\HMf\iota_\Mf^{-1},\k)$),
\item  an isomorphism of  $(\Hf, \HMf)$-bimodules between $\Hom_\HMf(\Hf,\HMf)$  and $\Hom_\k(\iota_\Mf\Hf,\k).$
\end{enumerate} and the map $f\mapsto \delta\circ f$ 
   induces
\begin{itemize}
\item[(3)]  an isomorphism of  $(\Hf, \HMf)$-bimodules between $\Hom_\Hf(\iota_\Mf\Hf\iota,\Hf)$  and $\Hom_\k(\iota_\Mf\Hf,\k)$. 

\end{itemize}

Combining (2) and (3) gives an isomorphism of  
$(\Hf, \HMf)$-bimodules $\Hom_\Hf(\iota_\Mf\Hf\iota,\Hf)\cong  \Hom_\HMf(\Hf,\HMf)$. 
Now consider   the natural bijection \begin{equation*}
\Hom_\Hf(\iota_\Mf \Hf \iota, \Hf)\longrightarrow \Hf \iota^{-1}\iota_\Mf, \: f\longmapsto f(1).
\end{equation*}
It is a morphism of  $(\Hf, \HMf)$-bimodules since 
$(hfh_\Mf)(1)= h(f(h_\Mf.1))=h(f(\iota_\Mf(h_\Mf)))=h(f(1.\iota^{-1}(\iota_\Mf(h_\Mf))))=hf(1)\iota^{-1} \circ \iota_\Mf(h_\Mf)$ for $(h, h_\Mf)\in \Hf\times \HMf$. This proves  the Proposition provided we verify (1), (2) and (3).

\medskip

We prove (1). 
 We write the proof in the case $\Mf=\Gf$ so as to avoid the multiple subscripts. 
The map  $ \Hf\rightarrow \Hom_\k(\iota \Hf,\k)$ that we are studying  is denoted  by $F$ in this proof. It sends  $h\in \Hf$ onto 
$x\mapsto \delta(hx)$.  It is easy to see that it is  right $\Hf$-equivariant.
It is injective since the kernel of $\delta$ does not contain any non zero left ideal
(Prop. \ref{tracemap}(1)).  We need to prove that it is surjective. 
For any $w\in\norm_\Gf$,  denote by  $e_w\in  \Hom_\k(\iota \Hf,\k)$ the  map sending $\tau_w$ onto $1\in\k$ and  $\tau_n$ onto $0$ for all $n\in \norm_\Gf-\{w\}$. These elements form a $\k$-basis of $\Hom_\k(\iota \Hf,\k)$ and we need to prove that they all lie in the image of $F$.
If  $w\in\norm_\Gf$ has maximal length, then there is $t\in \Zf$ such that $wt=\nf_\longw$  and $e_w=F(\tau_t)$. Now  let $k$ with $1\leq k \leq \ell(\longw)$ and suppose that $e_w$ lies in the image of $F$ for any $w\in \norm_\Gf$ such that $\ell(w)\geq k$. Let $v\in \norm_\Gf$ with length $k-1$ and $s\in S$ such that $\ell(\nf_sv)=\ell(v)+1=k$.   By definition, the image  of $e_{\nf_sv}$ under the right action of  $\tau_{\nf_s}$  is the $R$-linear form $e:= e_{\nf_sv}(\tau_{\nf_s}-)$ which lies in the image of $F$ by hypothesis and since $F$ is right $H$-equivariant. We have $e(\tau_{v})=1$ and we now prove  that  for 
 $x\in \norm_\Gf-\{v\}$, if  $\tau_x$ is  in the support of $e$  then $x$ has length $\geq k$. 
 Let  $x\in \norm_\Gf-\{v\}$ such that   $\tau_x$ is  in the support of $e$.
 First notice that  we have  $\ell(\nf_s x)=\ell(x)-1$.  Otherwise we would have 
  $\tau_{\nf_s}\tau_x=\tau_{\nf_s x}$ and $v=x$ since $x$ is in the support of $e$, contradiction. Therefore, $\tau_{\nf_s}\tau_x=\tau_{\nf_s}^2\tau_{ \nf_s^{-1}x}=(q_s \tau_{\nf_s^2}+  \sum_{z\in \Zf'_s} {c_{n_s}(z) }\tau_z \tau_{\nf_s}) \tau_{\nf_s^{-1}x}=q_s \tau_{\nf_s x}+  \sum_{z\in \Zf'_s} {c_{n_s}(z) } \tau_{zx}$. Since $\tau_{x}$ is in the support of $e$, we have either
 $\nf_s x=\nf_s v$ and  $x=v$  which is not true, or
 $zx= \nf_s v$ for some $z\in \Zf'_s$ and  $x$ has length $k$. It proves the claim. 
This shows   that $e$ is the sum of  $e_v$ and of $e_x$ for $x\in \norm_\Gf$ with length $\geq k$. By induction hypothesis and since $e$ lies in the image of $F$,  it implies that $e_v$ lies in the image of $F$. 
We have proved that $F$ is surjective. 
The proof of the isomorphism   between $\Hf$ and $\Hom_\k(\Hf\iota^{-1},\k)$ is similar.

We now prove (2).    
By (1), the map $f\mapsto \delta_\Mf\circ f$  gives a linear isomorphism  $\Hom_{\HMf}(\HMf, \HMf)\overset{\simeq}\longrightarrow \Hom_\k(\HMf, \k)$. Recall that $\Hf$ is a free left  $\HMf$-module (Prop. \ref{liberte}).   This implies that the map  $\phi\mapsto \delta_\Mf\circ \phi$  defines a linear  isomorphism   $\Hom_{\HMf}(\Hf, \HMf)\cong \Hom_\k(\Hf, \k)$.  It remains to verify that it induces a morphism of $(\Hf, \Hf_\Mf)$-bimodules  $\Hom_{\HMf}(\Hf, \HMf)\cong \Hom_\k(\iota_\Mf\Hf, \k)$ which is immediate.

To prove (3), first notice that the map $f\mapsto \delta\circ f$ induces a  morphism of  $(\Hf, \Hf_\Mf)$-bimodules from
$\Hom_\Hf(\iota_\Mf \Hf\iota , \Hf)\rightarrow \Hom_R(\iota_\Mf \Hf, R)$.
Using (1)  with $\Mf=\Gf$,  we know that for any element  $\varphi\in \Hom_R(\Hf, R)$ there is a unique $f\in \Hom_\Hf(\Hf, \Hf)$ such that $\varphi= \delta\circ f$.
But precomposing by $\iota^{-1}$ induces a linear isomorphism 
$ \Hom_\Hf(\Hf,\Hf)\overset{\simeq}\longrightarrow\Hom_\Hf(\iota_\Mf \Hf\iota , \Hf)$ so $\varphi=\delta\circ f\circ \iota^{-1}$  where $f\circ\iota^{-1}\in \Hom_\Hf(\iota_\Mf \Hf\iota , \Hf)$.

\end{proof}

\begin{rema}\label{RbimoHnonp}
If $\k$ is a ring of characteristic different from $p$, we may attach to $\Mf$ the 
non degenerate symmetric map $\delta'_\Mf:\HMf\rightarrow R$
as in Proposition \ref{tracemap}(2) and the corresponding automorphism $\iota'_\Mf$ of $\HMf$ is  the identity.  We then have an  isomorphism of  $(\Hf, \HMf)$-bimodules
$\Hom_\HMf (\Hf, \HMf)\cong \Hf$.
\medskip

The strategy to prove this is the same as the one in the proof of Prop. \ref{RbimoH} with $\delta'$,  $\delta_\Mf'$,  $\iota'$  and $\iota'_\Mf$ instead of, respectively, $\delta$,  $\delta_\Mf$,  $\iota$  and $\iota_\Mf$. 
All arguments go through immediately except for the analogous of point (1) of that proof: we verify here that the map  $F': \Hf\rightarrow \Hom_\k(\Hf,\k)$ sending  $h\in \Hf$ onto 
$x\mapsto \delta'(hx)$ is surjective. We need to prove that  $e_w$ (same notation as in the above mentioned proof) lies in the image of $F'$  for  any $w\in \norm_\Gf$. If $w$ has length zero, then $e_w=F'(\tau_{w^{-1}})$. Now suppose that $e_w$ lies in the image of $F'$ for all $w$ with $\ell(w)\leq k$. Let $v\in \norm_\Gf$ with length $k+1$ and $s\in S$ such that $\ell( \nf_sv)=\ell(v)-1$. The $R$-linear form $e:= e_{\nf_s v}.\tau_{\nf_s}=e_{\nf_s v}(\tau_{\nf_s}-)$  lies in the image of $F'$ by hypothesis and since $F'$ is right $\Hf$-equivariant. 
Because $R$ has characteristic different from $p$ and using \eqref{Q}, we have
 $e(\tau_{v})=q_s\neq 0$. In fact, still using \eqref{Q} one easily checks that $e= q_s e_v$. So $e_v$ lies in the image of $F'$.

\end{rema}
\begin{rema}
Suppose that $R$ is a field. Then the above results follow from the  theory of Frobenius algebras.
We  mentioned in \eqref{ssBNpJ} that $(\Mf, \Bf_{\Mf}, \norm_{\Mf}, S_J)$ together with the decomposition $\Bf_{\Mf}=\Zf\Uf_{\Mf}$   is a (strongly) split $BN$-pair of characteristic $p$  \cite[Def. 2.20]{CE}.  The results of \cite[Prop. 3.7]{Tin}
\cite[Thm. 2.4]{Saw},\cite[Prop. 6.11]{CE} apply  to the finite Hecke algebras  $\Hf_\Mf$. In particular, 
they are Frobenius.   
If $\k$ is a field, then the proof of Prop. \ref{RbimoH} simplifies greatly (by argument of dimension) and one can prove in fact more generally that 
for  $\m$ a $(\HMf, \Hf)$-bimodule there is an isomorphism of $(\Hf, \HMf)$-bimodules
$\Hom_\Hf(\iota_\Mf \m \iota, \Hf)\cong \Hom_\HMf(\m,\HMf)$.

\end{rema}

\begin{rema} \label{rema:tracemapfield}
 Prop. \ref{tracemap} (2) implies that if $R$ is a field of characteristic different from $p$, then  $\HMf$ is not only Frobenius but also  symmetric. 
If $\k$ is a field of characteristic $p$, then $\HMf$ is not symmetric in general \cite[Addendum to Theorem 24]{Saw}. 

\end{rema}

\section{Parabolic induction for the finite group and unipotent finite Hecke algebra}

Let $R$ be  a commutative ring. For $\Pf$ a standard parabolic subgroup of $\Gf$ with Levi decomposition $\Pf=\Mf \, \Nf$,  we consider the category $\Rep(\Mf)$ (resp. $\Rep(\Pf)$) of $\k$-representations of $\Mf$ (resp. $\Pf$), that is, the category of left $R[\Mf]$-modules (resp. $R[\Pf]$-modules ), and the category $\Mod(\Hf_\Mf)$ of  right $\Hf_\Mf$-modules.  

\subsection{Parabolic induction and restriction for the representations of  the finite group.\label{parabfinitegroup}}

For $\V\in \Rep(\Mf)$, we consider the representation  of $\Gf$ on the space $\Ind ^\Gf_\Pf(\V)$ of functions $f: \Gf\rightarrow \V$ such that $f(mng)=m.v$ for all $m\in \Mf$, $n\in \Nf$ and $g\in \Gf$. The action of $\Gf$ is by right translations, $(g,f)\mapsto f( \, .\,g)$. This define the functor  of parabolic induction 
$$\Ind_\Pf^\Gf: \Rep(\Mf)\longrightarrow \Rep(\Gf).$$

For $\V\in \Rep(\Gf)$, we consider the representation of $\Mf$ on the $\Nf$-invariant subspace $\V^\Nf$ which  defines a functor denoted by 
$$(-)^\Nf:\Rep(\Gf)\longrightarrow  \Rep(\Mf)$$ and the representation of $\Mf$ on the $\Nf$-coinvariant space $\V_\Nf$ which  defines a functor denoted by 
$$ (-)_\Nf:\Rep(\Gf)\longrightarrow  \Rep(\Mf).$$

\begin{prop} \label{finiteadjoints}The functor  $\ii_\Pf^\Gf$ is faithful, exact,  of  right adjoint $(-)^\Nf$ and left adjoint $(-)_\Nf$. 

\end{prop}

 \begin{rema}

The $\Nf$-invariant  and $\Nf$-coinvariants functors are more naturally defined on the category $\Rep(\Pf)$ of the $R$-representations of $\Pf$, where they are respectively the  right and left adjoint functors of the inflation $\Rep(\Mf)\to \Rep(\Pf)$.

\end{rema}

\begin{proof} Although this property is well known we do not know a reference and we give a proof.   By Frobenius reciprocity, the restriction $\Rep(\Gf)\to \Rep(\Pf)$ from $\Gf$ to $\Pf$ admits a left and right adjoint equal respectively to the induction $R[\Gf]\otimes_{R[\Pf]} -$ and to the coinduction $  \Hom_{R[\Pf]}(R[\Gf], -)$  \cite[2.8.2]{Benson}. Here, $R[\Gf]$ is seen as a right $R[\Pf]$-module for the induction and as a left $R[\Pf]$-module for the coinduction.  
It is well known that the coinduction and induction functors $\Rep(\Pf)\to \Rep(\Gf)$  coincide.


 The parabolic induction  $\ii_\Pf^\Gf$ is the composite of  the  inflation $\rm{inf}_\Pf:\Rep(\Mf)\to \Rep(\Pf)$ from $\Mf$ to $\Pf$ and of the 
coinduction $\Rep(\Pf)\to \Rep(\Gf)$. 
 As the $\Nf$-coinvariant functor $\Rep(\Pf)\to \Rep(\Mf)$  and the $\Nf$-invariant functor $\Rep(\Pf)\to \Rep(\Mf)$ are  respectively the left adjoint and the right adjoint of the inflation,
 the functor $(-)_\Nf$ is a left adjoint of  $\ii_\Pf^\Gf$ and the functor $(-)^\Nf$
is a right adjoint of $\ii_\Pf^\Gf$. As   a functor with a left and a right adjoint, $\ii_\Pf^\Gf$  is exact. 
\\
Lastly, it is easy to see that $\ii_\Pf^\Gf$ is faithful since $R[\Gf]$  is a free right $R[\Pf]$-module.
\end{proof}

\begin{rema}\label{rema:indcontra} Suppose that $R$ is a field and
 consider the   contravariant endofunctor $\V\mapsto \V^\vee$ of $\Rep(\Gf)$ attaching to a representation $\V$ its contragredient representation $\V^\vee=\Hom_R(\V, R)$. 
We check that it commutes with parabolic induction, \emph{i.e.} for any representation $\V$ of $\Mf$ we have a natural isomorphism of representations of $\Gf$:
$$\Ind_\Pf^\Gf(\V^\vee)\longrightarrow  \Ind_\Pf^\Gf(\V)^\vee.$$ 
Let $\langle \,.\, ,\,.\,\rangle: \V^\vee\times\V\rightarrow k$ denote the duality between $\V$ and $\V^\vee$. To an element 
$\xi\in  \Ind_\Pf^\Gf(\V^\vee)$ we associate the linear map 
$$\Ind_\Pf^\Gf(\V)\rightarrow R, \: f\mapsto \sum_{g\in \Pf\backslash \Gf} \langle\xi(g), f(g) \rangle.$$ This clearly defines an injective linear map $\Ind_\Pf^\Gf(\V^\vee)\rightarrow (\Ind_\Pf^\Gf(\V))^\vee$. 
It is surjective:  $\Phi\in
(\Ind_\Pf^\Gf(\V))^\vee$  is the image of the element $\xi\in \Ind_\Pf^\Gf(\V^\vee)$ such that $\langle \xi(g),v\rangle=\Phi( f_{\Pf g, v})$, where 
$f_{\Pf g, v}\in \Ind_\Pf^\Gf(\V) $ is the function with support in $\Pf g$ and value $v$ at $g$. 

Now  to a representation $\V\in \Rep(\Pf)$, we can 
also attach its contragredient representation $\V^\vee=\Hom_R(\V, R)$. 
The
 contragredient of the $\Nf$-coinvariants of $\V$ and the  $\Nf$-invariants of the contragredient of $\V$ are functorially isomorphic representations of $\Mf$: there is a natural isomorphism of  representations of $\Mf$:
 \begin{equation}\label{commcont}(\V_\Nf)^\vee\cong (\V^\vee)^\Nf.\end{equation}
 Denote by  $\V(\Nf)$   the subspace of $\V$ generated by $nv-v$ for all $n\in \Nf, v\in \V$.   By definition,  $\V_\Nf=\V/\V(\Nf)$. An element of $ \V^\vee$ is $\Nf$-invariant if and only it vanishes on  $\V(\Nf)$. Therefore $(\V_\Nf)^\vee \simeq (\V^\vee)^{\Nf}$. It is obvious that   the isomorphism is $\Mf$-equivariant and functorial.  
\end{rema}

\subsection{\label{indufinite}Induction, coinduction and restriction for finite Hecke modules.}
We recall  that $\Hf_\Mf$ is isomorphic  to the subalgebra of   $\Hf$  of basis $(\tau_w)_{w\in \norm_{\Mf}}$. We consider the natural restriction functor
$$\r_{\Hf_\Mf}^\Hf= -\otimes _\Hf \Hf: \Mod(\Hf)\longrightarrow \Mod(\Hf_\Mf)$$ where $\Hf$ is seen as a  $(\Hf,\Hf_\Mf)$-bimodule.
 For $\mm$ a right $\Hf_\Mf$-module, we consider the induced $\Hf$-module $\mm\otimes_\HMf \Hf$
 where $\Hf$ is seen as a  $(\HMf,\Hf)$-bimodule.
This defines the functor  of induction for Hecke modules, which we will sometimes call parabolic induction by analogy with the definitions in \S\ref{parabfinitegroup}:
$$\ii_{\HMf}^\Hf= -\otimes _\HMf \Hf: \Mod(\Hf_\Mf)\longrightarrow \Mod(\Hf).$$
It  is left adjoint to $\r_{\Hf_\Mf}^\Hf$.
The space $\Hom_\HMf(\Hf, \mm)$ has a structure of right $\Hf$-module given by $(f,h)\mapsto [f.h: x\mapsto f(hx)]$ and this defines the functor of coinduction  for Hecke modules
$$\cii_{\HMf}^\Hf= \Hom_\HMf(\Hf, -): \Mod(\Hf_\Mf)\longrightarrow \Mod(\Hf).$$
It  is right adjoint  to $\r_{\Hf_\Mf}$. 

\medskip

The functor $\ii^\Hf_{\HMf}$   has a left adjoint because it  commutes with small projective limits as the left $\HMf$-module $\Hf$ is free  (Prop. \ref{liberte}).  We denote it by $$\ll^\Hf_{\HMf}: \Mod(\Hf)\longrightarrow \Mod(\HMf)$$ and  describe it explicitly in the following proposition.

\medskip

Recall that we attach to $\Mf$ the choice of a pair $(\delta_\Mf, \iota_\Mf)$ as in Proposition \ref{tracemap}(1). We refer to the comments before Prop. \ref{RisobimoH}.

\begin{prop}\label{prop:compafunc}

Let $R$ be a ring. 
 Let $\Mf':= \longw \Mf  \longw^{-1}$.  \\ The isomorphism of algebras 
$\Hf_{\Mf'}\xrightarrow{\iota_{\Mf} ^{-1}\circ \iota} \Hf_{\Mf}$
 induces an equivalence of categories $- \, \iota_{\Mf} ^{-1}\iota: \Mod(\HMf)\rightarrow \Mod(\Hf_{\Mf'})$ with quasi-inverse  $- \, \iota^{-1}\iota_{\Mf}: \Mod(\Hf_{\Mf'})\rightarrow \Mod(\HMf) $.\\
 The functors  $$-\otimes _\HMf \Hf\textrm{ and }\Hom_{\Hf_{\Mf'}}(\Hf , - \, \iota_{\Mf} ^{-1}\iota) \::\: \Mod(\HMf)\rightarrow \Mod(\Hf)$$ are naturally isomorphic that is to say 
$$\ii_{\HMf}^\Hf\cong \cii_{\Hf_{\Mf'}}^\Hf  (-\, \iota_{\Mf} ^{-1}\iota).$$
The left adjoint of the functor $\ii_\HMf^\Hf$ is
$$\ll^\Hf_{\HMf}=(\r^\Hf_{\Hf_{\Mf'}}(-))\iota ^{-1}\iota_{\Mf}: \Mod(\Hf)\rightarrow \Mod(\HMf).$$
In particular, $\ll^\Hf_\HMf$ is an exact functor.
 \end{prop}
 
\begin{proof} There is a natural natural morphism of functors 
$$-\otimes _\HMf \Hf\longrightarrow \Hom_\HMf(\Hom_\HMf(\Hf, \HMf), -)$$ where $\Hf$ is seen as a $(\Hf_\Mf, \Hf)$-bimodule.
By Proposition \ref{liberte}, the  left $\HMf$-module $\Hf$ is free  (of finite rank)  and it easily follows that the natural transformation above is an isomorphism of functors:
$$-\otimes _\HMf \Hf\cong \Hom_\HMf(\Hom_\HMf(\Hf, \HMf), -): \Mod(\Hf_\Mf)\longrightarrow \Mod(\Hf).$$
Combining this with \eqref{RisobimoH},  we obtain an isomorphism of functors $\Mod(\HMf)\rightarrow \Mod(\Hf)$
$$-\otimes _\HMf \Hf\cong \Hom_\HMf(\Hf\iota^{-1} \iota_\Mf, -)$$ where we recall that the right action of $\HMf$ on $\Hf\iota^{-1} \iota_\Mf$ is via the  embedding
 $\HMf\xrightarrow{ \iota^{-1}\circ \:\iota_\Mf}\Hf.$
The automorphisms   $\iota$ and $\iota_\Mf$ are described explicitly in Prop. \ref{tracemap}(1).
Let  $\tw$ (resp. $\twm$)  be a lift in $\norm _\Gf$ (resp. $\norm _\Mf$) of the longest element $\longw$ of $\Wf$ (resp. $\longw_{\Mf}$ of $\Wf_\M$). 
We pick $\widetilde{\longw_{\Mf'}}:=\tw^{-1} \twm \tw$ as a lift in $\norm _{\Mf'}$ of the longest element $\longw_{\Mf'}$ of $\Wf_{\Mf'}$. The corresponding automorphism $\iota_{\Mf'}$ of $\Hf_{\Mf'}$ as defined in Prop. \ref{tracemap}(1) then coincides with the restriction to  $\Hf_{\Mf'}$  of $\iota^{-1}  \iota_\Mf    \iota$.

For any right $\HMf$-module $\m$,  have  the following isomorphisms of right $\Hf$-modules (see Notation \ref{tmod}):
\begin{align*}\Hom_\HMf(\Hf\iota^{-1}   \iota_\Mf, \mm)&\cong
  \Hom_{\iota^{-1} (\HMf)}(\Hf\iota^{-1}  \iota_\Mf    \iota  , \mm \iota  )= \Hom_{\Hf_{\Mf'}}(\Hf\iota^{-1}   \iota_\Mf   \iota  , \mm \iota )\cong \Hom_{\Hf_{\Mf'}}(\Hf\iota_{\Mf'} , \mm \iota )\\ &\cong \Hom_{\Hf_{\Mf'}}(\Hf , \mm \iota\iota_{\Mf'}^{-1})\cong \Hom_{\Hf_{\Mf'}}(\Hf , \mm \iota_{\Mf}^{-1}\iota ).\end{align*}
So the functors 
$-\otimes _\HMf \Hf\cong  \Hom_{\Hf_{\Mf'}}(\Hf , - \iota_{\Mf}^{-1}\iota )$ are isomorphic.
\end{proof}

Recalling that $\Hf$ is free as a left $\HMf$-module (Prop. \ref{liberte}), with the notations of  Prop. \ref{prop:compafunc}  we have:

\begin{coro}  \label{cor:exactadj} The induction $\ii_\HMf^\Hf$ is faithful, exact,  of exact right adjoint  $\r^\Hf_{\Hf_{\Mf}}$ and exact left adjoint $(\r^\Hf_{\Hf_{\Mf'}}(-))\iota ^{-1}\iota_{\Mf}$.
\end{coro}

 Compare with the parabolic  functor $\ii_\Pf^\Gf$ for the group (\S \ref{parabfinitegroup})  where  the right adjoint $(-)^\Nf$ and the left adjoint  $(-)_\Nf$   are not  always exact (see Prop. \ref{prop:notexact}).

\begin{rema}\label{L=res}
 If $\k$ is a ring of characteristic different from $p$, then 
 by Remark \ref{RbimoHnonp}, we have an isomorphism of $(\Hf, \HMf)$-bimodules
 $\Hom_\HMf (\Hf, \HMf)\cong \Hf$, so following the arguments of the proof of 
 Prop. \ref{prop:compafunc} we would prove in this case that 
 the functors $\ii^\Hf_{\HMf}$ and $\cii^\Hf_{\HMf}$ coincide. In particular, $\r^\Hf_{\HMf}$ is not only the right adjoint but also the left adjoint of $\ii^\Hf_{\HMf}$, and   we have $\ll^\Hf_{\HMf}\cong  \r^\Hf_{\HMf}$. This means that the pair $(\HMf, \Hf)$ is  a Frobenius pair. Note that by Corollary \ref{cor:exactadj}, we also have 
$\ll^\Hf_{\HMf}\cong  \r^\Hf_{\HMf}\cong (\r^\Hf_{\Hf_{\Mf'}}(-))\iota ^{-1}\iota_{\Mf}$.  
  \medskip

\end{rema}

 \begin{rema}  Let $\Mf$ and $\Mf'$ as in Prop. \ref{prop:compafunc}.  From the proposition, we also get the following isomorphism of functors $$\cii_{\HMf}^\Hf\cong \ii_{\Hf_{\Mf'}}^\Hf  (-\, \iota_{\Mf} ^{-1}\iota).$$
Denote by
  $\Triv_{\HMf}$, respectively $\Triv_{ \Hf_{\Mf'}}$, the restriction to $\Hf_{\Mf}$, respectively  $\Hf_{\Mf'}$,  of the trivial character of $\H$  as defined in Remark \ref{rema:idempo}.  Then one easily checks that 
     $\Triv_{\HMf}\, \iota_{\Mf} ^{-1}\iota\cong \Triv_{ \Hf_{\Mf'}}$ and therefore
  $$\cii_{\HMf}^\Hf(\Triv_{ \Hf_{\Mf}})\cong  \ii_{\Hf_{\Mf'}}^\Hf(\Triv_{ \Hf_{\Mf'}}).$$ Suppose that $R$ has characteristic $p$. Then one checks that, in general,  the right $\Hf$-modules
  $\ii_{\Hf_{\Mf'}}^\Hf(\Triv_{ \Hf_{\Mf'}})$ and   $\ii_{\Hf_{\Mf}}^\Hf(\Triv_{ \Hf_{\Mf}})$   are not isomorphic. (For example, when $\Gf={\rm GL}_3(\mathbb F_q)$  and $R$ is a field of characteristic $p$, we have $\Pi=\{\alpha_1, \alpha_2\}$. For $\{i, i'\}=\{1,2\}$  let $\Mf_i$ the standard Levi subgroup corresponding to  $\{\alpha_i\}$ and $s_i\in \Wf$ the corresponding reflection. 
  We have $\Mf_{i'}= \longw \Mf_i \longw^{-1}$.
  Then
   $\ii_{\Hf_{\Mf_i}}^\Hf(\Triv_{ \Hf_{\Mf_i}})$ 
   is a $3$-dimensional vector space. As a right $\Hf$-module, it has a $2$-dimensional socle and the quotient of   
 $\ii_{\Hf_{\Mf_i}}^\Hf(\Triv_{ \Hf_{\Mf_i}})$  by its socle is the character
     $\varepsilon_1\mapsto 1$,  $\varepsilon_1\tau_{n_{s_i}}\mapsto 0$, 
$\varepsilon_1\tau_{n_{s_{i'}}}\mapsto -1$  for $\Hf$.)
   So if $R$ has characteristic $p$, then the functors 
  $\ii_{\HMf}^\Hf$ and $\cii_{\HMf}^\Hf$
  are not isomorphic in general. \end{rema}

\subsection{Commutative diagrams: parabolic induction and unipotent-invariants functor}  
Let $\Pf$ be  a standard parabolic subgroup of $\Gf$ with Levi decomposition $\Pf=\Mf \, \Nf$.
\subsubsection{The $\Uf$-invariant functor and its left adjoint.\label{Uinv}}
Let $R$ be a commutative ring. We consider the universal module $\XMf= \ind_{ \Uf_\Mf}^\Mf(1)$. Recall from \ref{finiteHecke} that it is naturally  a left $\HMf$-module  and a  representation in $\Rep(\Mf)$. When $\Mf=\Gf$, we write simply $\Xf$ instead of $\Xf_\Gf$.
The  functor 
     \begin{align}\label{finite-tens} -\otimes_{\Hf} \Xf:\Mod(\Hf)\to \Rep(\Gf)  \end{align}
  is the left adjoint of the $\Uf$-invariant functor 
      \begin{align}\label{finite-inv} \Hom_\Gf(\Xf, -)=(-)^{\Uf}:\Rep(\Gf)
\to \Mod(\Hf).  \end{align}
It is therefore  right exact  and the  functor $(-)^{\Uf }$ is left exact.  The analogous definitions and statements hold for 
$(-)^{\Uf_\Mf}$ and 
$ -\otimes_\HMf\XMf$.

 The functor $-\otimes_{\Hf} \Xf$ is left exact if and only if the $\Hf$-module $\Xf$ is  flat.   
  In general, $\Xf$ is not flat over $\Hf$:
if  $\Gf={\rm GL}(n , \mathbb F_q)$ and $R$ is a field of characteristic $p$,  then $\Xf$ is a flat $\Hf$-module if $n=2$ and $q=p$ or if $n=3$ and $q=2$. It is not flat if  $n=2$ and $q\neq p$, if  $n=3$ and $q\neq 2$, and if $n\geq 4$ and $q\neq 2$ (\cite[Thm B]{OS}).

 \begin{rema} \label{rema:equiv}Let $R$ be a field  of characteristic different from $p$. Then, the $\Uf$-invariant functor and its left adjoint have  properties coming from the existence of  an  idempotent $e$ of $R[\Uf]$ such that $V^\Uf=eV$ for $V\in \Rep(\Gf)$.    In particular, the functor $(-)^\Uf$ is exact. Moreover, it admits a right adjoint because it commutes with small direct sums (\cite[Prop. 2.9]{Vigadjoint}). We have $\Xf  =eR[\Gf]$ and $\Hf=eR[\Gf]e$. The functor $ -\otimes_\Hf\Xf$   is fully faithful because it has a right adjoint and the map $\M\mapsto (\M\otimes_\Hf\Xf )^{\Uf}= \M\otimes_{eR[\Gf]e}eR[\Gf]e  $ is a functorial  isomorphism in  $\Mod(\Hf)$  (\cite[Prop. 1.5.6]{KS}).  If  $ -\otimes_\Hf\Xf$   admits a left adjoint, then it is exact. Let $\Rep(\Gf)^\dagger$ be the full subcategory of $\Rep(\Gf)$ of representations   generated by their $\Uf $-fixed vectors.
The  following properties are equivalent: \begin{enumerate}
 \item  $(-)^\Uf$ is fully faithful,
 \item  $(-)^{\Uf }: \Rep(\Gf)^\dagger\to \Mod(\Hf)$ is an equivalence of categories of inverse $ -\otimes_\Hf\Xf$,
 \item    $ -\otimes_\Hf\Xf: \Mod(\Hf)\to\Rep(\Gf)^\dagger$ is  an equivalence of categories of inverse $(-)^{\Uf }$,
 \item the category $\Rep(\Gf)^\dagger$ is abelian.
  \end{enumerate}
The equivalences between (1), (2), (3) follow from \cite[Prop. 2.4]{Vigadjoint},  the equivalence with (4)  follows from  \cite[\S I.6.6]{Vig1}. The properties (1), (2), (3) and (4)  are not satisfied already for $\Gf = {\rm GL}(2,\mathbb F_q)$ if $q^2-1=0$ in $R$, because the representation $ \Ind_{\Bf}^{\Gf}1$ generated by its $\Uf$-invariants  has length $3$  (James  Proc. London Math Soc 52(1986) 236-268) and   $\dim_R(\Ind_{\Bf}^{\Gf}1)^{\Uf }=2$.   When $R$ is an algebraically closed field of characteristic $p$, it is proved  \cite[Prop. 2.13]{OS}  (based on work by Cabanes  in \cite{Cabanes}) that the above conditions are equivalent to  $\X$ being flat over $\Hf$.
 \end{rema}

When $R$ is a field, let $\Rep^f(\Gf)$ denote the full subcategory of $\Rep(\Gf)$ of all finite dimensional representations.

\begin{rema} \label{rmk:proj} Suppose that $R$ is a field. We recall  that a finite dimensional representation of $\Rep(\Gf)$ is projective in $\Rep(\Gf)$ if and only if it is projective in $\Rep^f(\Gf)$.
\end{rema}
\begin{proof} This is a general and straightforward property. It is obvious that  a projective  objet of  $\Rep(\Gf)$ which is finite dimensional is     projective   in $\Rep^f(\Gf)$. 
  Conversely,  let $P$ be an object  in $\Rep^f(\Gf)$ and a surjective $R[\Gf]$-equivariant morphism  $\pi: R[\Gf]^s\twoheadrightarrow P$ for some $s\geq 0$.
 If $P$ is projective in $\Rep^f(\Gf)$, then this surjective morphisms splits and $P$ is a direct summand of  the free module $R[\Gf]^s$. So $P$ is projective in 
 $\Rep(\Gf)$.   
\end{proof}

 \begin{lemma} \label{Xnotproj}If $R$ is  a field of characteristic  $p$ and $\Uf\neq \{1\}$, then $\Xf$ is not  a projective object in $\Rep^f(\Gf)$. 

\end{lemma}

\begin{proof}
We use the results of  modular representation theory of finite groups  recalled in \cite[\S4]{Pas}.
We suppose that $R$ is a field of characteristic $p$.
Recall that $\Uf$ being a $p$-group, the trivial representation is its unique irreducible representation and  $R[\Uf]$ is the only principal indecomposable module in the category $\Rep(\Uf)$  of all $R[\Uf]$-modules. This implies that for any  finite dimensional projective object $P$ in  $\Rep (\Gf)$, the dimension of $P$   is equal to the product of $\vert \Uf\vert$ by the dimension of the $\Uf$-invariant subspace $P^\Uf$ (\cite[Cor. 4.6]{Pas}). If $\Xf$ were projective,  its dimension would be   $\vert \Uf\vert \,\vert \norm_\Gf\vert$. But the dimension of $\Xf$  is strictly smaller than $\vert \Uf\vert \,\vert \norm_\Gf\vert$ because it  is the sum over $w\in\norm_\Gf$ of all $\vert \Uf\backslash \Uf w\Uf\vert$,     we have   $\vert \Uf\backslash \Uf w\U \vert \leq  \vert \Uf \vert$ the integer $\vert \Uf\backslash \Uf w\U\vert $  is equal to $1$ when $w=1$, and  $\Uf\neq \{1\}$.  
\end{proof}

\begin{prop}  \label{prop:notexact}
   If $R$ be a field of characteristic  $p$ and $\Uf\neq \{1\}$,  then
\begin{itemize}\item[-] The  $ \Uf $-invariant functor  $(-)^\Uf$   (resp. the  restriction of $(-)^\Uf$    to  
 $\Rep^f(\Gf)$) is not right exact. 

\item[-] The  $ \Uf $-coinvariant functor $(-)_\Uf$ (resp. the  restriction of $(-)_\Uf$    to  
 $\Rep^f(\Gf)$) is not left exact. 

\end{itemize}
\end{prop}

\begin{proof} That the restriction to  $\Rep^f(\Gf)$ of the $ \Uf $-invariant functor  is not right exact follows immediately from Lemma \ref{Xnotproj}. This implies that  the $ \Uf $-invariant functor  is not right exact.
 For the second point, note first that 
the contragredient functor is a contravariant  involution of $\Rep^f(\Gf) $ exchanging the $ \Uf $-invariant functor   and the  $ \Uf $-coinvariant functor in the sense of \eqref{commcont}.  Therefore, the restriction of  the  $ \Uf $-coinvariant functor   to 
$\Rep^f(\Gf) $ is exact if and only if  the restriction of  the  $ \Uf $-coinvariant functor  to  $\Rep^f(\Gf) $
is exact. We deduce that  the restriction of  the  $ \Uf $-coinvariant functor  to $\Rep^f(\Gf) $ is not left exact which implies that of  the  $ \Uf $-coinvariant functor is not left exact.
\end{proof}

\begin{rema} 
By Prop. \ref{prop:notexact}, if $R$ is a field of characteristic $p$ and $\Uf\neq \{1\}$, 
then $(-)^{\Uf }$ does not have a right adjoint. 
  \end{rema}

\subsubsection{\label{equivUinv}On the $\Uf$-invariants functor in characteristic $p$}  In this paragraph, $R$ is a field of characteristic $p$ which we suppose algebraically closed. This ensures that 
 the simple $\Hf$-modules are $1$-dimensional (\cite[Lem. 3.13]{Tin}). 
By work of Carter and Lusztig \cite{CL}, the functor \eqref{finite-inv} induces a bijection between irreducible 
$R$-representations of $\Gf$ and simple $\Hf$-modules.

Since $R$ has characteristic $p$, every non zero representation $\V\in \Rep(\Gf)$ has a non zero $\Uf$-fixed vector. 
Furthermore, the irreducible representations of $\Gf$ are finite dimensional. Therefore, considering  $\Mod^f(\Hf)$ the category of finite dimensional $\Hf$-modules and $\Rep^{f,\Uf}(\Gf)$ the full subcategory  of $\Rep(\Gf)$ of the finite dimensional representations generated by their $\Uf$-invariant subspace,  it is  natural to ask if \eqref{finite-tens} and \eqref{finite-inv} are quasi-inverse equivalences of categories between $\Mod^f(\Hf)$ and  $\Rep^{f,\Uf}(\Gf)$.

The answer is 
no in general. 
However, Cabanes described in \cite{Cabanes} an additive subcategory $\Bb(\Gf)$ of $\Rep^f(\Gf)$ which is equivalent to $\Mod^f(\Hf)$ via the functor \eqref{finite-inv}. In \cite[Prop. 2.13]{OS} it is proved (in the setting of ${\rm GL}(n,\mathbb F_q)$ but the proof is general) that $\Bb(\Gf)$ and  $\Rep^{f,\Uf}(\Gf)$ coincide if and only if $\Xf$ is a flat $\Hf$-module.  Compare with Remark \ref{rema:equiv}. We recalled in \S\ref{Uinv} flatness conditions for $\Xf$ when   $\Gf={\rm GL}(n,\mathbb F_q)$.

\subsubsection{Commutative diagrams: questions} We  study the connection between  the functors  of induction for Hecke modules and for representations via the functors \eqref{finite-tens} and \eqref{finite-inv} and their analogs for $\Mf$.  We examine the following questions.

\begin{question}\label{qf1} Does the  $\Uf$-invariant  functor commute with  parabolic induction: is the following diagram commutative?
 \begin{equation}\begin{CD}\Rep(\Mf) @>{(-)^{\Uf_\Mf}}>>  \Mod(\Hf_\Mf)\cr
{\ii_\Pf^\Gf}@VVV  {\ii_{\Hf_\Mf} ^\Hf}@VVV \cr
 \Rep(\Gf) @>{(-)^{\Uf}}>>        \Mod(\Hf)\cr
\end{CD}\label{diag1}\end{equation}
Passing to left adjoints, this is equivalent to asking whether the following diagram is commutative:

 \begin{equation}\begin{CD}\Mod(\Hf) @>{-\otimes_\Hf \Xf}>>  \Rep(\Gf)\cr
{\ll_{\HMf}^\Hf}@VVV  {(-)_\Nf}@VVV \cr
 \Mod(\Hf_\Mf) @>{-\otimes_\HMf \XMf}>> \Rep(\Mf)\cr
\end{CD}\label{diag1'}\end{equation}

We answer this question positively in \S\ref{answers} (Proposition \ref{prop:Q1})  for an arbitrary  ring $R$. In passing, we  will 
consider the functor $\dagger: \Rep(\Mf)\rightarrow \Rep(\Mf)$ sending a representation $\V$ onto the subrepresentation generated by  the $\UMf$-fixed subspace $\V^\UMf$ (and likewise for $\dagger: \Rep(\Gf)\rightarrow \Rep(\Gf)$). We show (Proposition \ref{daggercommute}) that $\ii_\Pf^\Gf$ and $\dagger$ commute.\\
\end{question}
\begin{question}\label{qf2} Does the $\Uf$-invariant functor commute with  the right adjoints of parabolic induction:
is the following diagram commutative?

 \begin{equation}\begin{CD}\Rep(\Gf) @>{(-)^\Nf}>> \Rep(\Mf)\cr
{(-)^{\Uf}}@VVV  {(-)^{\Uf_\Mf}}@VVV \cr
 \Mod(\Hf) @>{\r_{\HMf}^\Hf}>>        \Mod(\Hf_\Mf)\cr
\end{CD}\label{diag2}\end{equation}
Passing to left adjoints, this is equivalent to asking if the following diagram is commutative:
 \begin{equation}\begin{CD}\Mod(\Hf_\Mf) @>{-\otimes_{\Hf_\Mf} \XMf}>>  \Rep(\Mf)\cr
{\ii_{\Hf_\Mf}^\Hf}@VVV  {\ii^\Gf_{\Pf }}@VVV \cr
 \Mod(\Hf) @>{-\otimes_{\Hf} \Xf}>> \Rep(\Gf)\cr
\end{CD}\label{diag2'}\end{equation}
We answer this question positively in \S\ref{answers} for an arbitrary ring  and describe the natural transformations corresponding to 
\eqref{diag2'} in \eqref{eq:Q2} and \eqref{eq:Q22}.\\

\end{question}

\begin{question}\label{qf3}
Does the $\Uf$-invariant functor commute with  the left adjoints of parabolic induction:
is the following diagram commutative?
 \begin{equation}\begin{CD}\Rep(\Gf) @>{(-)_\Nf}>> \Rep(\Mf)\cr
{(-)^{\Uf}}@VVV  {(-)^{\Uf_\Mf}}@VVV \cr
 \Mod(\Hf) @>{\ll_{\HMf}^{\Hf}}>>        \Mod(\Hf_\Mf)\cr
\end{CD}\label{diag3}\end{equation}

The answer is negative in general, but positive if $R$ is a field of  characteristic  different from $p$ (Proposition \ref{prop:Q3} and comments following it).

\end{question}
\subsubsection{\label{answers}Commutative diagrams: answers} 

\noindent The \textbf{answer  to Question} \ref{qf1} is positive  when $R$ is an arbitrary  ring: 
\begin{prop}\label{prop:Q1}  For any representation $\V\in \Rep(\Mf)$, there is a natural isomorphism of right $\Hf$-modules
$$\V^{\Uf_\Mf}\otimes_{\HMf} \Hf\cong (\Ind_{\Pf}^{\Gf}(\V))^\Uf.$$

\end{prop}

\begin{coro}  
The diagrams \eqref{diag1} and \eqref{diag1'} commute.
\end{coro}

 The isomorphism of Proposition \ref{prop:Q1} is given by the map \eqref{naturaliso1} below. We prove the proposition after the following lemma. 
\begin{lemma}\begin{enumerate}
\item The $R$-module  $(\Ind_\Pf^\Gf(\V))^\Uf$ is generated by all
$$f_{\Pf \hat d \Uf, v}$$ where $v$ (resp. $d$) ranges over  $\V^{ \UMf}$ (resp. over $\dist$) and $f_{\Pf \hat d \Uf, v}$ is the unique $\Uf$-invariant function in $\Ind_\Pf^\Gf(\V)$ with support $\Pf \hat d \Uf$ and value $v$ at  a chosen lift $\hat d$ of $d$ in $\norm_\Gf$.  
\item For $v\in \V^{ \UMf}$ and $d\in \dist$, we have $f_{\Pf \hat d \Uf, v}= f_{\Pf, v}\tau_{\hat d}$.\item For $v\in \V^{ \UMf}$ and $w\in \norm_\Mf$, we have $f_{\Pf, v}\tau_w= f_{\Pf, v\tau^{\Mf}_{w}}$. 
\end{enumerate}
\label{Ugene}
\end{lemma}

Recall that for   $W$ a $\Gf$-representation,  $f\in W^\Uf$, $w\in \norm_\Gf$, and $\Uf w\Uf=\sqcup_{x} \Uf w x$ a decomposition in simple cosets with $x$ ranging over $\Uf\cap w^{-1}\Uf w\backslash \Uf$, we have
$f \tau_ w=\sum_{x} (wx)^{-1} f$.

\begin{proof}  (1)  Let $\tilde \V$ denote in this proof the representation $\V$ trivially inflated to a representation of $\Pf$. By Lemma \ref{PI}(1), a function $(\Ind_\Pf^\Gf(\V))^\Uf$ is   a  unique linear combination of $\Uf$-invariant functions with support in $\Pf \hat d \Uf$, for $d\in \dist$.   A $\Uf$-invariant function  with support in $\Pf \hat d \Uf$
is determined by its value at $\hat d$ which is an element  in ${\tilde \V}^{{\Pf}\cap \hat d\Uf \hat d^{-1}}=\V^{\UMf}$ by Lemma \ref{PI}(1)(b).  \\ 
(2) For $v\in \V^{ \UMf}$ and $d\in \dist$, the function $f_{\Pf, v}\tau_{\hat d}$ is $\Uf$-invariant with support in $\Pf \hat d \Uf$ and value at $\hat d$ equal to $v$ because   $\hat d \Uf \hat d^{-1}\cap\Pf\subset \Uf$ by Lemma \ref{PI}(1)(b). \\
(3) Let $w\in \norm_\Mf$. Note that $\Uf  w \Uf=\Uf w \Nf \UMf= \Uf w \UMf$ since $ w\in \Mf$ normalizes $\Nf$. Let $\UMf  w \UMf =\sqcup_{u}\UMf w u$ be a disjoint decomposition in simple cosets with $u$ ranging over $\UMf\cap  w^{-1}\UMf w\backslash \UMf$. Then $\Uf w \Uf=\Uf w\UMf=\cup_{u}\Uf w u$ and the union is disjoint because  $ (w^{-1}\Uf w)\cap\UMf=(w^{-1}\UMf w w^{-1}\Nf w)\cap\UMf=(w^{-1}\UMf w \Nf )\cap\UMf=(w^{-1}\UMf w)\cap\UMf $.
It proves that for $v\in \V^{ \UMf}$, the  $\Uf$-invariant function $f_{\Pf, v}\tau_w$ has support in $\Pf$ and value at $1$ equal to $v\tau_w^{\Mf}$.\\ 
\end{proof}
\begin{proof}[Proof of Proposition \ref{prop:Q1}]  
The map
\begin{equation}\label{naturaliso1}\V^{\Uf_\Mf}\otimes_{\HMf} \Hf\rightarrow (\Ind_{\Pf}^{\Gf}(\V))^\Uf, v\otimes h\mapsto f_{\Pf, v} h \end{equation} is a well defined   morphism of right $\Hf$-modules.
By Proposition \ref{liberte}, the vector space $\V^{\Uf_\Mf}\otimes_{\HMf} \Hf$ decomposes as the direct sum of all 
$\V^{\Uf_\Mf}\otimes \tau_d$ for $d\in \dist$. For $v\in \V^{\Uf_\Mf}$ and $d\in \dist$, the image of $v\otimes\tau_{\hat d}$ is equal to $f_{\Pf \hat d \Uf, v}$ by Lemma \ref{Ugene}(2). The map is bijective by Lemma \ref{Ugene}(1) and its proof.
\end{proof}

\begin{prop} \label{daggercommute} For $\V\in \Rep(\Mf)$, there is a natural isomorphism of representations of $\Gf$:  $$(\Ind_\Pf^\Gf(\V))^\dagger\cong \Ind_\Pf^\Gf(\V^\dagger).$$
 
 \end{prop}

\begin{proof}   It is easy to see that the representation $\Ind_\Pf^\Gf(\V^\dagger)$ is generated by the functions of the form $f_{\Pf, v}$ for $v\in  \V^\UMf$ so in particular it is generated its $\Uf$-invariant subspace. The natural $\Mf$-equivariant  injection  $\V^\dagger\rightarrow \V$ induces, by exactness of the functor $\ii_\Pf^\Gf$, a $\Gf$-equivariant injection 
$\Ind_\Pf^\Gf(\V^\dagger)\rightarrow \Ind_\Pf^\Gf(\V)$ whose image is  contained in  $(\Ind_\Pf^\Gf(\V))^\dagger$. Its image is exactly $(\Ind_\Pf^\Gf(\V))^\dagger$ by points (1) and (2) of  Lemma \ref{Ugene}.
\end{proof}

\bigskip

 The \textbf{answer to Question} \ref{qf2} is positive  when $R$ is an arbitrary ring: \\

The composition 
$$\Rep(\Gf)\rightarrow \Rep(\Mf)\rightarrow \Mod(\HMf), \: \V\mapsto \V^{\Nf}\mapsto (\V^\Nf)^{\UMf}$$
coincides with
$$\Rep(\Gf)\rightarrow \Mod(\Hf)\rightarrow \Mod(\HMf), \: \V\mapsto \V^{\Uf}\mapsto \r_\HMf^\Hf(\V^{\Uf})$$ because
$\Uf$ is the semidirect product  $\UMf\ltimes \Nf$.  
Therefore, the diagram \eqref{diag2} commutes and so does \eqref{diag2'} by adjunction. 
We now describe a natural transformation between
$\ii_\HMf^\Hf(-)\otimes _\Hf\Xf$ and $\ii_\Pf^\Gf(-\otimes _\HMf\XMf)$. Let $\mm$ be a right $\HMf$-module.
There is a natural $\HMf$-equivariant map 
\begin{equation*}
\mm\longrightarrow (\mm\otimes_\HMf \XMf)^{\UMf}, \: m\mapsto m\otimes {\bf 1}_{\UMf}.
\end{equation*} where, for $Y$ a set,  ${\bf 1}_{Y}$ denotes the characteristic function of $Y$.
Using Proposition \ref{prop:Q1} and adjunction, we have a morphism of right $\Hf$-modules
\begin{equation*}
\mm\otimes _\HMf\Hf \longrightarrow (\mm\otimes_\HMf \XMf)^{\UMf}\otimes _\HMf\Hf\cong  (\Ind_{\Pf}^{\Gf}(\mm\otimes_\HMf \XMf))^\Uf
\end{equation*} and a morphism of representations of $\Gf$ 
\begin{equation}\label{eq:Q2}
\ii_\HMf^\Hf(\mm)\otimes _\Hf\Xf=\mm\otimes _\HMf\Xf \longrightarrow   \Ind_{\Pf}^{\Gf}(\mm\otimes_\HMf \XMf)
\end{equation}  sending $m\otimes {\bf 1}_\Uf$ onto the function $f_{\Pf, m\otimes  {\bf 1}_{\UMf}}$ with support in $\Pf$ and value $m\otimes  {\bf 1}_{\UMf}$ at $1$. 
On the other hand, the natural $\HMf$-equivariant map $$\mm\rightarrow\r^\Hf_\HMf((\mm\otimes _\HMf\Xf)^{\Uf})=((\mm\otimes _\HMf\Xf)^\Nf)^{\UMf}$$ induces a morphism of representations of $\Mf$ 
\begin{equation*}
\mm\otimes_\HMf \XMf
 \longrightarrow(\mm\otimes _\HMf\Xf) ^\Nf\end{equation*} which in turn by adjunction (Prop. \ref{finiteadjoints}) induces
 a morphism of representations of $\Gf$
\begin{equation}\label{eq:Q22}
\Ind_\Pf^{\Gf}(\mm\otimes_\HMf \XMf)
 \longrightarrow \mm\otimes _\HMf\Xf= \ii^\Hf_\HMf(\mm)\otimes _\Hf\Xf  \end{equation} 
  and one checks that 
 \eqref{eq:Q2} and  \eqref{eq:Q22} are inverse to each other.
  
\medskip

We study \textbf{Question} \ref{qf3}. We will show the following:
if $R$ is a field of characteristic different from  $p$, then the answer is positive.   (recall that in that case ${\ll_{\HMf}}=\r_{\HMf}^\Hf$ by Remark \ref{L=res});
if $R$ is a field of characteristic   $p$, then the answer is negative.  
\begin{prop}\label{prop:Q3}Suppose that $R$ is a field of characteristic different from  $p$.  Let $\V$ be a representation of $\Gf$. Then we have an isomorphism of $\HMf$-modules
$$(\V_\Nf)^{\UMf}\simeq \r^\Hf_\HMf(\V^\Uf).$$

\end{prop}

\begin{proof} The map  $\nu:\V^\Uf\to (\V_\Nf)^\UMf$, $v\mapsto v\mod \V_\Nf$ is well defined and right $\HMf$-equivariant.
Consider the map $\pi:\V\to \V^\Uf$, $v\mapsto\dfrac{1}{\vert \Uf\vert} \sum_{u\in \Uf} u.v$.
It factors through  $\bar\pi:\V_\Nf\to \V^\Uf$ since $\Nf$ is a subgroup of $\Uf$. Note that for any $v\in \V^\Uf$, we have $v=\bar\pi(v\mod \V_\Nf)$. This proves that the restriction of $\bar\pi$ to $(\V_\Nf)^\UMf$ 
is an inverse for $\nu$ which is therefore bijective.
\end{proof}

Suppose that $R$ is a field of  characteristic $p$.  
 When  $\Pf=\Bf$, we have $\Mf=\Zf$,  $\UMf=\{1\}$, $\Nf=\Uf$  and the Hecke algebra $\Hf_\Zf$ is naturally isomorphic to the group algebra $R[\Zf]$.  Via this identification, the second vertical map of the diagram \eqref{diag3} is the natural equivalence  between $\Rep(\Zf)$ and $\Mod(R[\Zf])$.
Since $\ll^\Hf_{\Hf_\Zf}$  and the $\Uf$-invariants functor are  left exact, the commutativity of the diagram \eqref{diag3} would imply the left exactness of $(-)_\Uf$. This contradicts  Prop. \ref{prop:notexact} when $\Uf\neq\{1\}$..

\section{Parabolic induction for the $p$-adic group and the pro-$p$ Iwahori Hecke algebra}
Let $R$ be a commutative ring. For $\P$ a standard parabolic subgroup of $\G$ with standard Levi decomposition $\P=\M \, \N$ and corresponding to  the subset $J\subseteq \Pi$,  we consider the category $\Mod(\Hh)$ (resp. $\Mod(\Hh_\M)$) of  right $\Hh$-modules (resp. $\Hh_\M$-modules) and the category
 $\Rep(\G)$ (resp. $\Rep(\M)$) of smooth $\k$-representations of $\G$  (resp. $\M$):
 a smooth $\k$-representation of $\G$ is a $\k$-module endowed with a $\k$-linear action of $\G$ such that the stabilizers of the points are open subgroups of $\G$.
\subsection{\label{sec:paraind-padic}Parabolic induction and restriction for the representations of  the $p$-adic group.}

For $\V\in \Rep(\M)$, we consider the representation  of $\G$ on the space of functions $f: \G\rightarrow \V$ such that $f(mng)=m.v$ for all $m\in \M$, $n\in \N$ and $g\in \G$, the action of $\G$  being  by right translation $(g,f)\mapsto f( \, .\,g)$. Its smooth part is denoted by  $\Ind_\P^\G(\V).$  This defines the functor  of \emph{parabolic induction} 
$$ \Ind_{\P}^\G: \Rep(\M)\longrightarrow \Rep(\G).$$
It is  exact and  faithful \cite[this follows from (5) and (6)]{Vigadjoint}. It has a left adjoint:  the $\N$-coinvariant functor
\begin{equation}\label{N-coinv} (-)_\N:\Rep(\G)\longrightarrow  \Rep(\M)\end{equation} which associates to $\V$ the representation of $\M$ on the coinvariant space $\V_\N$ \cite[2.1 Remarque]{Viglivre}. It has a right adjoint \cite[Prop 4.2]{Vigadjoint} which we denote by 
$$\RR^\G_\P:  \Rep(\M)\longrightarrow  \Rep(\G).$$

\subsection{Induction, coinduction and restriction for pro-$p$ Iwahori Hecke modules.\label{indcoind-prop}}

We identify $\Hh_{\M^+}$ and $\Hh_{\M^-}$  with their respective natural images in $\Hh$ via the linear map $\theta$ and $\theta^*$ respectively (\S\ref{LeviHecke}).   
 For $\mm$ a right $\HM$-module, we consider the induced $\H$-module $$\mm\otimes_{\HMpt}\H$$ where the notation 
 $\HMpt$ is a reminder that we identify  $\HMp$ with its image in $\Hh$ via $\theta$.
 We will also denote this module by  $\Ind_{\HM}^\H(\mm)$. 
This defines a functor called \emph{(parabolic) induction} for Hecke modules 
\begin{equation}\label{ind} \Ind_{\HM}^\H:= -\otimes _{\HMpt} \H: \Mod(\HM)\longrightarrow \Mod(\H).
\end{equation}
The space $$\Hom_{\HMmt}(\H, \mm),$$where the notation 
 $\HMmt$ is a reminder  that we identify  $\HMm$ with its image in $\Hh$ via $\theta^*$,
has a structure of right $\H$-module given by $(f,h)\mapsto [f.h: x\mapsto f(hx)]$.  We will also denote this module by  $\cII_{\HM}^\H(\mm)$. 
This defines the functor of \emph{(parabolic) coinduction}  for Hecke modules 
\begin{equation}\label{coind}\cII_{\HM}^\H:= \Hom_{\HMmt}(\H, -): \Mod(\HM)\longrightarrow \Mod(\H).
\end{equation}

 Let $\longw$ denote  as before the longest element of $\Wf$ and $\longw_\Mf$ the longest element in $\Wf_\Mf$. As in Proposition \ref{tracemap},  we pick lifts  $\tw$ and $\twm$ for these elements in $\norm_\Gf$ and $\norm_\Mf$ respectively. Recall (see \eqref{isomorphismW}) that $\norm_\Gf$ (resp. $\norm_\Mf$) identifies with a subgroup of $\W(1)$ (resp. $\W_\M(1)$) and we may therefore consider     $\tw$ (resp. $\twm$)  as a element in $\W(1)$ (resp. $\W_\M(1)\subseteq \W(1)$).
 We have automorphisms of $\k$-algebras
 $$\upiota: \H\rightarrow \H,\: \tau_w\mapsto \tau_{{\tw} w {\tw} ^{-1}} \:\: \textrm{ for }w\in \W(1)\quad\textrm{and } \upiota_\M: \H_\M\rightarrow \H_\M,\: \tau_w\mapsto \tau_{{\twm} w {\twm} ^{-1}} \:\: \textrm{ for }w\in \W_\M(1)$$ whose restrictions  to the finite Hecke algebras, $\Hf$ and $\Hf_\Mf$ respectively, coincide  with $\iota$ and $\iota_\Mf$ defined in Proposition \ref{tracemap}.
 Let $\M':=\longw \M\longw^{-1}$. It is the Levi subgroup of the standard parabolic subgroup $\P'$ corresponding to $\longw.J\subseteq \Pi$. 
Note that  ${\tw}^{-1}\tw$ is a lift for $\longw\longw_\M\in \W$ (the latter is denoted by $w_0^\M$ in \cite{Vig5} above Definition 1.7).
We have an $\k$-algebra isomorphism
$\upiota_\M^{-1} \circ \upiota:\H_{\M'}\overset{\simeq}\longrightarrow \H_{\M}$  (compare with \cite[Prop. 2.20]{Vig5}). It 
induces an equivalence of categories $- \, \upiota_{\M} ^{-1}\upiota: \Mod(\HM)\rightarrow \Mod(\H_{\M'})$ with quasi-inverse  $- \, \upiota ^{-1}\upiota_\M: \Mod(\H_{\M'})\rightarrow \Mod(\HM)$.
By \cite[Theorem 1.8]{Vig5}, we have  an  isomorphism of functors (compare with Prop. \ref{prop:compafunc}):
\begin{equation}\label{twistedind}\Ind _{\HM}^\H\cong \cII^\H_{\H_{\M'}}  (-\, \upiota_{\M} ^{-1}\upiota).\end{equation}
Using \eqref{twistedind} and the classical properties of induction and coinduction (\cite[\S3]{Vig5}), one obtains  (\cite[Thm 1.6 and 1.9]{Vig5}):

 \begin{enumerate}
 \item The induction functor $\Ind_{\HM}^\H$ 
is faithful and exact. It has an exact  left  adjoint  denoted by $\L_{\HM}^\H$
and the functor
\begin{equation}\label{rightadj}\R^\H_\HM:=\Hom_{\Hh}(\HM\otimes_{\HMpt} \H, -)
\end{equation} as a right adjoint. 
\item If $\m$ is a $\HM$-module which is finitely generated over $\k$, then  $\Ind^\H_{\HM}(\m)$ is finitely generated over $\k$.
\item Suppose that $\k$ is a field,  and  let  $\m$ be a $\H$-module which is finite dimensional over $\k$.
The image of  $\m$ by the  left  (resp.  right) adjoint to $\Ind^\H_{\HM}$  is  finite dimensional over $\k$.

\end{enumerate}

Note that the left adjoint of the induction   is exact  because it is given by a localization \cite[Theorem 1.9]{Vig5}. Analogous results about the coinduction functor  hold   (\cite[Cor. 1.10]{Vig5}). 

Let   $\m$ be a $\H$-module. We have $\R^\H_\HM(\m)=\Hom_{\HMpt}(\HM, \m)$ and we recall that $\HM=\HMp[\tau^{-1}]$ is a localization of $\HMp$ at a central element $\tau$ (\S \ref{LeviHecke}).   The map $f\mapsto f(1)$ is an  homomorphism from $\Hom_{\HMpt}(\HM, \m)$  to the submodule $\cap_{n\in \N} \m \,\theta(\tau)^{n}$ of infinitely $\theta(\tau)$-divisible elements of $\m$. When $\m$ a noetherian $R$-module,  this map is bijective (\cite[Lemma 4.11]{AHV}). Therefore  (3) for  the right adjoint  to $\Ind^\H_{\HM}$ is a particular case of:

\hspace{1mm} (4)  {\sl Let   $\m$ be a $\H$-module which is noetherian over $\k$. Then  $\R^\H_{\HM}(\m)$ is finitely generated over $R$.}

\begin{rema}\label{rema:decompind}
We have an isomorphism of $R$-modules  $$ \bigoplus _{d\in \dist} \m\cong  \Ind_\HM^\H(\m)$$ given by 
$(m_d)_{d\in \dist}\mapsto \sum_{d\in \dist} m_d\otimes \tau_{\tilde d}$  (\cite[Rmk. 3.7]{Vig5}, see also \cite[Prop. 5.2]{Oparab} when $\G={\rm GL}_n( \F)$).
\end{rema}

\begin{rema}\begin{enumerate}\item The functor $\L_\HM^\H$ is explicitly described in \cite[Thm. 1.9]{Vig5}: from  \eqref{coind} and \eqref{twistedind} we get an isomorphism 
 \begin{equation}\label{explicitleftadj1} \L_{\HM}^\H\cong  (- \otimes_{\H_{(\M')^-, \theta^*}} \H_{\M'})  \upiota ^{-1}\upiota_\M .\end{equation} 
 
\item 
Abe has recently proved  \cite[Prop. 4.13 (2)]{Abe2}  that the coinduction  functor $\cII_{\HM}^\H  $ is isomorphic to 
$\Hom_{\H_{(\M')^{+}, \theta^*}}( \H, - \upiota_{\M} ^{-1}\upiota )$. Applying \eqref{twistedind}, this means  that $\Hom_{\H_{\M^+, \theta^*}}( \H, -  )$  is isomorphic to the induction functor $\Ind^\H_{\HM}= -\otimes _{\HMpt} \H \simeq \Hom_{\H_{(\M')^{-}, \theta^*}}( \H, -\, \upiota_{\M} ^{-1}\upiota)$ where the notation 
 $\H_{\M^+, \theta^*}$ is a reminder  that we identify  $\HMp$ with its image in $\Hh$ via $\theta^*$.
 This implies that  
we have an isomorphism of functors
\begin{equation}\label{explicitleftadj2}\L_\HM^\H\cong -\otimes _{\H_{\M^+, \theta^*}} \H_\M.\end{equation} \end{enumerate} \label{rema:tensor-hom} 
\end{rema}

\subsection{Commutative diagrams: parabolic induction and the pro-$p$ Iwahori invariants functor}

\subsubsection{\label{prop-inv}The functor of $\I$-invariants  and its left adjoint}
We consider the universal module $\XXM=\ind_\IM^\M(1)$ introduced in \S\ref{prophecke}. It is naturally  a left $\HM$-module  and a  representation in $\Rep(\M)$. Recall that when $\M=\G$, we write simply $\XX=\ind_\I^\G(1)$ instead of $\XX_\G$.
The  functor 
     \begin{align}\label{tens} -\otimes_{\H} \XX:\Mod(\Hh)\to \Rep(\G)  \end{align}
  is the left adjoint of the $\I$-invariant functor 
      \begin{align}\label{inv} \Hom_\G(\XX, -)=(-)^{\I}:\Rep(\G)
\to \Mod(\H).  \end{align}
It is therefore  right exact  and the  functor $(-)^{\I}$ is left exact.

\subsubsection{On the functor of $\I$-invariants in characteristic $p$\label{onfunc}}  

Suppose that $R$ is an algebraic closure of the residue field of $\F$. The functor \eqref{inv} has been studied in the case of ${\bf G}={\rm GL}_2$ and in the case of $\G={\rm SL}_2$ (\cite{Oinv}, \cite{Koz}). It yields an equivalence between the category of representations of ${\bf G}={\rm GL}_2(\mathbb Q_p)$ (resp. $\G={\rm SL}_2(\mathbb Q_p)$) generated by their $\I$-fixed vectors and the category of $\H$-modules. This equivalence fails when $
\G={\rm GL}_2(\F)$ for $\F\neq \mathbb Q_p$. 
Although it has not been studied in general, it is expected to fail for groups of higher rank: for example, the functor \eqref{tens}  is not exact when ${\bf G}={\rm GL}_3$ and $p\neq 2$ (\cite{OSepreprint}).
 
\begin{rema}
  For $\pi\in \Rep(\G)$ admissible,
$\pi$ irreducible does not imply $\pi^\I$ irreducible. Counter examples are constructed in  \cite{BP} in the case of $\G={\rm GL}_2( \F)$ where $ \F$ is a strict  unramified extension of $\mathbb Q_p$.

\end{rema}

\subsubsection{Commutative diagrams: questions}  We study the connection between  the functors  of induction for Hecke modules and for representations via the functors \eqref{tens} and \eqref{inv} and their analogs for $\M$. We examine the following questions. 
 \begin{question}\label{qp1}Does the  $\I$-invariant  functor commute with  parabolic induction: is the following diagram commutative?
 \begin{equation}\begin{CD}\Rep(\M) @>{(-)^{\IM}}>>  \Mod(\HM)\cr
{\Ind_\P^\G}@VVV  {\Ind_\HM^\H }@VVV \cr
 \Rep(\G) @>{(-)^{\I}}>>        \Mod(\H)\cr
\end{CD}\label{diag1prop}\end{equation}
 We answer this question positively in \S\ref{answersprop} (Proposition \ref{prop:Q1prop}). \\
\end{question}
 \begin{question}\label{qp2}  Does the $\I$-invariant functor commute with  the right adjoints of parabolic induction:
is the following diagram commutative?

 \begin{equation}\begin{CD}\Rep(\G) @>{\hspace{13mm}\bf R_\P^\G\hspace{13mm}}>> \Rep(\M)\cr
{(-)^{\I}}@VVV {(-)^{\IM}}@VVV\cr
 \Mod(\H) @>{\R^\H_\HM=\Hom_{\Hh}(\HM\otimes_{\HMpt} \H, -)}>>        \Mod(\H_\M)\cr
\end{CD}\label{diag2prop}\end{equation}
Passing to left adjoints, this is equivalent to asking whether the following diagram is commutative:

 \begin{equation}\begin{CD}\Mod(\HM) @>{-\otimes_{\HM} \XXM}>>  \Rep(\M)\cr
{\Ind^\H_{\HM}}@VVV  {\Ind^\G_{\P }}@VVV \cr
 \Mod(\H) @>{\hspace{3mm}-\otimes_{\H} \XX\hspace{3mm}}>> \Rep(\G)\cr
\end{CD}\label{diag2'prop}\end{equation}

We answer this question positively in \S\ref{answersprop} (Corollary \ref{coro:Q2}).
In passing, we  will 
consider the functor $\dagger: \Rep(\M)\rightarrow \Rep(\M)$ sending a representation $\V$ onto the subrepresentation generated by  the $\IM$-fixed subspace $\V^\IM$ (and likewise $\dagger: \Rep(\G)\rightarrow \Rep(\G)$). We show (Corollary \ref{daggercommuteprop}) that $\Ind_\P^\G$ and $\dagger$ commute.\\

\end{question}

\begin{question}\label{qp3}  Does the $\I$-invariant functor commute with  the left adjoints of parabolic induction:
is the following diagram commutative?
 \begin{equation}\begin{CD}\Rep(\G) @>{\hspace{9mm}(-)_\N\hspace{9mm}}>> \Rep(\M)\cr
{(-)^{\I}}@VVV  {(-)^{\I_\M}}@VVV \cr
 \Mod(\H) @>{\L_\HM^\H }
>>        \Mod(\H_\M)\cr
\end{CD}\label{diag3}\end{equation}

 We recall that the left adjoint to $\Ind_\HM^\H$ is described in Remark \ref{rema:tensor-hom}.  The answer is negative in general as we show in Corollary \ref{contreex} where we work with $\G={\rm GL}_2(\mathbb Q_p)$, $\M=\T$ and $R$ is a field of characteristic $p$.  But when $R$ is a field of characteristic different from $p$,   we prove in Proposition \ref{prop:rem??}  that
 the answer is positive. \end{question}


\subsubsection{\label{answersprop}Commutative diagrams: answers}

The \textbf{answer to Question}  \ref{qp1} is positive.  We prove that the diagrams \eqref{diag1} and \eqref{diag2} commute by proving the following result. 

\begin{prop}\label{prop:Q1prop}  For any representation $\V\in \Rep(\M)$, there is a natural isomorphism of right $\H$-modules
$$\Ind_\HM^\H(\V^\IM)=\V^{\IM}\otimes_{\HMp} \H\cong (\Ind_\P^{\G}(\V))^\I.$$

\end{prop}

 The isomorphism of Proposition \ref{prop:Q1prop} is given by the map \eqref{naturaliso1prop} below. We prove the proposition after the following lemma. 
 Here we consider $\dist$ as a subset of $\W$ and fix a lift $\hat d\in \norm_\G\cap\K$ for each $d\in\dist$.
 \begin{lemma}\label{Igene}Let $\V$ be a representation  of $\M$. \begin{enumerate}
\item The $R$-module  $(\Ind_\P^\G(\V))^\I$ is generated by all
$$f_{\P \hat d \I, v}$$ where $d$ (resp.  $v$) ranges over $\dist$ (resp. $\V^{ \IM}$) and
$f_{\P  \hat d \I, v}$ denotes the unique $\I$-invariant function in $\Ind_\P ^\G (\V)$ with support $\P  \hat d \I$ and value $v$ at $\hat d$.\item For $v\in \V^{ \IM}$ and $d\in \dist$, we have $f_{\P \hat d \I, v}= f_{\P\I, v}\tau_{\hat d}$.
\item For $v\in \V^{ \IM}$ and $w\in \W_\Mp(1)$, we have $f_{\P\I, v}\,\tau_w= f_{\P\I, v\tau^\M_{w}}$.\end{enumerate}
\end{lemma}

\begin{proof} These results are proved for $\G={\rm GL}(n, \F)$ in \cite[6A 1,2 and 3]{Oparab}  and  for $\G$ split in  \cite[ \S 5.5.1 and 5.5.2]{Compa}. We recall the arguments here and check that they are valid for general $\G$.
They are similar to the ones in the proof of Lemma \ref{Ugene}.

 (1)  By Lemma \ref{PI}(2), a function $(\Ind_\P^\G(\V))^\I$ is a unique  linear combination of $\I$-invariant functions with support in $\P \hat d \I$, for $d\in \dist$.  Denote in this proof by $\tilde\V$ the representation $\V$ trivially inflated to a representation of $\P$. A $\I$-invariant function  with support in $\P \hat d \I$
is determined by its value at $\hat d$ which is an element  in ${\tilde\V}^{\P\cap d\I d^{-1}}={\tilde\V}^{\IM\N}=\V^\IM$ by Lemma \ref{PI}(2)(c).\\\\
Now recall that for   $X$ a representation of $\G$,   and $f\in X^\I$, $w\in \W(1)$, if $\I \hat w\I=\sqcup_{x} \I\hat w x$ is a decomposition in simple cosets with $x$ ranging over $(\I\cap \hat w^{-1}\I\hat w)\backslash \I$, then we have
$f \tau_ w=\sum_{x} (\hat wx)^{-1} f$. Using this we check (2):
for $v\in \V^{ \IM}$ and $d\in \dist$, the function $f_{\P\I, v}\tau_{\hat d}$ is $\I$-invariant with support in $\P \hat d \I$ and value at $\hat d$ equal to $v$ because   $\hat d \I \hat d^{-1}\cap\P\I\subset \I$ by Lemma \ref{PI}(2)(b).\\
Point (3) is a  direct consequence of Lemma \ref{lemma:cosets}(1). \end{proof}

\begin{proof}[Proof of Proposition \ref{prop:Q1prop} ] Analogously to  the proof of Proposition \ref{prop:Q1}, consider 
the  morphism of right $\H$-modules
\begin{equation}\label{naturaliso1prop}\V^{\IM}\otimes_{\HMp} \H\rightarrow (\Ind_{\P}^{\G}(\V))^\I, \quad v\otimes h\mapsto f_{\P\I, v} \, h.
\end{equation}
It is well defined by Lemma  \ref{Igene}(3). By Remark \ref{rema:decompind}, the vector space $\V^{\IM}\otimes_{\HMp} \H$ decomposes as the direct sum of all 
$\V^{\IM}\otimes \tau_{\hat d}$ for $d\in \dist$. For $v\in \V^{\Uf_\Mf}$ and $d\in \dist$, the image of $v\otimes\tau_{\hat d}$ is equal to $f_{\P \hat d \U, v}$ by Lemma \ref{Igene}(2). This proves that the map is bijective by Lemma \ref{Igene}(1) and its proof.
\end{proof}
We now  \textbf{answer Question}  \ref{qp2} positively by proving that the diagram \eqref{diag2'prop} commutes. 
  Let $\m$ be a right  $\HM$-module. The  natural morphism of right $\HM$-modules $\m \rightarrow ( \m \otimes_{\HM} \XXM)^{\IM},$  $m\mapsto m\otimes {\bf 1}_{\IM}$ induces a morphism of right $\H$-modules
$$\Ind^\H_{\HM}(\m )=\m \otimes_{\HMpt}\H \rightarrow  (\m \otimes_\HM \XXM)^{\IM}\otimes _{\HMpt}\H\cong ({\Ind}_\P ^\G\:\:( \m \otimes_\HM \XXM))^{\I }.$$ 
(the isomorphism    comes from Proposition \ref{prop:Q1prop}).   By adjunction, this in turn  induces  a  $\G$-equivariant map
\begin{equation}\label{parabparah} 
\Ind^\H_{\HM}(\m )\otimes_\H\XX \longrightarrow {\Ind}_\P ^\G\:\:( \m \otimes_\HM \XXM)
\end{equation} sending $m\otimes {1}_{\I }\in \m \otimes_{\HMpt}\XX$  to $f_{\P \I ,m\otimes  {1}_{\IM}}$ with the notation in Lemma \ref{Igene}.  
Our goal is to prove that \eqref{parabparah} is an isomorphism which will prove that  \eqref{diag2'prop} commutes. \\

  In the case when $\m=\HM$, note  that the space  ${\Ind}_\P^\G\:\:( \XXM)$ is naturally endowed with a structure of  left $\HM$-module which commutes with the action of $\G$:  the image $hf$ of $f\in  {\Ind}_\P^\G\:\:( \XXM)$ by  $h\in \HM$   is given by $ (hf)(g)= hf(g)$ for $g\in \G$.  In this case, 
 \eqref{parabparah} is the  $\G$-equivariant and left $\HM$-equivariant map
\begin{equation}\label{parabparahuniv}\HM \otimes_{\HMpt}\XX\longrightarrow {\Ind}_\P ^\G\:\:(\XXM) \end{equation} 
sending $1\otimes {1}_{\I }\in \HM \otimes_{\HMpt}\XX$  to $f_{\P \I ,  {1}_{\IM}}$.

\begin{prop} \label{prop:Q2univ} The map \eqref{parabparahuniv}  is an isomorphism of representations of $\G$ and of left $\HM$-modules.

\end{prop}
Before proving the proposition, we draw two corollaries. 
\begin{coro} For any $\m\in\Mod(\HM)$, the $\G$-equivariant map \eqref{parabparah} is an isomorphism.

\label{coro:Q2}
\end{coro}

\begin{proof}[Proof of Corollary \ref{coro:Q2}] By \cite[Lemma 7.7]{HVcomp}, the canonical $\G$-equivariant map
\begin{equation}\label{tensoparab}\m \otimes_{\HM}{\Ind}_\P ^\G\:\:(\XXM)\longrightarrow {\Ind}_\P ^\G\:(\m\otimes_\HM\XXM)
\end{equation}  is an isomorphism. 
Tensoring the isomorphism \eqref{parabparahuniv} by $\m$ and  composing with 
 \eqref{tensoparab} gives the composite  isomorphism
$$\m \otimes_{\HMpt}\XX\overset{\simeq}\longrightarrow \m\otimes _\HM{\Ind}_\P ^\G\:\:(\XXM)\overset{\simeq}\longrightarrow {\Ind}_\P ^\G\:\:(\m\otimes_\HM\XXM)$$ which coincides with \eqref{parabparah}.  
\end{proof}

The notation $\dagger$ was introduced after \eqref{diag2'prop}. We have:
\begin{coro} \label{daggercommuteprop} For $\V\in \Rep(\M)$, there is  an isomorphism of representations of $\G$:
 $$(\Ind_\P^\G(\V))^\dagger\cong \Ind_\P^\G(\V^\dagger).$$
The representation $\V\in \Rep(\M)$ is generated by its $\IM$-invariant subspace if and only if the representation $\Ind_\P^\G(\V)\in \Rep(\G)$ is generated by its $\I$-invariant subspace.
\end{coro}
\begin{proof} 
Let $\V\in \Rep(\M)$.\\
$\bullet$ The  representation $(\Ind_\P^\G(\V))^\dagger$ is the image  of the  natural $\G$-equivariant morphism
\begin{equation}\label{m1}(\Ind_{\P}^{\G}(\V))^\I\otimes_{\H} \XX \to \Ind_\P^\G(\V),\quad  f\otimes {{\bf 1}_\I}\mapsto f,\end{equation} corresponding by adjunction to the identity map of $(\Ind_{\P}^{\G}(\V))^\I$.  By Proposition \ref{prop:Q1prop}, the $\H$-equivariant map $$\V^\IM\otimes_{\HMpt}\H\rightarrow (\Ind_{\P}^{\G}(\V))^\I$$ sending $v\otimes 1$ onto $f_{\P\I,v}$ is an isomorphism. It induces by adjunction  an isomorphism $\V^\IM\otimes_{\HMpt}\XX\rightarrow (\Ind_{\P}^{\G}(\V))^\I\otimes_{\H} \XX$. Composing the  latter with \eqref{m1} gives the $\G$-equivariant map 
\begin{equation}\label{m1'}\V^\IM\otimes_{\HMpt}\XX \to \Ind_\P^\G(\V),\quad  v\otimes {{\bf 1}_\I}\mapsto f_{\P\I, v}.\end{equation} which has image $(\Ind_\P^\G(\V))^\dagger$.
\\$\bullet$ The  representation $\V^\dagger$ is the image  of the  natural $\M$-equivariant morphism  
\begin{align}\label{m2}\varphi_\V: \V^{\IM}\otimes_{\HM} \XXM\to \V, \quad v\otimes {\bf 1}_\IM\mapsto v\end{align} 
corresponding by adjunction to the identity map of $\V^{\IM} $.  By right exactness of the functor $\Ind_\P^\G$, the map \eqref{m2} induces  a $\G$-equivariant map \begin{equation}\label{ind-m2} \Ind_\P^\G(\varphi_\V):\Ind_\P^\G(\V^{\IM}\otimes_{\HM} \XXM)\to \Ind_\P^\G(\V), \quad f \mapsto \varphi_\V \circ f.
\end{equation} which has image $\Ind_\P^\G(\V^\dagger)$. By  Corollary \ref{coro:Q2}, we have a $\G$-equivariant isomorphism
$\V^{\IM}\otimes_{\HMp,\theta} \XX \cong \Ind_\P^\G(\V^{\IM}\otimes_{\HM} \XXM)$ that sends $v\otimes {\bf 1}_\I$ onto 
$f_{\P\I, v\otimes {\bf 1}_\IM }$. Composing the latter with $\Ind_\P^\G(\varphi_\V)$ gives the  $\G$-equivariant map 
\begin{equation}\label{m2'}\V^{\IM}\otimes_{\HMpt} \XX \to \Ind_\P^\G(\V), v\otimes {\1_{\bf \I}}\mapsto f_{\P\I, v}.\end{equation} which has image  $\Ind_\P^\G(\V^\dagger)$.\\
$\bullet$ Since  \eqref{m1'} and \eqref{m2'} coincide, we proved that the two  subrepresentations $\Ind_\P^\G(\V^\dagger)$ and $(\Ind_\P^\G(\V))^\dagger$ of $\Ind_\P^\G(\V)$ coincide.
This will implies the second statement  $\Ind_\P^\G(\V^\dagger)=\Ind_\P^\G(\V)$ is equivalent to $\V^\dagger=\V$,  as
the functor $\Ind_\P^\G$ is exact and faithful \cite[this follows from (5) and (6)]{Vigadjoint}.  \end{proof}

We now prove Proposition \ref{prop:Q2univ} via the  following lemmas. First we prove the surjectivity of 
 \begin{equation} \tag{\ref{parabparahuniv}}\HM \otimes_{\HMpt}\XX\longrightarrow {\Ind}_\P ^\G\:\:(\XXM),   \ 1\otimes {1}_{\I } \mapsto f_{\P \I ,  {1}_{\IM}}\end{equation} in Lemma \ref{ppsurj}. Its injectivity is proved in  Lemma \ref{lem1}.
\begin{lemma}\label{ppsurj} The map \eqref{parabparahuniv}  is surjective.
\end{lemma}

\begin{proof}This is a direct consequence of  \cite[Lemma 6.3]{HVcomp} (which is valid for smooth representations over a commutative ring) as we explain now.  Let $t\in \T$ be a central element in $\M$ satisfying $-\alpha (\tilde \nu (t))>0$ for all $\alpha \in \Sigma^+-\Sigma_\M^+$ (see Remark \ref{lemma:Mp}). We denote  its by $\mu_\Mp$ its  image in $ \Lambda_{ \M^+}(1)\subset \W_{\M^+}(1)$.
 Then
${\Ind}_\P ^\G( \XXM)$ is generated as a representation of $\G$ by  the functions $f_{n}$, $n\in \mathbb Z$ with support  
$\P \I=\P \I_{\Nm}$ and  respective value $t^{-n}{\bf 1}_{\IM }$ on 
$ \I_{\Nm}$. Note that such a function is $\I$-invariant: for $k=k^+k_0k^-\in \I=\I_\N \IM \I_{\Nm}$, we have
$f_n(k)=k_0. t^{-n}{\bf 1}_{\IM }=t^{-n}{\bf 1}_{\IM }=f(1)$ since $t$ in central in $\M$. 
  Since $t^{-n}{\bf 1}_{\IM }=(\tau^\M_{\mu_{\Mp}})^{n}$ seen as an element of $\XXM$, we have in fact  $f_n= f_{\P \I, (\tau^\M_{\mu_{\Mp}})^{n}} =(\tau_{\mu_{\Mp}}^\M )^{n}f_{\P \I, {\bf 1}_\IM}$.  
By definition, this element is in the image of \eqref{parabparahuniv}. \end{proof}


\begin{lemma} 
\label{lem1}
\begin{enumerate}
\item The set $\I\Mp \K$ is the disjoint union of all $\I \Mp \hat d \I$ for $d\in \dist$.
\item 
Let $\mathbf Y$ be the set the functions in $\XX$ with support in  $\I\Mp \K$.  The restriction  to $\mathbf Y$ of the $\G$-equivariant map $$F:\XX\longrightarrow {\Ind}_\P ^\G(\XXM), \: {\bf 1}_\I\longmapsto f_{\P \I, {\bf 1}_\IM}$$ is injective. 
\item  For any element $x\in \XX$, there is $r\in \N$ such that $\tau_{\mu_\Mp}^r x\in \mathbf Y$, where $\mu_\Mp\in \W_{\M^+}(1)$ was defined in the proof of Lemma  \ref{ppsurj}.
\item The map \eqref{parabparahuniv}
 is injective.
\end{enumerate}
\end{lemma}

  \begin{proof}[Proof of Lemma \ref{lem1}] Proof of (1).  Compare this proof with \cite[Lemma 2.4]{Oparab}. Recall that  $\W_0'(1)=\norm_\G\cap \K/\Zp^1\subset \W(1)$ and  we denote by $\W_{\M,0}'(1)$ the subgroup $(\norm_\M\cap \K)/\Zp^1$ of  $\W(1)$. 
We may see  $\dist$  defined in Lemma  \ref{lemmaDJ}  as a system of representatives of  the right cosets
  $\W_{\M,0}'(1)\backslash \W_0'(1)$.
  We have 
$ \I \Mp  \K= \I \W_\Mp(1)  \I \W_0'(1) \I =   \I \W_\Mp(1)  \I_\M  \W_0'(1) \I$.  The last equality   follows from 
 $\I=  \I _\N  \I _\M \I _\Nm$, from $k \I _\Nm k^{-1} \subset \I$ for $k\in \K$  and the fact that $\W_\Mp(1)$ normalizes $  \I _\N $.   Hence

 $\I \Mp \K=\cup_{d\in  \dist}   \I \W_\Mp(1)  \I_\M \W_{\M,0}'(1) \hat d \I   =  \cup_{d\in  \dist} \I \M^+\hat d \I $.

  To prove that the union is disjoint, we first write $\I \Mp \hat d\I=\I  \W_\Mp(1) \I \hat  d\I$   for  $d\in \dist$. By Lemma \ref{lem:addlengths}, the latter is equal to $\I   \W_\Mp(1)  \hat d\I$.   
  This proves that  $\I \Mp \hat d\I$ is the disjoint union of all $\I  w  \hat d\I$ for $w\in   \W_\Mp(1)$ and that the 
sets   $\I \Mp \hat d\I$  are disjoint when $d$ ranges over $\dist$ . \\\\ 
 Proof of (2). By (1), the space $\mathbf Y$ decomposes as the direct sum of the subspaces $\mathbf Y_d$ of the functions with support in $\I \Mp \hat d\I$, for $d\in \dist$. 
  The image of  $\mathbf Y_d$ under $F$ is contained in the subspace of the functions in $\Ind_\P^\G(\XXM)$ with support in $\P \I \Mp \hat d\I$. The set  $\P \I \Mp \hat d\I=\P \I  \W_\Mp(1) \I \hat d\I$ is $ \P \I  \hat d\I$ by  Lemma \ref{lemma:cosets}(1).  Therefore, $F(\mathbf Y_d)$ is contained in the subspace $ \Ind_\P^ {\P \hat d \I}(\XXM)$ of the functions in $\Ind_\P^\G(\XXM)$ with support in $\P \hat d \I$.  Using Lemma \ref{PI}(2), the restriction to $\mathbf Y$ of   $F$ is injective if and only if its restriction to each $\mathbf Y_d$ is injective.
 
   Let $d\in \dist$.  By     Lemma \ref{lemma:cosets}(3) (and  the proof  of point (1) above), the set  $\I \Mp \hat d\I$ is  the disjoint union $\sqcup_{w}  \I \hat  w\hat d  \I   =
  \sqcup_{w, x, u_{x}, y }  \I   {\hat w} x  u_x {\hat d } y$ for 
  \begin{equation}\label{ranges}w\in  \W_\Mp(1), \ x\in \hat w^{-1} \IM  \hat w\cap  \IM \backslash  \IM , \ u_{x}\in (\hat wx)^{-1} \I  \hat wx\cap \I_\Nm\backslash \I_\Nm, \ y\in \I\cap  {\hat d }^{-1}\I  {\hat d }\backslash \I.\end{equation} We may write an element $g\in \mathbf Y_d$ as a $R$-linear combination   $$g=\sum_{y\ } g_y, 
\quad\text{with} \  g_y= \sum_{w,x, u_{x}}\lambda_{w,x, u_{x}, y} {\bf 1}_{\I  \hat w x  u_x {\hat d } y}, \quad \lambda_{w,x, u_{x}, y} \in R$$ with $w,x,u_x$ as in \eqref{ranges}. 
  By Lemma \ref{lemma:cosets}(2), the set 
  $\P \hat d \I$ is the disjoint union  $\sqcup_{y} \P \I \hat d y$ and note that $F(g_y)$ is the component of $F(g)$ with support in  $\P \I \hat d y$. Suppose that $F(g)=0$. Then $F(g_y)=0$ for any $y $.   We admit for a moment that for $u\in \I_\Nm$ and $w\in \W_\Mp(1)$, $x\in \IM$, $u_x\in \I_\Nm$, $y\in\I$  we have  the following formula  \begin{equation}\label{calc}[F({\bf 1}_{\I  \hat w x  u_x {\hat d } y})](u\hat{d} y ) = \begin{cases} {\bf 1}_{\IM \hat wx}&\textrm{if $\I \hat wxu_x=\I \hat wxu$}\cr 0&\textrm{otherwise.} \end{cases}\end{equation}
  Then  for  $y$ as above and $u\in \I_\Nm$ we have
$[F(g_y)](u \hat d y)=\sum_{w,  x}\lambda_{w, x, u_x, y}{\bf 1}_{\IM \hat wx} $   for 
  $w\in  \W_\Mp(1)$, $x\in \hat w^{-1} \IM  \hat w\cap  \IM \backslash  \IM$ , and $u_{x}$ the only element in  $(\hat wx)^{-1} \I  \hat wx\cap \I_\Nm\backslash \I_\Nm$ such that $\I \hat  wxu_x=\I \hat wxu$.
This proves that  $\lambda_{w, x, u_x, y}=0$ for  all $w, x, u_x, y$.\\
{Proof of \eqref{calc}:}
Recall that  $f:=F({\bf 1}_{\I  \hat w x  u_x {\hat d } y})= f_{\P\I, {\bf 1}_\IM}( . (\hat w x  u_x {\hat d } y)^{-1})$
has support in $\P\I  \hat w x  u_x {\hat d }y \subseteq \P\I \hat d y$ by  Lemma \ref{lemma:cosets}(1), and we have $f(u\hat{d} y)= f_{\P\I, {\bf 1}_\IM}( u(\hat w x  u_x )^{-1})$. If   $\I \hat wxu_x=\I \hat wxu$, then  $\hat w x  u_x{u}^{-1}\in \I \hat w x$ and $f( u \hat dy)= f_{\P\I, {\bf 1}_\IM}((\hat w x  u_x{u}^{-1})^{-1})= f_{\P\I, {\bf 1}_\IM}((\hat w x)^{-1})=(\hat w x)^{-1}{\bf 1}_\IM={\bf 1}_{\IM \hat wx}.$ If $\I \hat wxu_x\neq \I \hat wxu$, then $u\notin \P \I \hat w x u_x$ by Lemma \ref{lemma:cosets}(1) and therefore $f(u\hat d y)=0$. This proves  \eqref{calc}.\\\\  Proof of (3).
The idea of the argument comes from  \cite[Lemma 12 p.80]{SS91} where $\G={\rm GL}(n,F)$. A version of it in the case when  $\G$ is split can be found in  \cite[Lemma 2.20]{Her} and in the general case in  \cite[Lemma 6.2]{HVcomp}. 
From the proof of the latter we get that $(\G, \Iw, \norm_\G)$   is a generalized Tits system. It follows that  $\G= \Iw \Zp \K=\I \Zp \K$ and it is enough to prove the statement for $x={\bf 1}_{\I z}$ for $z\in \Zp$.
It also follows that there is a finite subset set $A \subset \norm_\G$   such that  for any $z'\in \Zp$  we have 
$$z' \I z\subset z'\Iw z\subset \cup_{a\in A}\Iw z' a\Iw\subset  \cup_{a\in A}\Iw z' a \K.$$
Since $\K$ is special, for any $a\in \norm_\G$, there is $z_a\in \Zp$ such that $ a\K=z_a \K$. 
Pick $r\geq 1$ such that $t^r z_a\in \Mp$ for any $a\in A$, where $t\in \T\subseteq \Zp$ is a lift for $\mu_\Mp$ as in the proof of Lemma \ref{ppsurj}. Then  we have
$$t^r \I z\subset \cup_{a\in A}\Iw  t^r z_a\K\subset \Iw  \Mp\K=\I \Mp\K.$$
Now $\tau_{\mu_\Mp}^r {\bf 1}_{\I z} ={\bf 1}_{\I t^r \I z}$ so  $\tau_{\mu_\Mp}^r {\bf 1}_{\I z}\in \bf Y$.

Proof of (4). By definition, the element $\mu_\Mp$ is central in $\W_\M(1)$ and  $\tau_{\mu_\Mp}^\M$ is invertible in $\HM$.   Consider an element in $\HM\otimes _{\HMp,\theta}\XX$. By  (3) it can be written   $(\tau^\M_{\mu_\Mp})^{-r}\otimes y$ for $y\in \mathbf Y$  and $r\in \N$ large enough. 
Its image by  the left $\HM$-equivariant map \eqref{parabparahuniv} is  $(\tau^\M_{\mu_\Mp})^{-r} F(y)$. Suppose that this image is zero, then  $F(y)=0$ and by (2), we have $y=0$.  This concludes the proof of (4).
  \end{proof}

The \textbf{answer to Question} \ref{qp3} is  positive when $R$ is a field of characteristic different from $p$.
\begin{prop}\label{prop:rem??}(i) When  $p$ is invertible in $R$,  $$\L_\HM^\H\cong   - \otimes _{\H_{\M, \theta^-\delta_\P}}\H_\M,$$   where the notation  $\H_{\M, \theta^-\delta_\P}$ is a reminder that we identify $\H_\M$ with its image in $\H$ via    $\theta^-\delta_\P$. 

(ii) When $R$ is a field of characteristic different from $p$ and 
 $V\in \Rep (G)$  we have $(V_\N)^{\I_\M} \simeq V^{\I}  \otimes _{\H_{M, \theta^-\delta_\P}}\H_M$.
\end{prop}
\begin{proof}  When $p$ is invertible in $R$ we have $\theta^{*,+} =  \theta^-\delta_\P$ by Lemma \ref{lemma:thetacompar}  and $\L_\HM^\H\cong -\otimes _{\H_{\M, \theta^{*,+} }} \H_\M $ by \eqref{explicitleftadj2}. 
 When $R$ is a field of characteristic different from $p$, the isomorphism $(V_\N)^{\I_\M} \simeq V^{\I}  \otimes _{\H_{\M, \theta^-\delta_\P}}\H_\M$ is a particular case of \cite[II.10.1 3)]{Selecta}.
\end{proof}

 The \textbf{answer to Question} \ref{qp3} is  negative when $R$ is a field of characteristic $p$. The existence of a counter-example was suggested to us by Noriyuki Abe.
 
 \begin{prop}\label{cex}
   When $\G={\rm GL}(2, \mathbb Q_p)$ and $R$ is a field of characteristic $p$, there exists an extension $\pi$ of the trivial representation $\Triv_\G$ by the Steinberg representation $\St_\G$ which is not generated by its $\I$-invariants.     
 \end{prop}

   \begin{coro}    \label{contreex} Let $\G={\rm GL}(2, \mathbb Q_p)$.  Suppose that $R$ is a field of characteristic $p$. There is no functor $\mathbf F:\Mod(\Hh)\rightarrow \Mod(\Hh_\T)$ such that   $\mathbf F(\V^\I)= (\V_\U)^\IM$ for any representation $\V\in \Rep(\G)$.
     \end{coro}

   \begin{proof}[Proof of the corollary] Here we work with the parabolic subgroup $\B= \T\U$.  
    Consider the  representation $\pi$ of the proposition. We  have 
 $\pi^{\I}= \St_\G^\I$ but $\pi_\U\neq \{0\}$ and $(\St_\G)_\U=\{0\}$. 
 Indeed, recall that $\St_\G$ is irreducible \cite[Thm. 4.7]{VigGL2} (and therefore generated by $\St^\I$), so the natural inclusion $\St_\G^\I\subset \pi^{\I}$ is an equality   as $ \pi^{\I}$ does not generate $\pi$ and $\pi/ \St$ is one dimensional. Now 
 $(\Triv_\G)_\U$ is one dimensional  and the $\U$-coinvariant functor is right exact so
 $\pi_\U \neq \{0\}$.  The restriction of $\St_\G$ to  $\U$ is  the space $C_c^\infty (\U,R)$ of locally constant functions  with compact support in $\U$ with the natural action of $\U$ \cite[Lemma 3]{VigSP}. The group $\U$ does not have a Haar measure with value in $R$ of characteristic $p$ (since $\U$ does not contain a compact open subgroup of  pro-order prime to $p$) hence $(\St_\G)_\U=\{0\}$.  
 \end{proof}

\begin{proof}[Proof of Proposition \ref{cex}]   Let $R$ be a field of characteristic $p$.
The
 one dimensional $\H$-modules  $\sign_\H$ and $\Triv_\H$  have been described in Remark \ref{rema:idempo}. They are respectively isomorphic to  $\St_\G^\I$ and $\Triv_\G^\I$  (\cite[4.4]{VigGL2}). 
 We have    $\sign_\H(\tau_w)=(-1)^{\ell(w)}$,   $\Triv_\H(\tau_w)=1$  if  $\ell(w)=0$ and  $\Triv_\H(\tau _w)=0$  if  $\ell(w)>0$  for any $w\in  \W(1) $  since $R$ has characteristic $p$,.
 
 First we suppose that   $R$ is  an algebraic closure of the residue field of $\F$.

The existence of a representation $\pi$ sitting in an exact sequence of the form 
$0\rightarrow \St_\G\rightarrow \pi\rightarrow  \Triv_\G\rightarrow 0$ and not generated by its   $\I$-invariants follows from the following facts:  

 \begin{enumerate}
  
  \item   The $\I$-invariant functor is an equivalence  from the  subcategory of representations of $\G$ generated by their $\I$-invariants to  $ \Mod(\H)$ with inverse $-\otimes \bf X$ (\S\ref{onfunc}).
  
\item 
 $\dim_R \Ext^1_{G}(\Triv_\G, \St_\G)>1 $  in the category  $\Rep(\G)$. 
    
\item    $\dim _R \Ext^1_{\mathcal H}(\Triv_\H, \sign_\H)=1$ in the category $\Mod(\H)$. 
\end{enumerate}

 The property  (2) follows from \cite[Thm. VII.4.18]{Colmez} (there is no symetry, $\dim_R \Ext^1_{G}(\St_\G,\Triv_\G )=1 $ by \cite[Prop.4.3.30(ii)]{Emerton}).  Colmez works over a finite subfield of $R$, but  $R$ being the union of its finite subfields,  the results of Colmez are valid over $R$. A smooth extension $E$ of $\Triv_\G$ by $\St_\G$ gives by (non smooth) duality an extension $E^\vee$ of  
  the (non-smooth) dual $\St_\G^\vee$ by $\Triv_\G$.   Let $A^+=\begin{pmatrix} \mathbb Q_p^\times&0\cr 0&1\end{pmatrix}$.  As $\dim_R(\St_\G^\vee)^{A^+}=1$  \cite[Lemma. VII.4.16]{Colmez}, the inverse image of 
 $(\St^\vee)^{A^+} $ in $E^\vee$ gives a representation of $A^+$ which is an extension of  $\Triv_{A^+}$ by $\Triv_{A^+}$. Such extensions are classified by the group $\Hom (\mathbb Q_p^\times, R)$ of smooth (i.e. locally constant) group homomorphisms $\mathbb Q_p^\times\rightarrow R$. This construction induces a  map
 $$res_{A^+}:\Ext^1_{\G}(\Triv_\G, \St_\G)\to \Hom (\mathbb  Q_p^*, R).$$
For each $\tau\in \Hom (\mathbb  Q_p^*, R)$, Colmez constructs  a smooth extension $E_\tau$ of  $\Triv_\G$ by $\St_\G$ with $res_{A^+}(E_\tau)=\tau$.  This implies  $\dim_R \Ext^1_{\G}(\Triv_\G, \St_\G)>1 $.  Note that Colmez shows that $E_\tau$ admits a central character and that $res_{A^+}$ is an isomorphism if we restrict ourselves to smooth representations of $\G$ with a central character. 

To prove (3), first we verify that as vector spaces we have
\begin{equation}\label{exti}\Ext^i_{\mathcal H}( \Triv_\H, \sign_\H)\cong
\Ext^i_{\mathcal H}( \sign_\H, \Triv_\H)
\end{equation}  for any $i\geq 0$ and then we show that $\dim _R \Ext^1_{\mathcal H}( \sign_\H, \Triv_\H)=1$.

 Consider the involutive automorphism  $\eta$ of the algebra $\Hh$ (\cite[Propositions 4.13 and  4.23]{Vig1}, \cite[(4)]{Vig3}, introduced in the split case  in  \cite[Cor. 2]{Vigprop} (see  also \cite[\S 4.7]{OSch} and note that in the last two references where $\tau_w^* $ is defined differently from the current article) satisfying 
\begin{equation}\eta(\tau_w)=(-1)^{\ell(w)}\tau_w^* \quad (w\in \W(1)). \end{equation}For $M$ a right $\Hh$-module, define $M\eta$ to be the right $\Hh$-module 
on the vector space $M$ with the action of $\Hh$ twisted by $\eta$, i.e. $(m, h)\mapsto m \eta(h)$. 
For example, using \eqref{simplify}, we obtain that the sign character $\sign_\H$ and   the trivial character $\Triv_{\H}$
(defined in Remark \ref{rema:idempo}) satisfy    $$\Triv_{\H}= \sign_\H \eta\quad\textrm{and}\quad \sign_{\H}= \Triv_\H \eta.$$ When $\ell(w)=0$ we have indeed $\tau_w^*=\tau_w$ and  when $\ell(w)>0$ we have
$ \sign_\H (\tau_w^*)=0 $
 because $\tau_{n_s}^*=\tau_{n_s}+\vert  \Zf'_s \vert^{-1} \sum_{z\in \Zf'_s} \tau_z$ and therefore
 $\sign_\H(\tau_{n_s}^*)=-1+1=0$ for any $s\in S_{aff}$
   (in the case when $\G$ is split, compare with \cite[\S 5.4.2]{Compa}).

The map $\Hh\rightarrow \Hh \eta$, $h\mapsto \eta(h)$ being an isomorphism of right $\Hh$-modules,  $M$ is  projective  if and only if $M\eta$ is  projective. Lastly, given $M$ and $N$ two right $\Hh$-modules, the identity map yields an isomorphism of vector spaces 
$\Hom_\Hh(M,N)\cong \Hom_\Hh(M\eta, N\eta)$. Therefore, if $P_\bullet=(P_i)_{i\geq 0}$ is a projective resolution of $ \sign_\H$, then $P\eta_\bullet:=(P_i\eta)_{i\geq 0}$ is naturally   a projective resolution of  $\Triv_\H$. We apply the functor $\Hom_{\Hh}(-,  \sign_\H)=\Hom_{\Hh}(-,  \Triv_\H \eta)$ to the complex $P\eta_\bullet \rightarrow \Triv_\H$ and obtain that  $\Ext^i_{\mathcal H}( \Triv_\H, \sign_\H)$
is the $i^{th}$ cohomology space of the complex 
$\Hom_{\Hh}(P\eta_\bullet,  \Triv_\H \eta)$ which is isomorphic to the 
 complex 
$\Hom_{\Hh}(P_\bullet,  \Triv_\H)$. This proves  \eqref{exti}.\\

We  now prove  $\dim _R \Ext^1_{\mathcal H}( \sign_\H, \Triv_\H)=1$.  In $\Mod(\H)$, the exact sequence \begin{equation}\label{ses}0\to \Triv_\H \to {\Triv_{\H_\T}\otimes_{\HTpt}\H} =\Ind_{\HT}^\H(\Triv_{\H_\T})\to \sign_\H \to 0\end{equation}does not split: by  Prop. \ref{prop:Q1prop}, we know that  $ \Ind_{\HT}^\H(\Triv_{\H_\T})$
is isomorphic to the $\I$-invariant subspace of the parabolic induction $\Ind^\G_\B(1)$. This $\I$-invariant subspace is described explicitly in \cite[Thm. 4.2]{VigGL2} and one easily checks that it sits indeed in the nonsplit exact sequence \eqref{ses}. 
Applying $\Hom_\H(\sign_\H, -)$,  it induces 
 an exact sequence  
 $$\{0\}\to   \Hom_\H(\sign_\H,\sign_\H )\to   \Ext^1_{\mathcal H}(\sign_\H, \Triv_\H) \to  \Ext^1_{\mathcal H}(\sign_\H, \Ind_\HTp^\H(\Triv_\HT)).$$
The exactness of the functor $\Ind_\HM^\H$ and of its left adjoint $\L^\H_\HT:\Mod(\H)\rightarrow\Mod(\H_\T)$   (\S\ref{indcoind-prop}, property (1)) imply that we have isomorphic functors (\cite[p. 163]{HS}) :
 $$\Ext^1_{\mathcal H}(-,\Ind_\HT^\H(-))\cong \Ext^1_{\mathcal H_\T}({\L^\H_\HT}(-), - ) .$$
 The claim in (3) then easily follows from ${\L^\H_\HT} (\sign_\H)=\{0\}$ which we prove here. In our context, the Hecke algebra $\H_\T$ identifies with  the group algebra $R[\T/\T^1]$. 

Using  \eqref{explicitleftadj1}, we have 
$${\L^\H_\HT}= (-\otimes_{\H_{\T^-, \theta^*}}\H_\T) \upiota^{-1}\upiota_\T.$$
We claim that $\sign_\H\otimes_{\H_{\T^-, \theta^*}}\H_\T=\{0\}$.  \\
Let $t\in \T$ be an element satisfying $-\alpha (\tilde \nu (t))<0$ for all $\alpha \in \Sigma^+$. We denote   by $\mu_{-}$ its  image in $ \Lambda_{ \T^-}(1)=\W_{\T^-}(1)$. Seen in $\W_\T(1)$, the element $\mu_{-}$   has length zero and $\tau^{*,\T}_{\mu^-}=\tau^{\T}_{\mu^-}$. Seen in 
$\W(1)$, it has non zero length and  therefore 
$$\sign_\H(\theta^*(\tau_{\mu_{-}}^\T))=\sign_\H(\theta^*(\tau_{\mu_{-}}^{*,\T}))=\sign_\H(\tau^*_{\mu_{-}})=\triv_\H(\eta(\tau^*_{\mu_{-}}))=(-1)^{\ell(\mu_{-})}\triv_\H(\tau_{\mu_{-}})= 0.$$
In  
 $\sign_\H\otimes_{\H_{\T^-, \theta^*}}\H_\T$ we have 
 $1\otimes 1=1\otimes  \tau^\T_{\mu_{-}}(\tau^\T_{\mu_{-}})^{-1}=\sign_\H(\theta^*(\tau_{\mu_{-}}^\T))\otimes (\tau^\T_{\mu_{-}})^{-1}=0$. This proves the claim and concludes the proof of $\dim _R \Ext^1_{\mathcal H}( \sign_\H, \Triv_\H)= 1$.

When $R$ is an algebraic closure of the residue field of $F$, we deduce from (1), (2) and (3) the existence of   an extension $E_\tau$  of $\Triv_\G$ by $\St_\G$ which is not generated by its $\I$-invariants. 

Now, let $R$ be an arbitrary field of characteristic $p$. We are not aware of a version of (1) over a non algebraically closed field, but we do not need it. We note that 
(3) and its proof are valid  for $R$, that the representation $E_\tau$  constructed by Colmez
is defined over the finite field generated by the values of $\tau$,  and that we  can choose  $\tau\in \Hom (\mathbb Q_p^*, \mathbb F_p)$. Therefore we can  choose the extension $E_\tau$  to be  defined on  the prime  field $\mathbb F_p$. By tensor product with $R$, we obtain an extension (defined over  $R$) of $\Triv_\G$ by $\St_\G$ which is not generated by its $\I$-invariants. \end{proof}

\section{\label{main:ss}Supersingularity for Hecke modules and representations}  

Unless otherwise mentioned,  $R$ will be  an algebraically closed field of characteristic $p$. 
The goal of this section is to prove Theorem \ref{theo:ss}.
\subsection{Supersingularity for Hecke modules and representations: definitions}\label{def:ss}
 
  The classification of all smooth irreducible admissible $R$-representations of $\G$ up to the supercuspidal representations is
described in \cite{AHHV}. This generalizes    \cite{Abeclassi} 
and \cite{Her} which treat respectively the case when  $\G$ split and the case when $\G={\rm GL}(n,\F)$. A  classification of all simple right $\H$-modules up to the supersingular  modules is given in \cite{Abe}. The classification of  all supersingular  right $\H$-modules  is described in \cite{Vig3}. The latter generalizes \cite{Compa} which treats the case when $\G$  is split.
 We recall the following definitions:

\begin{enumerate} 
\item An $R$-representation   of $\G$
  is \emph{supercuspidal} if  it  is  non zero admissible   irreducible,  and if furthermore it is not a subquotient of $\Ind_\P^\G \tau$ for some standard parabolic subgroup $\P$ of $\G$ with standard Levi decomposition $\P=\M\N$ where $\M\neq \G$  and $\tau \in \Rep(\M)$  is an admissible irreducible representation \cite[I.3]{AHHV}.  
\item  A smooth representation  $\pi$ of $\G$  is \emph{supersingular} if it is  non zero  admissible  irreducible and 
if, furthermore,  there is an irreducible smooth representation $\rho$ of $\K$ such that the space $$\Hom_\G(\ind_\K^\G\rho, \pi)=\Hom_\K(\rho, \pi\vert_\K)\neq\{0\}$$  is non zero and  contains  an eigenvector for the left action of the center of  $\mathcal H(\G,\K,\rho)=\End_\G(\ind_\K^\G(\rho))$ with a supersingular eigenvalue $\chi$. This means that given a  standard parabolic subgroup $\P$  with Levi decomposition $\P=\M\N$,
the character $\chi$  extends  to the center of  the spherical Hecke algebra $\mathcal H(\M,\K\cap \M,\rho_{\N\cap \K}) $   via the  Satake homomorphism  $\mathcal H(\G,\K,\rho)\hookrightarrow \mathcal H(\M,\K\cap \M,\rho_{\N\cap \K})$ only when $\P=\G$ (\cite[I.5]{AHHV}).
A smooth representation   of $\G$ which  is non zero  admissible  irreducible and not supersingular will be sometimes called \emph{nonsupersingular}.

\begin{rema}
 For $\pi\in \Rep(\G)$ irreducible admissible,  $\pi$ is supercuspidal if and only if $\pi$ is supersingular. This theorem is proved in \cite[Theorem 5]{AHHV} which generalizes   \cite{Her}  and \cite{Abeclassi}. 
\end{rema}

\item In  \cite[\S 1.2 and \S 1.3]{Vig3}, a  central  subalgebra $\mathcal Z_{\Tp}$ of the pro-$p$ Iwahori Hecke algebra $\H$ is  defined  (it was first introduced in the  split  case in \cite[Definition 5.10]{Compa} where it is denoted by $\mathcal Z^\circ(\H)$). 
There is an isomorphism  $\mathcal Z_{\Tp}\cong R[\X^+_*(\T)]$ (\cite[Prop. 2.10]{Compa},  \cite[Thm 1.4]{Vig3}). 
Denote by $\X^0_*(\T) \subset \X^+_*(\T)$ the maximal  subgroup of the monoid $\X^+_*(\T)$. Then  $ R[\X_*(\T)^+-  \X^0_*(\T) ]$ identifies with   a proper ideal 
$\mathcal Z_{\T,\ell >0} $ of $\mathcal Z_{\T}$.
 As a vector space, $\mathcal Z_{\T}$ is  isomorphic to the  direct sum $R[\X^0_*(\T) ]  \oplus \mathcal Z_{\T,\ell >0}$.
 
 A morphism of $R$-algebras $\mathcal Z_{\T}\rightarrow R$ is called \emph{supersingular} if its kernel contains $\mathcal Z_{\T,\ell >0} $.

 A non-zero element  $v$ of a right $\H$-module  $\ssmod$   is called \emph{supersingular}   if there is an integer $n\geq 1$ such that   $v\mathcal Z_{\T,\ell >0}^n=\{0\}$. The module $\ssmod$   is called \emph{supersingular} when all its  elements are supersingular (this is different from the definition \cite[Def 6.10]{Vig3} where we suppose $\ssmod \mathcal Z_{\T,\ell >0}^n=\{0\}$ for some integer $n\geq 1$).

 Note that a supersingular module as defined above  does not even necessarily have finite length. 
 
 A simple  right $\H$-module which  is not supersingular will be sometimes called \emph{nonsupersingular}.

 \begin{rema} \label{rema:supers}
 Recall that $R$ is an algebraically closed field of characteristic $p$. We recall the following facts which, in the case when $\G$ is split, are contained in \cite[\S5.3]{Compa}.
\begin{enumerate}

\item A  simple right $\mathcal H$-module is finite dimensional. This follows immediately from \cite[5.3]{Vigdurham} since $\mathcal H$ is finitely generated over its center $\mathcal Z(\mathcal H)$ and since  $\mathcal Z(\mathcal H)$ is an $R$-algebra of finite type \cite[Thm 1.2]{Vig2}.  These properties remain true when the center is replaced by the central subalgebra $ \mathcal Z_\T$ (\cite[Prop. 2.5]{Compa}, \cite[Thm 1.4]{Vig3}).

\item  The central subalgebra $ \mathcal Z_\T \subset \H$ acts on a simple (hence finite dimensional) $\H$-module $\m $ by a character. 
\item A  non zero  finite dimensional  $\H$-module $\m $ is supersingular if and only if all the subquotients of  $\m$ seen as a  $ \mathcal Z_\T$-module are supersingular characters of $ \mathcal Z_\T$.

\item A submodule and a quotient of a supersingular $\H$-module are supersingular.
\item A  non zero finite dimensional $\H$-module  is not supersingular if and only if it contains a nonsupersingular simple submodule.

\end{enumerate}
\end{rema}
\begin{quote}
Only the direct implication in Point (e) requires a justification. 
We first recall the following general facts: let $R$ be a field,  $A$ an $R$-algebra, and  $C$  a central subalgebra of  $A$, such that  $C$ is a finitely generated $R$-algebra  and   $A$ is a finitely generated $C$-module. The functors $\Ext_A^r (-,-)$ are considered in the abelian category $\Mod_A^{fd} $ of right $A$-modules which are finite dimensonal over $R$.  
Given an ideal $\mathcal J$ of $C$, a finite dimensional right $A$-module $\mathfrak m $ which is  killed by some positive power of $\mathcal J $ is called $\mathcal J$-torsion.   We have:  \\
 i.   The category $\Mod_A^{fd}$ decomposes into the direct sum, over the maximal ideals $\mathcal M$ of $C$, of the subcategory of  $\mathcal M$-torsion modules.\\
 ii. A finite dimensional $A$-module is $\mathcal M$-torsion if and only if, seen as a $C$-module, it admits a composition series where  all quotients  are isomorphic to $C/\mathcal M$.\\
 iii.  Given an ideal $\mathcal J$ of $C$, the  category  of finite dimensional  $\mathcal J$-torsion $A$-modules decomposes into the direct sum, over the maximal ideals $\mathcal M$ of $C$ containing $\mathcal J$, 
of the subcategory of  $\mathcal M$-torsion modules. \\
iv.  Let $\mathcal M$ and $\mathcal M'$ be two maximal ideals of $C$.
 If there exists a $\mathcal M$-torsion module $\mathfrak m\in \Mod_A^{fd} $ and  a $\mathcal M'$-torsion   module ${\mathfrak m}' \in \Mod_A^{fd} $ such that  $\Ext_A^r (\mathfrak m,{\mathfrak m}')\neq \{0\}$ for some integer $r\geq 0$, 
then $\mathcal M=\mathcal M'$.

  When $C=A$, the first three properties are \cite[Theorem 2.13 b, c]{Eisenbud}. 
In general, for  $\m\in \Mod_A^{fd}$,  the set of  $x\in \m$  which are killed by some positive power of $ \mathcal M$ is the component of 
$\m$  in the subcategory  of $\mathcal M$-torsion finite dimensional $C$-modules.
To prove iv.,  note that the $C$-module $\Ext_A^r (\m,\m')$  is killed by a positive power of   $\mathcal M$  and by a positive power of $\mathcal M'$.  The localization of $\Ext_A^r (\m,\m')$  at any  maximal ideal   of $C$ different from $\mathcal M$ or from $\mathcal M'$ is $0$. If  $\mathcal M\neq \mathcal M'$ then the localization of $\Ext_A^r (\m,\m')$  at any  maximal ideal of $C$ is $0$. This is equivalent to $\Ext_A^r (\m,\m')=\{0\}$ \cite[Lemma 2.8]{Eisenbud}.\\

These general statements apply  when $R$ is a field of characteristic $p$, $A=\H$,  $C = \mathcal Z_\T$ and $\mathcal J = \mathcal Z_{\T, \ell >0}$.
 When the field $R$  is furthermore algebraically closed,   we have $ \mathcal Z_\T/\mathcal M \simeq R$ for any maximal ideal $\mathcal M$ of $ \mathcal Z_\T$ and  a character $\chi : \mathcal Z_\T\to R$ is  supersingular if its kernel contains $\mathcal Z_{\T, \ell >0}$.
Statements i.-iv. in this context  have already been observed in \cite[\S 5.3]{Compa} (where $\G$ was supposed split).
The direct implication in Point (e) follows from them.

\end{quote}

 \end{enumerate}

\subsection{Main theorem}

\begin{theorem}\label{theo:ss}
Suppose that $R$ is an algebraically closed field of characteristic $p$. Let $\pi\in \Rep(\G)$  be an irreducible admissible representation. Then the following are equivalent:
\begin{enumerate}
\item
 $\pi$ is supersingular;
 
 \item   the finite dimensional $\H$-module $\pi^{\I}$ is supersingular.
 \item   the finite dimensional $\H$-module $\pi^\I$ admits a supersingular subquotient.
 \end{enumerate}
 \end{theorem}
  
   When $\bf G=\bf Z $,   the theorem is trivial because all finite dimensional $\H$-modules are supersingular and all irreducible admissible representations of $\G$ are supersingular.  
  
  \medskip

It   is obvious that (2) implies (3).  In this article, we prove that   (3) $\Rightarrow $ (1) $\Rightarrow$ (2),   that is to say,
\begin{equation}\label{implications} 
\text{\it $\pi^\I$ admits a supersingular subquotient   $\Rightarrow \pi$ supersingular  $\Rightarrow \pi^\I$ supersingular.} 
\end{equation}
 
  
The proof of  \eqref{implications}  is given in \S \ref{prooftheoss}. It requires the results of  Remark \ref{21} and \S \ref{Standard triples}.
Note that several of these preliminary results hold not only when $R$  is an algebraically closed field of characteristic zero but also when it is  an arbitrary commutative ring  $R$.

\begin{rema}\label{21} That  (2) implies  (1) comes from the following arguments.
When $\pi$   is not supersingular, 
there is an irreducible representation $\rho$ of $\K$ such that $$\Hom_\G(\ind_\K^\G\rho, \pi)=\Hom_\K(\rho, \pi\vert_\K)\neq \{0\}$$  and this space  contains an eigenvector  with a nonsupersingular eigenvalue for the left action of  the center of $\mathcal H(\G,\rho)$ by definition of the supersingularity.   Then
 $\pi^\I$ contains a 
nonsupersingular submodule \cite[Prop 7.10, Rmk 7.11]{Vig3} therefore the  $\H$-module $\pi^\I$ is not  supersingular.  
When $\bf G $ is $ \F$-split,  $ \bf G_{der}=  \bf G_{sc}$ and the characteristic of $ \F$ is $0$, this was proved in \cite{Compa} (see Equation (58) and Lemma 5.25 therein).
\end{rema}

\subsection{\label{Standard triples} Standard triples }
Let $R$ be a commutative ring.

 \subsubsection{\label{class-reps}  Standard triples of $\G$}  
 A standard triple of $\G$ is a triple $(\P,\sigma,\Q)$  where
\begin{itemize}
\item[-] $\P$ is a standard parabolic subgroup of $\G$ with standard Levi decomposition $\P=\M\N$,
\item[-] $\sigma$ is a  $R$-representation of $\M$,
 
 \item[-]  $\Q$ is a  parabolic subgroup  such that $\P\subset \Q\subset \P(\sigma )$ where $\P(\sigma) $  is the  parabolic subgroup of $\G$ attached to $\Pi(\sigma)=\Pi_\M \sqcup \Pi_\sigma$ where $ \Pi_\sigma$ is the set of $ \alpha \in \Pi$ such that 
 $\Zp\cap \M'_\alpha $ acts trivially on $\sigma$. 
 \end{itemize}
Standard triples $(\P,\sigma,\Q)$  and $(\P,\sigma',\Q)$  of $\G$ where $\sigma\simeq \sigma'$ are called equivalent. 
A standard triple $(\P,\sigma,\Q)$ of $\G$ is called smooth (resp. supersingular, resp. supercuspidal) if  $ \sigma $  is smooth (resp. supersingular, resp. supercuspidal). \\\\
Given a parabolic  subgroup $\Q$ as above, we will often denote its  standard Levi decomposition by $\Q= \M_\Q\N_\Q$.

  \begin{rema} \label{5.13}\begin{enumerate}
  \item   For a standard triple $(\P,\sigma,\Q)$ of $\G$,  the representation $\sigma$ of $\M$ extends to a  unique representation  $e_{\Q}(\sigma)$ of $\Q$ which is trivial on  $\N$    \cite[II.7 (i)]{AHHV}.	\item   Let   $\M_\sigma\subset \G$ denote the standard Levi subgroup of $\G$ such that  $\Pi_{\M_\sigma}=\Pi _\sigma$.   The  subgroup $ \Zp\cap \M'_\sigma $ is generated by 
 all  $\Zp\cap \M'_\alpha$  for $\alpha \in \Pi _\sigma$  \cite[II.3, II.4]{AHHV}; its  image 
 ${}_1\Lambda_{\M'_\sigma}$ in the pro-$p$ Iwahori Weyl group $ \W(1)$  of $\G$ is 
 generated by the images ${}_1\Lambda_{\M'_\alpha}$ of $ \Zp\cap \M'_\alpha $  in $\W(1)$. If $\sigma$ is supercuspidal, the roots $\alpha \in \Pi _\sigma $ are orthogonal to  $\Pi_{\M}$. If  the roots $\alpha \in \Pi _\sigma $ are orthogonal to  $\Pi_{\M}$,  the group ${}_1\Lambda_{\M'_\sigma}$  is contained in the group of elements of length $0$ of the pro-$p$ Iwahori Weyl group $\W_{\M}(1)$ of $\M$.  We have   $\sigma^{\I _{\M }}(\tau^{\M ,*}_w)=\Id$  for $w \in {}_1\Lambda_{\M'_\sigma} $, as $v\tau^{\M, *}_w= v\tau^{\M}_w  ={\hat w^{-1}}v=v$ for $v\in \sigma^{\I _\M}, {\hat w}\in \Zp\cap \M'_\sigma$.\end{enumerate}
  \end{rema}
Let   $(\P,\sigma,\Q)$ be a standard smooth triple of $\G$.  For $\Q\subset  \Q'\subset \P(\sigma)$, the $R$-representation  $ \Ind^\G _{ \Q'}(e_{\Q'}(\sigma))$ of $G$  is smooth and embeds naturally in the smooth $R$-representation $\Ind^\G_{\Q} (e_\Q(\sigma))$  because the  representation $e_{\Q'}(\sigma)$ of $\Q'$ extending $\sigma$  trivially on  $N$ is equal to  $e_{\Q}(\sigma)$ on $\Q$. We define the following smooth $R$-representation of $\G$: 
$$I_G(\P,\sigma,\Q):=\Ind^\G_{\Q} (e_\Q(\sigma))/\sum_{\Q\subsetneq \Q'\subset \P(\sigma)} \Ind^\G _{ \Q'}(e_{\Q'}(\sigma)).$$

 \subsubsection{\label{class-repsC} {Classification of  the irreducible admissible representations of $\G$}}Assume  that $R$ is an algebraically closed field of characteristic $p$. 
 The main result of \cite[Thm.s. 1, 2, 3 and 5]{AHHV} says that the map $(\P,\sigma,\Q)\to I_G(\P,\sigma,\Q)$ is a bijection from the standard supercuspidal triples of $\G$ up to equivalence onto the irreducible admissible smooth $R$-representations  of $\G$ modulo isomorphism.

\begin{rema}\label{rema:finite-supersing} Let $(\P,\sigma,\Q)$ be a standard supercuspidal triple of $\G$.
 The representation $I_\G(\P,\sigma,\Q)$ is supersingular if and only if $\P=\G$  \cite[Thm. 5)]{AHHV}. It is finite dimensional if and only if  $\P=\B $ and $\Q=\G $ \cite[I.5 Rmk. 2)]{AHHV} in which case we have $I_\G(\B,\sigma,\G)=e_{\G}(\sigma)$. Any irreducible smooth  representation  of $\Zp$ is supercuspidal and supersingular.
 \end{rema}

 \begin{rema} \label{5.12} The pro-$p$ Iwahori group $\I$ acts trivially on the irreducible smooth finite dimensional $R$-representations   of $\G$. Such a representation   is indeed of the  form $ e_{\G}(\sigma)$ for an irreducible smooth representation $\sigma$ of $\Zp$. Then use the 
 Iwahori decomposition \eqref{f:iwadecomp}   which   $\I$ admits with respect to $\B$, and  $\U$, $\U^{op}$, and the fact that $\Zp^1$ acts trivially on  $\sigma$ hence on 
 $ e_{\G}(\sigma)$.  \end{rema}
 
 \subsubsection{\label{class-modH} Standard triples of $\mathcal H$} 
  
 A standard triple of the pro-$p$ Iwahori Hecke $R$-algebra $\mathcal H$ is a triple $(\P,\sigma, \Q)$ where

\begin{itemize}
\item[-] $\P$ is a standard parabolic subgroup of $\G$ with standard Levi decomposition $\P=\M\N$,
\item[-] $\sigma$ is   right $\H_\M$-module,
 
 \item[-]  $\Q$ is a  parabolic subgroup with  $\P\subset \Q\subset \P(\sigma )$ where
  $\P(\sigma)$   is the    standard parabolic subgroup corresponding to $\Pi(\sigma)=\Pi_\M \sqcup \Pi_\sigma  \subset \Pi$ where $\Pi _\sigma $ is the set of  roots $\alpha \in \Pi$ which are orthogonal to  $\Pi_{\M}$  and such that $\sigma( \tau^{M,*}_{w})=\Id$  if  $w\in  {}_1 \Lambda_{\M'_\alpha}$ (as defined in Remark \ref{5.13}).
\end{itemize}  
Standard triples  $(\P,\sigma,\Q)$ and $(\P,\sigma',\Q)$ of $\mathcal H$ with  $\sigma \simeq \sigma'$ are called equivalent. 
 When $\sigma$ is finite dimensional (resp. simple, resp. supersingular), the standard triple  $(\P,\sigma, \Q)$ of $\mathcal H$ is  called  finite dimensional  (resp. simple, resp. supersingular).   

\medskip 

By Remark \ref{5.13}, if   $(\P,\sigma, \Q)$ is a superscuspidal standard triple   of $\G$  then    $(\P,\sigma^{\I_\M}, \Q)$ is a standard triple of $\H$.

\medskip

Let  $(\P,\sigma, \Q)$  a standard triple   of $\H$.   There exists a unique 
 $\mathcal H_{\M_\Q}$-module  structure $e_{\mathcal H_{\M_\Q} }(\sigma)$ on $\sigma$ such that  $e_{\mathcal H_{\M_\Q} }(\sigma)( \tau^{\M_\Q,*}_{w})=\sigma(\tau^{\M,*}_{w})$ for $w\in  \W_\M(1)$ and  $e_{\mathcal H_{\M_\Q} }(\sigma)( \tau^{\M_\Q,*}_{w}) =\Id_{{\sigma}}$  for $w\in   \W_{\M'_{\Q,\sigma}}(1)$ where $\M_{\Q, \sigma}\subset \G$ is the standard Levi subgroup corresponding to  $\Pi_{\M_\Q} \cap \Pi_{\sigma} $. Let $\P\subset \Q\subset \Q_1\subset \P(\sigma)$; we have $\W_{\M_{\Q_1}}(1)\subset \W_{\M_{\Q_1}}(1)$, and ${}^{\M_{\Q }}\mathbb W_{\M_{\Q_1 }}$ the set of $w\in W_{\M_{\Q_1}}$ with minimal length in the coset $ W_{\M_{ \Q}}w$;
 it is proved in  \cite[formula (60)]{AHV} that  the map
$$x\otimes 1  \mapsto x\otimes (\sum_{d \in  {}^{\M_{\Q }}\mathbb W_{\M_{\Q_1 }}} T _{\tilde d})
  : e_{\mathcal H_{\M_{\Q_1}} }(\sigma) \otimes _{\mathcal H_{\M_{\Q_1}^+},\theta}\mathcal H \to  e_{\mathcal H_{\M_\Q} }(\sigma) \otimes _{\mathcal H_{\M_\Q^+},\theta}\mathcal H$$  
is a well defined injective  $\mathcal H$-equivariant morphism from $\Ind_{\mathcal H_{\M_{\Q_1}}}^{\mathcal H}(e_{\mathcal H_{\M_{\Q_1}} }(\sigma))$ into 
$\Ind_{\mathcal H_{\M_{\Q}}}^{\mathcal H}(e_{\mathcal H_{\M_{\Q}} }(\sigma))$;
 one  may then  define the right $\mathcal H$-module
 \begin{equation}\label{IPsQ}
I_{\mathcal H}(\P,\sigma,\Q):=  \Ind_{\mathcal H_{\M_\Q}}^{\mathcal H}(e_{\mathcal H_{\M_\Q} }(\sigma))/ \sum_{\Q\varsubsetneq \Q_1\subset \P(\sigma)} \Ind_{\mathcal H_{\M_{\Q_1}}}^{\mathcal H}(e_{\mathcal H_{\M_{\Q_1}} }(\sigma)).
\end{equation}

 The smooth parabolic induction $\Ind_{\P }^\G$   corresponds to the induction $ \Ind_{\mathcal H_{\M }}^{\mathcal H}$  via  the pro-$p$-Iwahori invariant functor  and its left adjoint. This is proved in the current paper in Prop \ref{prop:Q1prop} and Corollary \ref{coro:Q2} (see the commutative diagrams in Questions \ref{qp1} and \ref{qp2}). It is proved in   \cite[Thm. 1.3, Cor. 10.13]{AHV} that 
this implies: 

\begin{theorem} \label{key}(i) Let $(\P,\sigma,\Q)$ be a supercuspidal standard  triple of $\G$. Then  $(\P,\sigma^{\IM},\Q)$ is a standard $\H$-triple and  the $\mathcal H$-modules
$$  I_{G}(\P,\sigma,\Q)^{\I}\cong  I_{\mathcal H}(\P,\sigma^{\IM},\Q)$$
 are isomorphic. 

 (ii) Let $(\P,\sigma,\Q)$ be a standard  triple of $\mathcal H$. Then  $(\P, \sigma \otimes_{\H _\M}\XX_\M,\Q)$ is a standard smooth triple of $\G$ and the  smooth $R$-representations  of $\G$ $$ I_{\mathcal H}(\P,\sigma ,\Q)\otimes_\H \XX \cong  
I_G(\P, \sigma \otimes_{\H _\M}\XX_\M,\Q)$$
are isomorphic. 
In particular 
 $e_{\mathcal H}(\sigma)\otimes _{\mathcal H}\XX   \cong e_\G( \sigma \otimes _{\mathcal H_\M}\XX_\M ) $ when $\Q=\P(\sigma)=\G$.
\end{theorem}

Note that in the case when $\P=\B$, $\M=\Zp$ and $\sigma$ is the trivial character of $\H_\Zp$,  the isomorphism  $e_{\mathcal H}(\sigma)\otimes _{\mathcal H}\XX   \simeq e_\G( \sigma \otimes _{\mathcal H_\Z}\XX_\Zp) $ is the isomorphism proved in Lemma \ref{lemmatriv}. 

 \subsubsection{ \label{class-mod} Classification of the simple $\H$-modules} 
 Assume that $R$ is an algebraically closed field of characteristic $p$.

\begin{theorem}\label{thm:class-mod}The map   $(\P,\sigma,\Q)\to I_{\mathcal H}(\P,\sigma,\Q)$  
  is a bijection from the set of standard simple supersingular triples  up to equivalence of  $\mathcal H$ onto the set of  simple $\mathcal H$-modules modulo isomorphism.
  \end{theorem}  

This  variant  of the classification of Abe \cite{Abe} (Abe's classification uses coinduction instead of induction) is proved in \cite[Cor. 9.30]{AHV}; we prefer the induction  because it is  compatible with the parabolic induction functor and with the $\I$-invariant functor (Theorem \ref{key}).

\begin{prop}\label{prop:ind} Let $\P=\M\N$ be a standard parabolic subgroup of $\G$ with its standard Levi decomposition and $\sigma$ a simple supersingular right  $\H_\M$-module. Then the composition factors of the induced module $\Ind_{\mathcal H_{\M }}^{\mathcal H}(\sigma)$ are  $I_\H(\P,\sigma,\Q)$ for $\P \subset \Q \subset \P(\sigma)$.
\end{prop}
This is the  variant for an \emph{induced}  $\H$-module of a result of Abe \cite[Cor. 4.26]{Abe}  which considers the  right $\H$-module \emph{coinduced} from a simple supersingular right  $\H_\M$-module.  

\bigskip  Theorem \ref{thm:class-mod} and Proposition \ref{prop:ind} follows from the comparison between induction and coinduction.

\begin{rema}\label{rema:supersing} (i) The  $\H$-module $I_\H(\P,\sigma,\Q)$ is simple and  supersingular if and only if $P=Q=G$ and $\sigma$ is simple and supersingular (this is a part of the classification).  

(ii) Let  $\tau$ be a simple right $\H_\M$-module and  $(\P^\M,\sigma ,\Q^\M)$  a standard simple supersingular triple of $\H_\M$ such that 
 $\tau=I_{\HM}(\P^\M,\sigma,\Q^\M)$.  Let $\R\subset \P$ be the standard parabolic subgroup of $\G$ such that the standard Levi subgroup $\M_\R$ of $\R$ is equal to the standard Levi subgroup  of $\P^\M$. 
Then, the  composition factors of the  $\H$-module  $I_{\H}(\P,\tau,\Q)$   are  composition factors  of  $\Ind_{\mathcal H_{\M_\R}}^{\mathcal H}(\sigma)$ (apply Proposition \ref{prop:ind}  and the   transitivity of the parabolic induction 
\cite[Thm. 1.4]{Vig5}).  Hence they are of the form $I_\H(\R,\sigma, {\rm S})$ for some parabolic subgroups ${\rm S}$ such that $\R \subset {\rm S} \subset \R(\sigma)$ which can be explicitely determined, if necessary.  From (i), none of the simple subquotients of the $\H$-module  $I_{\H}(\P,\tau,\Q)$ is  supersingular if $\R\neq \G$. 
Otherwise $\P=\Q=\R=\G$ and $I_{\H}(\G,\tau,\G)=\tau$ is supersingular.

(iii)  When  $\tau$ is a finite dimensional right $\H_\M$-module and $\P\neq \G$, 
  the $\H$-module  $I_{\H}(\P,\tau,\Q)$ admits no  supersingular subquotient (Apply (ii) to the composition factors of the  $\H_\M$-module $\tau$).  
  
  (iv) A simple right $\H_\M$-module is supersingular if and only if it is not a subquotient of $\Ind_{\mathcal H_{\M }}^{\mathcal H}(\tau)$ for some standard Levi subgroup $\M\neq \G$  and some  simple $\mathcal H_{\M }$-module $\tau$. Compare with the definition of a supercuspidal representation of $\G$ in \S\ref{def:ss} (1).  \end{rema}

\subsection{\label{prooftheoss}Proof of (\ref{implications})}

    We suppose  that $\G\neq \Zp$. 
 Let  $\pi\in\Rep(\G)$ be a non zero irreducible admissible  representation.  Then $\pi^\I$ is a  non-zero finite dimensional  right $\H$-module. \\\\
 \textbf{A)}  We prove  the second implication of \eqref{implications} \emph{i.e.} that if $\pi$ is supersingular, then $\pi^\I$ is supersingular.    \\
   Suppose that the $\H$-module $\pi^\I$ is not supersingular.  
   By Remark \ref{rema:supers}(e), the $\H$-module  $\pi^\I$ contains  a simple nonsupersingular submodule.   
 By    \S\ref{class-mod} and Remark \ref{rema:supersing}, there is a standard simple supersingular triple $(\P, \ssmod, \Q)$ of $\H$ with $\P=\M\N\neq \G$ such that $I_\H(\P, \ssmod, \Q)$ is contained in $\pi^\I$.
 We prove that the representation $\pi$ of $\G$ is not supercuspidal.

\begin{itemize}
\item[a/] Suppose $\Q\neq \G$. We have by adjunction and by  the commutativity of the diagram \eqref{diag2'prop}:
\begin{align*}0&\neq \Hom_{\mathcal H}(\Ind^{\H}_{\mathcal H_{\M_\Q}}(e_{\H_{\M_\Q}}(\ssmod)), \pi^\I)\cong \Hom_{kG}(\Ind^{\H}_{\mathcal H_{\M_\Q}}(e_{\H_{\M_\Q}}(\ssmod) )\otimes_{\mathcal H}  \XX, \pi)\cr
&\cong  \Hom_{k\G}({\Ind_{\Q }^\G} (e_{\H_{\M_\Q}}(\ssmod)\otimes _{\mathcal H_{\M_\Q}}  \XX_{\M_\Q}), \pi).\end{align*} Since the simple $\H_{\M_\Q}$-module $ e_{\H_{\M_\Q}}(\ssmod)$ has a central character, the representation $e_{\H_{\M_\Q}}(\ssmod)\otimes _{\mathcal H_{\M_\Q}}  \XX_{\M_\Q}$ of $\M_\Q$ has a central character.   
We deduce that  $\pi$ is not supercuspidal by \cite[Prop. 7.9]{HVcomp}.
 
 \item[b/]  Suppose $ \Q= \G$. We have  $e_{\H}( \ssmod) \subset  \pi^\I$ and  by adjunction,   $\pi$ is a quotient of  $e_{\H}( \ssmod)\otimes _{\H} \XX$. But    $e_{\H}( \ssmod)\otimes _{\H} \XX \cong e_\G(\ssmod\otimes _{\H_\M} \XX_\M)$ (Theorem \ref{key}(ii)).  The group  $\N$ acts trivially on $\pi$ because  $e_\G(\ssmod\otimes _{\H_\M} \XX_\M)$ is the  representation of $\G$ trivial on $\N$ extending the representation $\ssmod\otimes _{\H_\M} \XX_\M$ of $\M$. The space of $\N$-coinvariants of a supercuspidal representation of $\G$  is trivial  (this follows immediately
from \cite[II.7 Prop.]{AHHV} and from the classification  recalled in \S \ref{class-repsC}), hence 
 $\pi$ is not supercuspidal. 
  \end{itemize}\textbf{B)}  We prove the first implication of \eqref{implications} \emph{i.e} that  if $\pi^\I$ admits a supersingular subquotient then $\pi$  is supersingular.

Suppose that $\pi $ is   nonsupersingular.  
Consider a supersingular  standard triple  $(\P,\sigma,\Q)$ of $\G$ such that $\pi\cong I_\G(\P,\sigma,\Q)$. As $\pi$ is not supersingular,  we have $\P\neq \G$ . By Theorem \ref{key}(i) we have
$$\pi^\I\cong I_\G(\P,\sigma,\Q)^{\I} = I_{\H}(\P,\sigma^{\I_\M},\Q).$$ 
By Remark  \ref{rema:supersing} (iii) applied to $I_{\H}(\P,\sigma^{\I_\M},\Q)$, $ \pi^\I$ admits no
  supersingular subquotient because $\P\neq \G$.
  
  \begin{rema}  \label{rema:sssc} When $\bf G$ is semisimple and simply connected, we do not need to refer to \cite{AHV} to prove the second implication of \eqref{implications}, that is to say, $\pi$ is supersingular $ \Rightarrow $ $\pi^\I$ is supersingular.  This reference appears only through Theorem \ref{key} in  the proof. When $\bf G$ is semisimple and simply connected, it suffices to refer    to Lemma \ref{lemmatriv} as we explain here:
$\bf G$ is then the direct product of its almost simple components and we can reduce  to an isotropic component. When $\bf G$ is almost simple, simply connected and isotropic, we have,   in  case \textbf{A)} b/  of the proof:  $\P\neq \Q=\G$ implies   $\P=\B, \M=\Zp$ and  $\sigma$  is the trivial module of  $\H_{\Zp}$. This is because $\Pi = \Pi_\M \sqcup \Pi_\sigma$ (\S \ref{class-modH}) and $\Pi$ irreducible implies $ \Pi_\M=\emptyset$. 
As $\H_{\Zp}=R[\Zp/\Zp^1]$,  the module $\sigma$ identifies  with an irreducible  representation $\sigma_{\Zp}$ of $\Zp$.  But  $\Pi= \Pi (\sigma)$  implies  $\Pi= \Pi(\sigma_{\Zp})$, meaning that  $\sigma_{\Zp}$ extends to a representation of $\G$ trivial on $\G'$. But $\G=\G'$ because $\bf G$ is almost simple, simply connected and isotropic \cite[II.4 Prop.]{AHHV}.  Hence $\sigma_\Zp$ is the trivial representation of $\Zp$.
 \end{rema}

  \begin{rema}   Let ${\bf G}^{sc} \to {\bf G}^{der}$ be the simply connected cover of the derived group  ${\bf G}^{der}$ of $\bf G$,  and let $\bf C$ be the connected center of $\bf G$. We have a natural homomorphism $\iota:\G^{sc}\to \G$. 
We can prove that  $\pi$ is supersingular $ \Rightarrow $ $\pi^\I$ is supersingular   using Remark \ref{rema:sssc} when:
   \begin{equation} \label{rfl}\text{ The restriction  to $\iota(\G^{sc})$ of any  supercuspidal $R$-representation $\pi$ of $\G$ has finite length}. 
   \end{equation}
 Property \eqref{rfl} is   true when the index of the normal subgroup $\iota(\G^{sc})C$ of $\G$ is finite (for instance when the  characteristic of $F$ is $0$) or when  $R$ is the field  of complex numbers (\cite{Sil})  We do not know whether  \eqref{rfl}  is true for any algebraically closed field $R$.
      \  \end{rema}

 \section{\label{table}Parabolic induction  functors for Hecke modules and their adjoints in the $p$-adic and the finite case:  summary and comparison table}
Let $R$ be a commutative ring. We consider $\P$ a standard parabolic subgroup of $\G$ with standard Levi decomposition $\P=\M \, \N$  and $\Pf$ 
the standard parabolic subgroup of $\Gf$ with standard Levi decomposition $\Pf=\Mf \, \Nf$ image of $\P\cap \K$ in $\Gf$. 
We pick lifts $\tw$ and $\twm$ in $\W(1)$ 
 for the longest elements $\longw$ and $\longw_\Mf$ of $\Wf$ and $\Wf_\Mf$ respectively as in \S\ref{indcoind-prop}. Let $\M'=\tw \M \tw^{-1}$. Recall that we defined in \S\ref{indcoind-prop} quasi-inverse equivalences of categories   $- \, \upiota_{\M} ^{-1}\upiota: \Mod(\HM)\rightarrow \Mod(\H_{\M'})$ and   $- \, \upiota ^{-1}\upiota_\M: \Mod(\H_{\M'})\rightarrow \Mod(\HM)$. Up to isomorphism, these functors are not affected by the  choice  of other lifts   for  $\longw$ and $\longw_\Mf$.  
 Therefore, the automorphism of $R$-algebras 
  \begin{equation}
\H_{\M'}\rightarrow \H_{\M}, \: \tau_x\mapsto \tau_{\twm\tw x(\twm\tw )^{-1}}\textrm{ and   its inverse }\H_{\M}\rightarrow \H_{\M'}
 \end{equation}
define   quasi-inverse equivalences of categories  \begin{equation}
 \upiota_{\M', \M}: \Mod(\HM)\rightarrow \Mod(\H_{\M'})\textrm{ and   }  \upiota_{\M, \M'}: \Mod(\H_{\M'})\rightarrow \Mod(\HM)
 \end{equation}
which are respectively isomorphic to   $- \, \upiota_{\M} ^{-1}\upiota$
 and   $- \, \upiota ^{-1}\upiota_\M$. They restrict to quasi-inverse equivalences of categories
  \begin{equation}
 \upiota_{\Mf', \Mf}: \Mod(\HMf)\rightarrow \Mod(\Hf_{\Mf'})\textrm{ and   }  \upiota_{\Mf, \Mf'}: \Mod(\Hf_{\Mf'})\rightarrow \Mod(\HMf).
 \end{equation}
which are respectively isomorphic to   $- \, \iota_{\Mf} ^{-1}\iota$
 and   $- \, \iota ^{-1}\iota_\Mf$ (defined in Proposition \ref{prop:compafunc}).

\medskip

 In \S\ref{LeviHecke}, we recalled the definition of the embeddings $\theta$ and $\theta^*$ of $\HMp$ and $\HMm$ into $\H$. In the table below, given an embedding of algebras 
 $\phi: \HM^\pm\rightarrow \H$, 
 we recall that the index $\H_{\M^\pm, \phi}$ means that we see $\H$ as
an $\H_{\M^\pm}$-module via $\phi$.    In Remark \ref{extension} where $p$ is invertible in $R$, the embeddings $\theta^+$, $\theta^{*+}$  and $\theta^-$ of $\HM$ into $\H$ are defined. Again, in the table below, 
given an embedding of algebras 
 $\phi: \HM\rightarrow \H$, 
 the index $\H_{\M, \phi}$ means that we see $\H$ as
an $\H_{\M}$-module via $\phi$.

\bigskip

When $R$ is an arbitrary ring and when  $p$ is invertible in $R$, 
the table below summarizes and compares the definitions of the parabolic induction functors $\Ind_\HM^\H: \Mod(\HM)\rightarrow  \Mod(\H)$ and 
$\Ind_\HMf^\Hf: \Mod(\HMf)\rightarrow  \Mod(\Hf)$ 
 and  of their respective adjoints. 

\bigskip

The top left quadrant of the table (finite case, arbitrary $R$) is justified by Proposition \ref{prop:compafunc}.
The lower left quadrant of the table (finite case, $p$ invertible $R$) is justified by Remark  \ref{L=res}.
 The lower right quadrant of the table ($p$-adic case, $p$ invertible in $R$)  is justified in Remark \ref{rema:tensor-hom}(2)   based on recent work by Abe \cite{Abe2} (see isomorphism \eqref{explicitleftadj2} for $\L_\HM^\H$ and recall that it is valid for an arbitrary $R$) and in Proposition \ref{prop:rem??}.
The top right quadrant of the table ($p$-adic case, arbitrary $R$) is justified in \S\ref{indcoind-prop}  (see there the references to \cite{Vig5}) as well as by
the isomorphism \eqref{explicitleftadj2}.

\bigskip

\begin{table}
\begin{tabular}{ |l|c|c| }
\hline
&Finite case & $p$-Adic case\\
\hline
&&\\
\multirow{4}{*}{Arbitrary $R$} & $\Ind_\HMf^\Hf=-\otimes_\HMf\Hf\cong \Hom_{\Hf_{\Mf'}}(\Hf, -)\circ \upiota_{\Mf', \Mf}$ & $\Ind_\HM^\H=-\otimes_{\HMpt}\H\cong \Hom_{  \H_{ \M'^-, \theta^*}}(\H, -)\circ \upiota_{\M', \M}$ \\ &&\\
 &  $\R_{\HMf}^\Hf\cong \Res_{\HMf}^\Hf $  & $\R^\H_\HM\cong \Hom_{\Hh}(\HM\otimes_{\HMpt} \H, -)$ \\&&\\
 &$\L_{\HMf}^\Hf\cong \upiota_{\Mf,\Mf' } \circ \Res_{\Hf_{\Mf'}}^\Hf $  & $\L_\HM^\H\cong \Res^\H_{\H_{\M, \theta^{*+}}}\cong  \upiota _{\M,\M'}\circ  (- \otimes_{\H_{\M'^-, \theta^*}} \H_{\M'})  $ \\&& \\\hline
\multirow{4}{*}{$p$ invertible in $R$}
&&\\
& $\Ind_\HMf^\Hf-=-\otimes_\HMf\Hf\cong \Hom_{\Hf_{\Mf}}(\Hf, -)$ & $\Ind_{\HM}^\H=-\otimes_{\HM, \theta^+}\H\cong  \Hom_{\H_{\M, \theta^{*+}}}( \H, -)$  \\ &&\\
 &  $\R_{\HMf}^\Hf\cong \Res_{\HMf}^\Hf $  & $\R_\HM^\H\cong \Res^\H_{\H_{\M, \theta^+}}$ \\&&\\
 &$\L_{\HMf}^\Hf\cong  \Res_{\Hf_{\Mf}}^\Hf $  & $\L_\HM^\H\cong \Res^\H_{\H_{\M, \theta^{*+}}} \cong \Res^\H_{\H_{\M, \theta^{-} \delta_P}}$\\&& \\\hline
\end{tabular}
\end{table}
\bigskip


\end{document}